%% file: cubicsparts.tex
\documentclass[reqno,12pt,reqno]{amsart}

\usepackage{amsfonts, amsthm, amsmath, amssymb, enumerate, verbatim, bm, cite, manfnt, multicol}
\usepackage{hyperref}
\usepackage[utf8]{inputenc}
\usepackage{subfiles}
\hypersetup{colorlinks=false}

\RequirePackage{mathrsfs} \let\mathcal\mathscr

\numberwithin{equation}{section}

\newtheorem{theorem}{Theorem}[section]
\newtheorem{lemma}[theorem]{Lemma}
\newtheorem{proposition}[theorem]{Proposition}
\newtheorem{corollary}[theorem]{Corollary}
\newtheorem{conjecture}[theorem]{Conjecture}

\theoremstyle{definition}
\newtheorem*{ack}{Acknowledgements}
\newtheorem{remark}[theorem]{Remark}
\newtheorem{definition}[theorem]{Definition}

\newtheorem{hypothesis}{Hypothesis}

\renewcommand{\d}{\mathrm{d}}
\renewcommand{\rho}{\varrho}

\newcommand{\PP}{\mathbf{P}}
\renewcommand{\AA}{\mathbf{A}}

\newcommand{\FF}{\mathbf{F}}
\newcommand{\ZZ}{\mathbf{Z}}

\newcommand{\QQ}{\mathbf{Q}}
\newcommand{\RR}{\mathbf{R}}

\renewcommand{\leq}{\leqslant}

\renewcommand{\geq}{\geqslant}

\newcommand{\bs}{\boldsymbol}

\newcommand{\x}{\mathbf{x}}
\newcommand{\y}{\mathbf{y}}
\renewcommand{\v}{\mathbf{v}}
\newcommand{\uu}{\mathbf{u}}
\newcommand{\cc}{\mathbf{c}}

\newcommand{\z}{\mathbf{z}}
\newcommand{\w}{\mathbf{w}}
\renewcommand{\b}{\mathbf{b}}
\renewcommand{\a}{\mathbf{a}}

\renewcommand{\r}{\mathbf{r}}

\newcommand{\ve}{\varepsilon}
\newcommand{\e}{\mathbf{e}}

\DeclareMathOperator{\adj}{Adj}
\DeclareMathOperator{\rank}{rank}

\DeclareMathOperator{\vol}{vol}

\DeclareMathOperator{\sing}{Sing}

\DeclareMathOperator{\mat}{Mat}
\DeclareMathOperator{\diag}{Diag}
\DeclareMathOperator{\prim}{prim}

\DeclareMathOperator{\spec}{Spec}
\DeclareMathOperator{\spn}{Span}

\newcommand{\sumstar}{\sideset{}{^*}\sum}

\newcommand{\jacobi}[2]{\left(\frac{#1}{#2}\right)}

\renewcommand{\mod}[1]{\hspace{-2.9mm}\pmod{#1}}

\newcommand{\hf}[1]{h\left(r,#1\right)}

\begin{document}

\title[Counting rational points on cubic hypersurfaces]{On a conjecture of Wooley and lower bounds for cubic hypersurfaces}

\author{V.\ Vinay Kumaraswamy}
\author{Nick Rome}

\address{V.V.K. Department of Mathematics, KTH, 10044 Stockholm \\}
\email{vinay.visw@gmail.com} 

\address{N.R. TU Graz, Institute of Analysis and Number Theory, Kopernikusgasse 24/II, 8010 Graz, Austria.}
\email{rome@tugraz.at}


\begin{abstract}
Let $X \subset \PP_{\QQ}^{n-1}$ be a cubic hypersurface cut out by the vanishing of a non-degenerate rational cubic form in $n$ variables. Let $N(X,B)$ denote the number of rational points on $X$ of height at most $B$. In this article we obtain lower bounds for $N(X,B)$ for cubic hypersufaces, provided only that $n$ is large enough. In particular, we show that $N(X,B) \gg B^{n-9}$ if $n \geq 39$, thereby proving a conjecture of T. D. Wooley for non-conical cubic hypersurfaces with large enough dimension.
\end{abstract}


\date{\today}

\maketitle

\setcounter{tocdepth}{1}
\tableofcontents

\section{Introduction}

\subfile{intro.tex}

\section{Overview of the proof}\label{sectionsketch}
\subfile{sketch.tex}

\section{Preliminaries}

\subfile{preliminaries.tex}

\section{Solubility of fibres of $\pi$ and $\pi'$}\label{ekedahlapplication}
\subfile{ekedahl.tex}

\section{Counting solutions to linear equations}\label{secthunder}

\subfile{linear.tex}


\section{Proof of Theorems~\ref{thma1} and~\ref{thma2}}\label{secproofa}

\subfile{mainterm.tex}

\section{Proof of Theorem~\ref{thmb}}\label{secproofb}

\subfile{pieces.tex}

\setcounter{section}{-1}


\bibliographystyle{amsplain}
\bibliography{cubicbib}

\end{document}

%% file: intro.tex
Let  $C(x_1,\ldots,x_n) \in \ZZ[x_1,\ldots,x_n]$ be a cubic form with integer coefficients and let $X \subset \PP^{n-1}_{\QQ}$ denote the hypersurface cut 
out by the vanishing of $C$. Determining whether $X$ has a $\QQ$-rational point is a fundamental problem in number theory. This article
studies the distribution of the set of rational points $X(\QQ)$ for cubic hypersurfaces that have a $\QQ$-rational point. 
Since $C$ is a cubic form, it is clear that 
$X(\RR)\neq \emptyset.$ If $n \geq 10$, Demyanov~\cite{D50} and Lewis~\cite{L52} have 
shown that $X(\QQ_p) \neq \emptyset$ for each prime $p$. While these are clearly necessary conditions to ensure that $X(\QQ) \neq \emptyset$, a
folklore conjecture states that these conditions are also sufficient. 
\begin{conjecture}\label{conjfolklore}
Let $X\subset \PP_{\QQ}^{n-1}$ be a cubic hypersurface with $n \geq 10$. Then $X(\QQ)\neq \emptyset$.
\end{conjecture}

Although Conjecture~\ref{conjfolklore} is still unresolved, the circle method has been used to make
partial progress. Heath-Brown~\cite{HB07} has shown that 
$X(\QQ)\neq \emptyset$ provided that $n \geq 14$. 
Further, if $X$ is non-singular, a landmark result of Heath-Brown~\cite{HB83} shows that $X$ has a $\QQ$-rational point if $n \geq 10$. For non-singular 
hypersurfaces, another result by Hooley~\cite{Hooley88} has shown that if $n \geq 9$, then $X$ has a $\QQ$-rational point provided that $X$ has a 
rational point over $\QQ_p$ for each
prime $p$.

Turning to the distribution of rational points, Koll\'ar~\cite{k02} has shown that if $X$ is geometrically integral and not a cone, then $X/\QQ$ is unirational provided
that $X(\QQ) \neq \emptyset$. In particular, this implies that $X(\QQ)$ is Zariski dense in $X$. Thus it makes sense to define the counting function
$$
N(X,B) = \#\left\{\x \in X(\QQ) : H(\x) \leq B\right\},
$$
where $H(\x)$ denotes the usual naive height function on $\PP_{\QQ}^{n-1}(\QQ)$. 

In this article, we obtain lower bounds for $N(X,B)$ for general cubic hypersurfaces $X \subset \PP_{\QQ}^{n-1}$. 
Past work has primarily obtained lower bounds 
for $N(X,B)$ under additional hypotheses on $X$. Hooley~\cite{Hooley88} has shown
that if $n - \sigma \geq 10$,
then 
\begin{equation*}
N(X,B) \geq c_XB^{n-3},
\end{equation*}
where $c_X$ is a constant that depends only on $X$ and $\sigma$ is the dimension of the singular locus of $X$. Moreover, $c_X > 0$ if $X$ 
has a non-singular point over every completion of $\QQ$. 

Asymptotic formulae for $N(X,B)$ are known to hold for cubic hypersurfaces $X$ cut out 
by the vanishing of cubic forms $C$ with large enough $h$-invariant, 
a quantity which we will define below. Davenport
and Lewis~\cite{DL64} have shown that there exists $\delta > 0$ such that
\begin{equation}\label{eq:nxbasymptotic}
N(X,B) = cB^{n-3}+O(B^{n-3-\delta})
\end{equation}
holds if the $h$-invariant of $C$ is at least $17$. By combining the approach 
in~\cite{DL64} with~\cite{HB07}, one can deduce the following result, a proof of which can be found in~\cite{B23}.

\begin{theorem}\label{proph14intro}
Let $X \subset \PP_{\QQ}^{n-1}$ be a cubic hypersurface cut out by the vanishing of $C(\x)$ such that its $h$-invariant is at least $14$. Then there exists $\delta > 0$ such that the 
asymptotic formula~\eqref{eq:nxbasymptotic} holds.
\end{theorem}

In addition to the results mentioned above, generalising classical work of Birch, Davenport and Lewis~\cite{BDL62}, Schindler and Skorobogatov~\cite{SS14} have 
also obtained asymptotic 
formulae for $N(X,B)$ for the family
of hypersurfaces cut
out by the vanishing of cubic forms of the shape
$$
b_1N_{K_1/\QQ}(x_{1},x_{2},x_{3}) + b_2N_{K_2/\QQ}(x_{4},x_{5},x_{6}) + b_3N_{K_3/\QQ}(x_{7},x_{8},x_{9})
$$
where $b_i \in \ZZ$ and $N_{K_i/\QQ}$ are norm forms that arise from cubic extensions $K_i$ of $\QQ$. For $X$ associated to cubic forms 
of the shape $\sum_{i=1}^4 \Phi_i(x_i,y_i)$, where $\Phi_i$ are non-singular integral binary cubic forms, Br\"udern and Wooley~\cite{BW98} have obtained
lower bounds for $N(X,B)$.
In more recent work, Liu, Wu and Zhao~\cite{LWZ19} have 
obtained an asymptotic formula for $N(X_n,B)$ for the family of hypersurfaces  
$$
X_n: z^3 - y(x_1^2+\ldots+x_n^2) = 0,
$$
for $n \equiv 0 \bmod{4}$. If $X$ is cut out by the vanishing of a diagonal cubic form, then subject to local solubility conditions,  
an asymptotic formula for $N(X,B)$ follows from work of Vaughan~\cite{V86} if $n \geq 8$. In addition, 
lower bounds can be deduced if $n \geq 7$ by adapting the methods in~\cite{V89}. 

As a result,  
existing estimates for $N(X,B)$ are conditional on hypotheses on the singularities of $X$, on the dimension of linear subspaces in $X$, or
for hypersurfaces cut out by cubic forms that have a specific shape. 

Although there has been extensive work in understanding the counting function $N(X,B)$ for lower dimensional cubic hypersurfaces,
this will not be the focus of our present work. We refer the interested reader to
work of Blomer, Br\"{u}dern and Salberger~\cite{BBS14} on cubic fourfolds, work of de la Bret\`eche~\cite{dlB07} on Fano threefolds 
and the book~\cite[\S 2.3]{Br09} for an overview of results on cubic surfaces.

We will now define the $h$-invariant, denoted $h = h(C) = h(X)$, which will play an important role in our work. The $h$-invariant
is defined to be the smallest integer $k$ such that there exist linear forms $l_1,\ldots,l_k$  and quadratic forms $f_1,\ldots,f_k$ with integer coefficients
satisfying
$$
C(\x) = l_1(\x)f_1(\x) + \ldots + l_k(\x)f_k(\x).
$$
Equivalently, $h$ is the codimension of the largest $\QQ$-linear subspace contained in the affine cone over $X$. Then $0 \leq h \leq n$.
It is clear that $h = 0$ if and only if $C$ is identically $0$ and that 
$X(\QQ)\neq \emptyset$ if and only if $h \neq n$. Therefore, if $B$ is large enough, we see that 
\begin{equation}\label{eq:htrivial}
N(X,B) \geq 
\begin{cases}
1 &\mbox{ if $n = h+1,$} \\
c_{n-h}B^{n-h} &\mbox{if $n \geq h+2,$}
\end{cases}
\end{equation}
for some $c_{n-h} > 0$ that depends only on $X$.

Note that~\eqref{eq:htrivial} along with~\cite[Theorem 1]{HB07} and Theorem~\ref{proph14intro} imply the following `trivial' estimate.
\begin{theorem}\label{cortrivial}
Suppose that $X \subset \PP_{\QQ}^{n-1}$ is a cubic hypersurface with $n \geq 14$. Then if $B$ is large enough, we have
$$
N(X,B) \gg 
\begin{cases}
1 &\mbox{ if $n = 14,$} \\
B^{n-13} &\mbox{ if $n \geq 15$.}
\end{cases}
$$
\end{theorem}

In Section~\ref{sectionsketch} we will state our main result, which will improve on Theorem~\ref{cortrivial}. 
A simple probabilistic heuristic leads us to expect that $N(X,B) \approx B^{n-3}$. However,  
this can fail for general cubic hypersurfaces. 
For example, if $C(\x) = x_1x_2^2+x_3(x_4^2+\ldots+x_n^2)$,  
we clearly have $N(X,B) \gg B^{n-2}$. 
It is also possible to construct badly degenerate hypersurfaces $X$ that have far fewer points than predicted by the probabilistic heuristic. Consider the following example of Wooley~\cite{Wooley99}. 

Let $p$ be a fixed prime number and let $\FF_p$ denote the field with cardinality $p$. Let $K/\FF_p$ be a cubic extension with basis $\left\{\omega_1,\omega_2,\omega_3\right\}$ and consider the cubic polynomial $$
\overline{N}(x_1,x_2,x_3) = N_{K/\FF_p}(\omega_1x_1+\omega_2x_2+\omega_3x_3),
$$
where $N_{K/\FF_p}:K \to \FF_p$ denotes the norm map. Let
$N(x_1,x_2,x_3)$ be any lift of $\overline{N}$ to the integers. Define 
$$
F(x_1,\ldots,x_9) = N(x_1,x_2,x_3) + pN(x_4,x_5,x_6) + p^2N(x_7,x_8,x_9).
$$
Then it is easy to verify that if $F(x_1,\ldots,x_9) = 0$, we must have that $x_i = 0$. Let $L_1(\x),\ldots,L_9(\x)$ be $\QQ$-linearly independent 
linear forms with integer coefficients in $n \geq 9$ variables. Define
$$
C(\x) = F(L_1(\x),\ldots,L_9(\x)).
$$
If $X$ is the hypersurface cut out by the vanishing of $C$, then it follows that $N(X,B) \leq c B^{n-9}$, for some constant $c$ that depends only on $X$. Based on this example, Wooley~\cite{Wooley99} has made the following conjecture on the distribution of rational points on $X$, which may
be seen as a quantitative strengthening of Conjecture~\ref{conjfolklore}.

\begin{conjecture}[T. D. Wooley]\label{trevor}
Let $n \geq 10$ be an integer. Let $X\subset \PP_{\QQ}^{n-1}$ be a cubic hypersurface. Then there exists a constant $c > 0$ depending only on $X$ such that $N(X,B) \geq  cB^{n-9}$ as $B \to \infty$. 
\end{conjecture}

In this article, we prove Wooley's conjecture if $n$ is large enough and if $X$ is not a cone.
\begin{theorem}\label{thmw}
Suppose that $X \subset \PP_{\QQ}^{n-1}$ is a non-conical cubic hypersurface with $n \geq 39$. Then $N(X,B) \gg B^{n-9}$, where the implied
constant depends only on $X$.
\end{theorem}
In fact, one can do better for most non-conical cubic hypersurfaces $X$ with large enough dimension. 
Except for a specific family of hypersurfaces, we will show (see Theorem~\ref{cor1}
in the next section) for any $\delta > 0$ that $N(X,B) \gg B^{n-7-\delta}$ if $n \gg \delta^{-1}$. 

It would be desirous to remove the assumption in Theorem~\ref{thmw} that $X$ is non-conical, but this appears to be a 
challenging problem. However, if $X$ is cut out by the vanishing of a non-degenerate rational cubic form in
$n-l$ variables, then Wooley's conjecture holds trivially if $n-l \leq 9$ and from Theorem~\ref{thmw} if $n-l \geq 39$.

\subsection*{Notation}
By $\AA^h$ we will denote affine $h$-space. All implicit constants will be allowed to depend on the cubic form $C$ and the weight functions that will appear in the work. Any further dependence will be indicated by an appropriate subscript.

\begin{ack}
We would like to thank Adelina M\^{a}nz\u{a}ţeanu for her helpful inputs in the early stages of writing this paper. 
We are grateful to Tim Browning for bringing this problem to our attention, for his encouragement and for providing detailed comments on 
previous drafts
of this article. We would also like to thank Anish Ghosh, Jakob
Glas and Shuntaro Yamagishi for helpful conversations. 

Much of this work was done while V.V.K was a Visiting Fellow at the School of Mathematics in TIFR Mumbai, where he was supported by
a DST Swarnajayanti fellowship. He is currently supported by the grant KAW 2021.0282 from the Knut and Alice Wallenberg Foundation. N.R
is supported by Austrian Science Fund (FWF) project ESP 441-NBL. Part of this work was supported by the Swedish Research Council under grant
no. 2021-06594 while V.V.K participated in the programme in Analytic Number Theory at the Institute Mittag-Leffler in 2024.
\end{ack}

%% file: sketch.tex
In this section, we will give an outline of the proof of Theorem~\ref{thmw}. We begin by stating the following theorem, which improves on~\eqref{eq:htrivial} if $h \geq 8$ and
if $n$ is large enough in terms of $h$, and also show that it is sufficient to prove Theorem~\ref{thmw}.
\begin{theorem}\label{thmlb}
Let $X \subset \PP_{\QQ}^{n-1}$ be a non-conical cubic hypersurface cut out by the vanishing of a rational cubic form $C$ with $h$-invariant 
$h(C) \geq 8$. Assume that $n \geq h+17$. Let $\ve > 0$ be a small fixed real number. Set
\begin{align*}
\beta(n,h,\ve) &= \frac{2(h-1)(n-h-7)}{3(n-h) - 3} - 2 - \ve.
\end{align*}
and
$$
\alpha(n,h,\ve) = \min\left\{\frac{h-5}{2},\frac{2h-12}{3}, \beta(n,h,\ve)\right\}.
$$
Then we have
$$
N(X,B)  \gg_{\ve} B^{n-h+\alpha(n,h,\ve)}.
$$
\end{theorem}
To see that Theorem~\ref{thmw} can be deduced from Theorem~\ref{thmlb}, observe that Theorem~\ref{thmw} stems from~\eqref{eq:htrivial} if 
$h \leq 9$ and from Theorem~\ref{proph14intro} if $h \geq 14$. As a result, it suffices to restrict to the case where $h \in \left\{10,11,12,13\right\}$. For $h$ in this range, the estimate $N(X,B) \gg B^{n-9}$ is an immediate
consequence of Theorem~\ref{thmlb}, provided that $n \geq 39$. 

Theorem~\ref{thmlb} is the main technical result of our article, and it appears
to be the first result that obtains non-trivial lower bounds for general cubic hypersurfaces. We will give an overview of its proof below. We start by introducing the counting function 
we utilise in the proof.

Define 
\begin{equation*}
N(B) = \#\left\{\x \in \ZZ_{\prim}^n \cap B[-1,1]^n : C(\x) = 0\right\}.
\end{equation*}
As $N(B) = 2N(X,B)$, in order to obtain a lower bound for $N(X,B)$, it will suffice to study the counting function $N(B)$.

\subsection{The fibration method}
The fibration argument, which is at the heart of the proof of Theorem~\ref{thmlb}, is similar to an idea that was used
by Watson~\cite{W67} to study 
the solubility of cubic equations in at least $19$ variables. A similar approach was also deployed by Swarbrick Jones~\cite{Mike'} in his
work on weak approximation.

Recall that $h$ is the smallest integer such that $C(\x)$ can be written in the following form:
$$C(\z) = l_1(\z)f_1(\z) + \ldots + l_h(\z)f_h(\z),$$
where $l_i$ and $f_i$ are non-zero linear and quadratic forms respectively with rational coefficients. 
After making a change of variables, we may represent $C(\z)$ by 
$$
C(\z) = z_1g_1(\z) + \ldots + z_hg_h(\z),
$$ where $g_i$ are non-zero quadratic forms. Note that this (or any other) linear change of variables will only affect the implied constant 
in the lower bound we will obtain for $N(B)$, which is allowed to depend on $C$. 

Set $\z = (\x,\y)$ where $\x = (x_1,\ldots,x_{n-h})$ and $\y = (y_1,\ldots,y_h)$ so that 
\begin{equation}\label{eq:cubicxy}
\begin{split}
C(\z) = C(\x,\y) &= y_1g_1(\x,\y) + \ldots + y_h g_h(\x,\y) \\ 
&= \sum_{i=1}^hy_iF_i(\x) + \sum_{j=1}^{n-h}x_jq_j(\y)+R(\y),
\end{split}
\end{equation}
for quadratic forms $F_i(\x)$ and $q_j(\y)$ and a cubic form $R(\y)$.

Let $C$ be as in~\eqref{eq:cubicxy}. Let $Z \subset \AA^n$ denote the (affine) hypersurface cut out by the equation $C = 0$. Let
\begin{equation}\label{eq:map}
\begin{split}
\pi: Z& \to \AA^h \\
(\x,\y) &\mapsto \y
\end{split}
\end{equation}
be the projection map. Then the fibre over a point $\y \in \AA^h$ is given by
$$
Z_{\y}:  \sum_{i=1}^hy_iF_i(\x) + \sum_{j=1}^{n-h}x_jq_j(\y)+R(\y) = 0.
$$

Associated to each cubic hypersurface is the following invariant, which will play a key role in our fibration argument. 
\begin{definition}\label{genericrank}
Let $Z_{\eta}$ denote the generic fibre of the morphism $\pi$, where $\eta$ is the generic point in $\AA^h$. 
The rank of the fibration $\pi$, denoted $r = r(C)$, is defined to be the rank of the quadratic part of $Z_\y$.
\end{definition} 
In other words, $r$ is the rank of the quadratic form
$\sum_{i=1}^h y_iF_i(\x)$ over the algebraic function field $\QQ(y_1,\ldots,y_h)$. More concretely, it is the order of the largest minor of $M[\y]$, the matrix 
associated with $\sum_{i=1}^h y_iF_i(\x)$, which does not vanish identically in $\y$. Note that $0 \leq r \leq n-h$. 

If $r$ is the rank of $\pi$, then for `typical' $\y \in \ZZ^h$, we will have that $\rank_{\QQ} \sum_{i=1}^h y_iF_i(\x) = r$. Let 
$$
F_\y(\x) = C(\x,\y) \text{ and } Q_\y(\x) = \sum_{i=1}^h y_iF_i(\x),
$$
so that the fibre of $\pi$ over 
$\y$ is nothing but
$Z_\y : F_\y(\x) = 0.$ When $r$ is `large', 
we will use use the circle method to count points on the affine quadric $Z_\y$ and sum over $\y$ to get a lower bound for $N(B)$, i.e.,
\begin{equation}\label{eq:74430}
N(B) \geq \sum_{|\y| \leq B}\sum_{\substack{|\x| \leq B  \\ (\x,\y)=1 \\ F_\y(\x) = 0}}1.
\end{equation}
For this strategy to work, we must ensure that the equation $F_\y(\x)=0$ is soluble for sufficiently many $|\y| \leq B$, and also
obtain estimates for counting solutions to $F_\y = 0$ that are uniform in $\y$. To tackle the 
latter problem, we will use the smooth $\delta$-function form of the circle method~\cite[Theorem 1]{HB96}. 

We will encounter contrasting behaviour based on the real solubility of $Q_\y(\x)$, which we will now describe. 
If for each $\y$, the quadratic
form $Q_\y(\x)$ is semidefinite, then
we will show (see Lemma~\ref{lemmawt}) that $Q_\y(\x) = l(\y)F(\x)$, where $l$ is a linear form and $F$ is a semidefinite quadratic form. 
Our problem then reduces to counting solutions to a given definite quadratic form by `completing the square' to remove terms that are linear in $\x$ in
~\eqref{eq:cubicxy}, while ensuring that the resulting equation is soluble over $\ZZ$
for sufficiently many $\y \in B[-1,1]^n \cap \ZZ^n$. Using this approach, we will prove the following result in Section~\ref{proofthma1}.
\begin{theorem}\label{thma1}
Let $C(x_1,\ldots,x_n)$ be a rational cubic form as in~\eqref{eq:cubicxy} with $h=h(C)$ and let
$Q_\y(\x) = \sum_{i=1}^hy_iF_i(\x)$. 
Let $X \subset \PP_{\QQ}^{n-1}$ denote the cubic hypersurface $C = 0$.
Suppose that 
$Q_\y(\x) = l(\y)F(\x)$, with $l$ a linear form and $F$ a semidefinite quadratic form of rank equal to $r \geq 5$, where $r$ is the rank of the fibration $\pi$. 
Set 
$$
\gamma(n,h,r) = \min\left\{\frac{h-5}{2},\frac{r-n+3h-8}{3}\right\}.
$$ 
Then we have
$$
N(X,B) \gg B^{n-h+\gamma(n,h,r)}.
$$
\end{theorem}

On the other hand, if $Q_\y$ is not as in Theorem~\ref{thma1}, we proceed as follows. 
Let $1 \leq Y \leq B$ be a parameter. We will construct a non-empty compact set $\Omega_{\infty}$ 
such that whenever $\y \in \ZZ^h \cap Y\Omega_{\infty}$, the quadratic form $Q_\y(\x)$
is indefinite and
the product of its non-zero eigenvalues has absolute value $\gg Y^r$, where $r = r(C)$ is the rank of $\pi$. In particular, this product is, up to a constant, 
as large as it can be. This is because the eigenvalues of $Q_\y$ are $O(Y)$, as the matrix associated to $Q_\y$ has coefficients
that are $O(Y)$. Furthermore, if the set of minors of order $3$ in $M[\y]$ have
no common factor, we will show in Section~\ref{ekedahlapplication} using an 
application of the Ekedahl
sieve~\cite{E91} that for a positive proportion of integer vectors $\y \in Y\Omega_{\infty}$,  
the equation $F_{\y}(\x) = 0$ has non-singular solutions modulo $p$ for each prime $p \gg_C 1$. Combining this with the fact that
the smooth points of $X(\QQ_p)$ are Zariski dense in $X$, we deduce that a positive proportion of fibres of $\pi$ are 
everywhere locally soluble. Hence by the Hasse principle for affine quadrics, we get that for a positive proportion of $\y \in \ZZ^h \cap Y\Omega_{\infty}$, 
the equation $F_\y = 0$ is soluble over $\ZZ$. 

Next, we deploy the circle method, as in work of the first author~\cite{VVK24}, to show for a smooth function $w$ that
\begin{equation}\label{eq:introspiel}
\sum_{F_\y(\x) = 0}w(B^{-1}\x) = \sigma_{\infty}(w,\y)\mathfrak{S}(\y)B^{n-h-2} + O(\mathcal{E}(B,Y)),
\end{equation}
where $\sigma_{\infty}(w,\y)$ and $\mathfrak{S}(\y)$ are the `singular integral', which measures the density of real solutions
to the equation $F_\y = 0$ and the `singular series', which measures the density of $p$-adic solutions for each prime $p$ and $\mathcal{E}(B,Y)$ 
is the error term. It is important here that the error term is uniform in the coefficients of $F_\y$ and 
the parameter $Y$ is chosen to ensure that the main term exceeds the error term. We also remark that the error term in~\eqref{eq:introspiel}
grows larger as $r$ becomes smaller.  

In order to sum over $\y$ in~\eqref{eq:74430} using~\eqref{eq:introspiel}, we will need good lower bounds for both $\sigma_{\infty}(w,\y)$ as well 
as for $\mathfrak{S}(\y)$. However, having
ensured that the absolute value of the eigenvalues of $Q_\y$ is $\gg Y^r$, and that $F_\y$ has a smooth solution modulo $p$ for each $p \gg_C 1$,
we can deduce that $\sigma_{\infty}(w,\y)\mathfrak{S}(\y) \gg Y^{-1-\ve}$.  
This allows us to conclude that
$$
N(B) \gg B^{n-h-2}Y^{h-1-\ve},
$$
for $Y=B^{\delta}$, and $\delta = \delta(n,h,r) > 0$. 

It still remains to treat the case where the minors of order $3$
in $M[\y]$ have a common factor, in which case we can no longer apply the Ekedahl sieve. Nevertheless, we are
able to obtain optimal bounds for $\mathfrak{S}(F_\y)$ on average. To do this, we begin by characterising those cubic forms that have this property and
then we use exponential sum techniques over finite fields to handle the resulting counting problems.
These are some of the key ideas that go into the proof of the following result, which is the subject of Section~\ref{proofthma2}. 
\begin{theorem}\label{thma2}
Let $C(x_1,\ldots,x_n)$ be a rational cubic form as in~\eqref{eq:cubicxy} with $h = h(C) \geq 6$ and let
$Q_\y(\x) = \sum_{i=1}^hy_iF_i(\x)$. 
Let $X \subset \PP_{\QQ}^{n-1}$ 
denote the cubic hypersurface $C = 0$. Let $r$ denote the rank of the fibration $\pi$ and assume that $r \geq \min\left\{5,n-h-4\right\}$. 
Suppose that $Q_\y(\x) \neq l(\y)F(\x)$, for any linear form $l$ and any semidefinite
quadratic form $F$ of rank $r$.
Let $\ve > 0$ be a small fixed real number. Set
\begin{align*}
\delta(n,h,r,\ve) = \begin{cases} \frac{2(r-4)(h-1)}{r+2(n-h)}-2-\ve &\mbox{ if $r$ is even and $2r-m < 8$,} \\
\frac{2(r-3)(h-1)}{r+2(n-h)+1}-2-\ve &\mbox{ otherwise.}
\end{cases}
\end{align*}
Then we have
$$
N(X,B) \gg_{\ve} B^{n-h+\delta(n,h,r,\ve)}.
$$
\end{theorem}

This leaves us with those hypersurfaces 
for which $r \leq n- h- 5$. The following result allows us to extract additional information regarding such $C$.
\begin{lemma}\label{lemma4w}
Let $C(\x,\y)$ be as in~\eqref{eq:cubicxy}. If $r=r(C) < n-h$ then after
a linear change of variables, $C$ is linear in $n-h-r$ of the $x_i$. 
\end{lemma}
\begin{proof}
We will argue as in the proof of~\cite[Lemma 3]{W67}.
 Let $M_i$ denote the matrix associated with the quadratic
form $F_i$. Since $r$ is the rank of the fibration $\pi$, at least one of the matrices $M_i$ must have rank equal to $r$. We may assume without loss of generality that $\rank(M_1) = r$. Furthermore, we can make a change of variables and ensure that 
$M_1 =\diag (a_1,\ldots,a_r,0,\ldots,0)$ is a diagonal matrix. Then we may write
\begin{equation*}
M[\y] = \begin{pmatrix}
\sum_{i = 1}^{h}y_iN_{11}^{(i)} & \sum_{i = 2}^{h}y_iN_{12}^{(i)} \\
\sum_{i = 2}^{h}y_iN_{21}^{(i)} & \sum_{i = 2}^{h}y_iN_{22}^{(i)}
\end{pmatrix}, 
\end{equation*} 
with $N_{11}^{(j)}$ matrices of order $r$, such that $N_{11}^{(1)} = M_1$ and $N_{22}^{(j)}$ matrices of order $n-h-r$.

To show that $C(\x,\y)$ is linear in $x_i$ with $r+1 \leq i \leq n-h$ it suffices to show that $\frac{\partial^2 C}{\partial x_i x_j} = 0$ if $r+1 \leq  i,j \leq n-h$, or in other words that each matrix $N_{22}^{(j)}$ has all its elements equal to $0$.

By hypothesis, any minor of order $r+1$ in $M[\y]$ must vanish identically in $\y$. Consider in particular the $(r+1)\times (r+1)$ submatrices of 
$M[\y]$ of the form
$$
\begin{pmatrix}
\sum_{i=1}^h y_iN_{11}^{(i)} & \uu \\
\v & w
\end{pmatrix},
$$
for any column $\uu$ of $\sum_{i = 2}^{h}y_iN_{12}^{(i)}$, any row $\v$ of $\sum_{i = 2}^{h}y_iN_{21}^{(i)}$ and any matrix entry 
$w \in \sum_{i = 2}^{h}y_iN_{22}^{(i)}$, obtained by bordering $\sum_{i=1}^h y_i N_{11}^{(i)}.$ In the determinant of each of these matrices, consider the coefficients of $y_1^{r}y_j$, where $j \geq 2$. These terms constitute the determinant of the matrix 
$$
\begin{pmatrix}
M_1 & 0 \\
0 & \sum_{i=2}^hy_iN_{22}^{(i)}
\end{pmatrix}.
$$
As a result, there are no linear terms (in the $y_j$, with $j \geq 2$) in any minor of order $r+1$ of the above matrix. Since $\det M_1 \neq 0$, we have $N_{22}^{(i)} = 0$ and the lemma follows.
\end{proof}

Observe that if $r = n-h-k$, then $C$ is linear in $k$ variables. As a result, after relabelling the variables, we can write
\begin{equation}\label{eq:cxylinearink}
C(\x,\y) = x_1Q_1(\y)+\ldots+x_kQ_k(\y) + S(\y),
\end{equation}
where $\y = (y_1,\ldots,y_{n-k})$, $Q_i$ are quadratic forms and $S$ is a cubic form. Let $Z \subset \AA^{n}$ denote the hypersurface $C = 0$. 
Consider the morphism 
\begin{equation}\label{eq:fibrationpi'}
\begin{split}
\pi': Z &\to \AA^{n-k} \\
(\x,\y) &\mapsto \y.
\end{split}
\end{equation}
Then the fibres of $\pi'$ are the affine hyperplanes
$$
Z_\y : x_1Q_1(\y)+\ldots+x_kQ_k(\y) + S(\y) = 0.
$$
For $C$ as in~\eqref{eq:cxylinearink}, we obtain near optimal lower bounds for $N(X,B)$.
\begin{theorem}\label{thmb}
Let $C(x_1,\ldots,x_n)$ be a rational cubic form in $n$ variables such that $h(C) \geq 8$. 
Suppose that $C(\x)$ is of the shape~\eqref{eq:cxylinearink} with 
$k \geq 5$. 
Assume that $C(\x)$ is non-degenerate in at least $5$ of the variables $x_1,\ldots,x_k$. 
Let $X \subset \PP_{\QQ}^{n-1}$ denote the cubic hypersurface $C = 0$. Then for any $\ve > 0$, we have
$$
N(B) \gg_{\ve} B^{n-3-\ve}.
$$
\end{theorem}

To prove Theorem~\ref{thmb}, we use the morphism $\pi'$ to count points on $X$. Let 
$$
l_\y(\x) = C(\x,\y).
$$
The fibres of $\pi'$ are given by $
Z_\y : l_\y(\x) = 0.
$
Our task then will be to count solutions to $l_\y(\x) = 0$ with $|\x| \leq B$ for each fixed 
$\y = (y_1,\ldots,y_k)$ with $|\y| \leq Y$ for some parameter $1 \leq Y \leq B$. 

If the quadratic forms $Q_i$ in~\eqref{eq:cxylinearink} have no common linear factor, 
we will once again show using the Ekedahl
sieve
in Section~\ref{ekedahlapplication} that a positive 
proportion of fibres of $\pi'$ have
an integer solution. If this is not the case, we will use a Lang--Weil type estimate for cubic forms
due to Davenport and Lewis~\cite{DL62} to reach a similar conclusion. 

To count solutions to $l_\y(\x) = 0$, we adapt work of Thunder~\cite{thunder} to obtain (see Section~\ref{secthunder})
an asymptotic formula
whose error term depends on the successive minima of the lattice
$$
\Lambda_\y = \left\{\x \in \ZZ^k : \sum_{i=1}^kx_iQ_i(\y) = 0\right\}.
$$
It will be important to show that the first successive minimum of $\Lambda_\y$ is `large' on average. 
In this endeavour, we are led to calculating the ranks of certain quadric bundles
over algebraic function fields (see Section~\ref{secranksffields}). Using this and bounds towards the Dimension Growth Conjecture, 
we complete the proof of Theorem~\ref{thmb}, which is the carried out in Section~\ref{secproofb}.

With Theorems~\ref{thma1},~\ref{thma2} and~\ref{thmb} at hand, Theorem~\ref{thmlb} follows from Lemma~\ref{lemma4w} by the following argument.
Let $r$ denote the rank of the fibration $\pi$ in~\eqref{eq:map}. Suppose that $n-h-4 \leq r \leq n-h$. Then Theorem~\ref{thmlb} will follow from 
Theorem~\ref{thma1} if $Q_{\y}(\x) = l(\y)F(\x)$, for some linear form $l$ and semidefinite quadratic form $F$. If this is not the case, Theorem~\ref{thmlb}
follows from the lower bound in Theorem~\ref{thma2}, as $h \geq 8$ by assumption. 
Finally, if $r \leq n-h-5$, then $C(\x)$ is linear in five or more variables by Lemma~\ref{lemma4w},
whence $C(\x)$ is of the shape in~\eqref{eq:cxylinearink}. By our assumption that $C$ is non-degenerate, we may appeal to Theorem~\ref{thmb}
to complete the proof of Theorem~\ref{thmlb}.

We end this section by highlighting the following result, which is a consequence of Theorems~\ref{thma2} and~\ref{thmb}. 
\begin{theorem}\label{cor1}
Let $C(x_1,\ldots,x_n)$ be a non-degenerate cubic form as in~\eqref{eq:cubicxy} and let
$Q_\y(\x) = \sum_{i=1}^hy_iF_i(\x)$. Assume that $Q_\y(\x) \neq l(\y)F(\x)$, for any linear form $l$ and any semidefinite
quadratic form $F$ of rank $r$. Let $X \subset \PP_{\QQ}^{n-1}$ denote the hypersurface defined by $C = 0$. Let $\delta > 0$. Then if $n \gg \delta^{-1}$, 
then 
$$
N(X,B) \gg_{\delta} B^{n-7-\delta}.
$$
\end{theorem}
\begin{proof}
By Theorem~\eqref{eq:htrivial} and~\ref{proph14intro} it suffices to prove this for cubic hypersurfaces with $8 \leq h \leq 13$. 
As $Q_\y \neq l(\y)F(\x)$, 
we may apply Theorem~\ref{thma2} if the rank of $\pi$ is at least $n- h - 4$, and Theorem~\ref{thmb} otherwise. 
In the former case, observe that for each fixed $h$, $\gamma(n,h,\ve) \to 2(h-1)/3-2-\ve$ as $n \to \infty$. 
As a result, if $n \gg \delta^{-1}$, we get that $N(X,B) \gg B^{n-h-2+\frac{2(h-1)}{3}-\delta} \gg B^{n-7-\delta}$ and the corollary follows.
\end{proof}
Theorem~\ref{cor1} is the limit of our method, and it would be interesting to prove that a similar lower bound also holds for cubic hypersurfaces
as in the statement of Theorem~\ref{thma1}. We end by remarking that with additional work, it is possible to remove the hypothesis $h \geq 8$ in Theorem~\ref{thmb} by 
appealing to~\cite{Browning13} instead of~\cite[Theorem 3]{DL62}. However, we have decided not to do this here, in light of the estimate in Theorem~\ref{cor1}.

%% file: preliminaries.tex
\subsection{Cubic forms over $\QQ_p$}

In the proof of Theorems~\ref{thmsolubility2} and~\ref{thmsolubility}, we will need the following technical lemma, which asserts the existence
of non-singular $\ZZ_p$ points on the fibres $Z_\y$ we will encounter. The lemma is identical to~\cite[Lemma 4]{W67}. However, we 
record a proof below (for $C$ that is absolutely irreducible) that was suggested to us by Tim Browning.
\begin{lemma}\label{lemma0}
Let $C(x_1,\ldots,x_l,y_{l+1},\ldots,y_n)$ be a cubic polynomial with integer coefficients in $n$ variables such that
the partial derivatives $\frac{\partial C}{\partial x_i}$ do not all vanish identically. 
Assume that $C$ is irreducible. Let $p$ be a prime and assume that $C=0$ has a 
non-singular solution over $\QQ_p$. 
Then 
there exists $v = v_p\geq 1$ and $(\x,\r) \in \left(\ZZ/p^{2v-1}\ZZ\right)^n$ such that
$$
C(\x,\r) \equiv 0 \bmod{p^{2v-1}} \text{ and } \left(\frac{\partial C(\x,\r)}{\partial x_1},\ldots, \frac{\partial C(\x,\r)}{\partial x_l}\right) \not\equiv \mathbf{0} \bmod{p^v}.
$$
\end{lemma}
\begin{proof}
Given the existence of a non-singular solution the equation $C = 0$ over $\QQ_p$, it follows from the $p$-adic implicit function theorem
that
the $\QQ_p$ solutions to $C=0$ are Zariski dense (see, for example,~\cite[Lemma 3.4]{BDHB15}. The lemma
follows from our assumption that the partial derivatives $\frac{\partial C}{\partial x_i}$ do not all
vanish identically. 
\end{proof}

\subsection{Non-singular solutions to polynomial equations}
In this section, we will study the existence of non-singular solutions to quadratic polynomials. We begin
with the following lemma, which is a generalisation of~\cite[Lemma 1]{BD08}.
\begin{lemma}\label{lemmaevalgauss}
Let $p \neq 2$ be a prime. Let 
$$
F(x_1,\ldots,x_r) = Q(x_1,\ldots,x_r) + \sum_{i=1}^r B_ix_i + N
$$ 
be a quadratic polynomial with 
integer coefficients. Suppose that $Q$ is of rank $r$ and that $p$ does not divide the discriminant of $Q$. Let
$M$ denote the matrix associated with $Q$ and let $\bs{B} = (B_1,\ldots,B_r)$. Let $\ve_p = 1$ if $p \equiv 1 \bmod{4}$ and equal to $i$ if 
$p \equiv 3 \bmod{4}$. Then we have
\begin{align*}
\#\left\{\x \in \FF_p^m: F(\x) = 0 \mbox{ and } \nabla F(\x) \neq \mathbf{0}\right\} &= p^{m-1} +\ve_p^r \jacobi{\det M}{p} \times \\
&\quad p^{\frac{r}{2}-1}K_r(4N-\bs{B}^tM^{-1}\bs{B},p) \\
& - \kappa_p,
\end{align*}
where
\begin{equation}\label{eq:evalexpsums}
\begin{split}
K_r(w;p) =\begin{cases} \ve_p\jacobi{w}{p}p^{1/2}&\mbox{ if $r$ is odd,}\\
p-1 &\mbox{ if $p \mid w$ and $r$ is even,}\\
-1&\mbox{ if $p \nmid w$ and $r$ is even}
\end{cases}
\end{split}
\end{equation}
and
$$
\kappa_p = 
\begin{cases} 1 &\mbox{ if $p \mid 4N-\bs{B}^tM^{-1}\bs{B},$} \\
0 &\mbox{ otherwise.}
\end{cases}
$$
\end{lemma}

\begin{proof}
Let $\mathcal{N}^*$ denote the cardinality of non-singular solutions to the equation $F(\x) = 0$ in $\FF_p$. Then we have
\begin{align*}
\mathcal{N}^* = \frac{1}{p}\sum_{\substack{\x \bmod{p} \\ \nabla F(\x) \neq \mathbf{0}}}\sum_{a \mod{p}}e_p(aF(\x)).
\end{align*}
The condition $\nabla F(\x) \neq \mathbf{0}$ implies that $2M\x \not\equiv -\bs{B} \bmod{p}$, i.e., 
$\x \not\equiv -\overline{2}M^{-1}\bs{B} \bmod{p}$. Here, and in the rest of the paper, for $c \in \FF_p^*$, $\overline{c}$
will denote the multiplicative inverse of $c$. As 
$$
F(-\overline{2}M^{-1}\bs{B}) \equiv 
-\overline{4}\bs{B}^tM^{-1}\bs{B} + N \bmod{p},
$$
we have that
\begin{align*}
\mathcal{N}^* &= \frac{1}{p}\sum_{a \mod{p}}\sum_{\substack{\x \bmod{p}}}e_p(aF(\x)) \\
&\quad \quad - \frac{1}{p}\sum_{a \mod{p}}e_p(-a(4N - \bs{B}^tM^{-1}\bs{B}) \\
&= \frac{1}{p}\sum_{a \mod{p}}\sum_{\substack{\x \bmod{p}}}e_p(aF(\x)) - \kappa_p.
\end{align*}

To analyse the first sum, we recall that there exists a matrix 
$R$ such that $R^tMR = \diag (A_1,\ldots,A_r)$ with with the property that $p \nmid \det R$ and
$p \nmid A_i$. Put $\bs{D}= R^t\bs{B}.$ Replacing $\x$ with $R^t\x$, we get that
\begin{align*}
\mathcal{N}^* &= p^{m-1} + \frac{1}{p}\sumstar_{a \mod{p}}e_p(aN)\prod_{i=1}^r\sum_{x_i \mod{p}}e_p(a(A_ix_i+D_ix_i)) - \kappa_p \\
&= p^{m-1} + \ve_p^r \jacobi{\det M}{p} p^{\frac{r}{2}-1} \sumstar_{a \mod{p}}\jacobi{a}{p}^re_p(a(N - \overline{4}\sum_{i=1}^r \overline{A_i}D_i^2) \\
&\quad \quad -\kappa_p,
\end{align*}
where $\sum^*$ denotes restriction to coprime residue classes modulo $p$. 
The lemma follows from observing that 
\begin{align*}
\sum_{i=1}^r \overline{A_i}D_i^2 &= \bs{D}^t\diag(\overline{A_1},\ldots,\overline{A_r})D \\
&= \bs{B}^tRR^{-1}M^{-1}(R^t)^{-1}R^t\bs{B} = \bs{B}^tM^{-1}\bs{B}.
\end{align*}
The evaluation of $K_r(w;p)$ in~\eqref{eq:evalexpsums} is well-known and this completes the proof of the lemma.
\end{proof}
As an application of the preceding lemma, we have 
\begin{lemma}\label{lemmasolvequadricsmodp}
Let $p \neq 2$ be a prime. Let $F(x_1,\ldots,x_m) = Q(x_1,\ldots,x_m) + L(x_1,\ldots,x_m) + N$ be a quadratic polynomial with integer coefficients,
 where $Q(\x)$ is a quadratic form of rank at least $3$ over $\FF_p$, $L(\x)$ is a linear form and $N$ is an integer. Then 
$$
\#\left\{\x \in \FF_p^m: F(\x) = 0 \mbox{ and } \nabla F(\x) \neq \mathbf{0}\right\} = p^{m-1} + O(p^{m-2}).
$$
\end{lemma}
\begin{proof}
Since $p \neq 2$, we can diagonalise $Q$ over $\FF_p$. As a result, we may assume without loss of generality that 
$$
F(\x) = \sum_{i=1}^rA_ix_i^2+\sum_{j=1}^m B_ix_i + N,
$$ 
with $r \geq 3$, $p \nmid A_1,\ldots,A_r$ and $p \mid A_{r+1},\ldots,A_m$.  Let $\mathcal{N}^*$ denote the number of solutions $\x \in \FF_p$ to the equation $F(\x) = 0$ such that
$\nabla F(\x) = (2A_1x_1+B_1,\ldots,2A_nx_m+B_m) \not\equiv \mathbf{0} \bmod{p}$. Since $p \mid A_i$ for $i > r$, we see that $\nabla F(\x) \not\equiv \mathbf{0} \bmod{p}$ if and only if 
$$
(2A_1x_1+B_1,\ldots,2A_rx_r+B_r,B_{r+1},\ldots,B_m) \not\equiv \mathbf{0} \bmod{p}.
$$
Observe that if at least one of $B_{r+1},\ldots,B_m$ is coprime to $p$, $B_{r+1}$, say, then fixing all the other $x_i$ fixes $x_{r+1}$, whence it follows 
that $\mathcal{N}^* = p^{m-1}$ in this case. As a result, we may assume for the rest of the proof that $p \mid B_{r+1},\ldots,B_n$, in which case, we have
$
F(\x) = \sum_{i=1}^r (A_ix_i^2+B_ix_i)+ N
$
over $\FF_p$.

Consequently, we get from Lemma~\ref{lemmaevalgauss} that
\begin{align*}
\mathcal{N}^* &= p^{m-1} \\
&\quad + \ve_p^r\jacobi{\prod_{i=1}^rA_i}{p}p^{m-\frac{r}{2}-1}K_r(4N-\sum_{i=1}^r\overline{A_i}B_i^2;p) - \kappa_p. \\
\end{align*}
The lemma follows from~\eqref{eq:evalexpsums}, since $r \geq 3$.
\end{proof}

Next we will now record a well-known result due to Davenport on lifting non-singular solutions from 
prime moduli to congruences modulo prime powers.
\begin{lemma}\label{lemmadavhensel}
Let $G(x_1,\ldots,x_m)$ be a polynomial equation with integer coefficients. Suppose that
\begin{equation*}\label{eq:hyplemmalbforsfy1}
\#\left\{\x \in \FF_p^m : G(\x) = 0 \text{ and } \nabla G(\x) \neq \mathbf{0}\right\}  \geq p^{m-1} + O(p^{m-2})
\end{equation*}
Then for all $t \geq 1$ we have 
$$
\#\left\{\x \bmod{p^t}: G(\x) \equiv 0 \bmod{p^t}\right\} \geq p^{t(m-1)} + O(p^{t(m-1)-1}).
$$
Suppose that the set
\begin{equation*}\label{eq:hyp2lemmalbforsfy1}
\begin{split}
\left\{ \x \bmod{p^{2v_p-1}} \right.: 
&\left.G(\x) \equiv 0 \bmod{p^{2v_p-1}} \text{ and } \nabla G(\x) \not\equiv \mathbf{0} \bmod{p^{v_p}}\right\}
\end{split}
\end{equation*} 
is non-empty for some $v_p \geq 1$.
Then for all $t \geq 2v_p-1$ we have
$$
\#\left\{\x \bmod{p^t}: G(\x) \equiv 0 \bmod{p^t}\right\} \geq p^{-(2v_p-1)(m-1)}p^{t(m-1)} .
$$
\end{lemma}
\begin{proof}
This follows from~\cite[Lemma 17.1]{Davenport05}.
\end{proof}

\subsection{Calculating the rank of certain quadrics over function fields}\label{secranksffields}
Let $v \geq 4$ be an integer. Let $K$ denote the function field  $\QQ(x_1,\ldots,x_v)$ in $m$ variables. Let $\psi_1(\y),\ldots,\psi_v(\y)$ be 
integral quadratic forms. Set 
$$
\Psi(\x,\y) = x_1\psi_1(\y) + \ldots + x_v\psi_v(\y).
$$
Over $K$, the polynomial $\Psi$ is a quadratic form in $\y$. Note that we may diagonalise $\Psi$ over $K$. As a result, we have $\Psi(\x,\y) = \sum_{i=1}^r u_i(\x)y_i^2$, where $u_i(\x) \in K$ and $r = \rank_K \Psi$ is an invariant of $\Psi$. In this section, we will show that $\Psi$ is geometrically irreducible over $K$ provided that $\Psi(\x,\y)$ is irreducible over $\QQ$ and that $\Psi$ is non-degenerate in the $x_i$ variables. The following result is well-known. 
\begin{lemma}\label{quadred}
Suppose that $Q(\z)$ is a quadratic form over an algebraically closed field $k$. Then $Q(\z)$ is reducible if and only if $\rank Q \leq 2$.
\end{lemma}

Write $\Psi(\x,\y) = \sum_{1 \leq i,j \leq m}a_{i,j}(\x)y_iy_j$ with $a_{i,j}(\x)$ linear forms in $\x$ with integer coefficients. Suppose that $r = 2$. 
This implies that $\rank \psi_i \leq 2$ for each $1 \leq i \leq v$ and $\rank \psi = 2$ for some $i$. 
Assume without loss of generality that $\rank \psi_1 = 2$. Then
by making a linear change of variables, we may assume that $\psi_1(\y) = a_1y_1^2+a_2y_2^2$, for some non-zero
integers $a_1$ and $a_2$. 
Furthermore, by Lemma~\ref{lemma4w} we see that $a_{i,j}(\x)$ is identically zero if $i,j \geq 3$. Also observe that
$a_{k,l}(\x)$ is independent of $x_1$ whenever $k \neq l$. 

We clearly have $a_{1,1}(\x) \in K^*$. Completing the square, we get
\begin{equation}\label{eq:temp1125}
\begin{split}
\Psi(\x,\y) &= a_{1,1}(\x)\left(y_1 + \frac{a_{1,2}(\x)}{2a_{1,1}(\x)}y_2 + \ldots + \frac{a_{1,n}(\x)}{2a_{1,1}(\x)}y_m\right)^2 + \\
&\quad \quad G(\x,y_2,\ldots,y_m),
\end{split}
\end{equation}
where
\begin{align*}
G(\x,y_2,\ldots,y_m) &= \sum_{2\leq i,j \leq m}\left(a_{i,j}(\x) - \frac{a_{1,i}(\x)a_{1,j}(\x)}{4a_{1,1}(\x)}\right)y_iy_j.
\end{align*}
Set 
\begin{align*}
\Delta_{i,j}(\x) = 4a_{1,1}(\x)a_{i,j}(\x) - a_{1,i}(\x)a_{1,j}(\x).
\end{align*}
Note that $a_{1,i}(\x)/2a_{1,1}(\x), \Delta_{i,j}(\x) \in K$ and $\Delta_{i,j}(\x) = -a_{1,i}(\x)a_{1,j}(\x)$ if $i,j \geq 3$, as 
$a_{i,j}(\x) = 0$ whenever $i,j \geq 3$. Observe that $\rank_K \Psi = 2$ if and only if $\rank_K G = 1$. Suppose that this is the case. As
$a_{1,2}(\x)$ is independent of $x_1$, we find that $\Delta_{2,2}(\x) \in K^*$. As a result, we may write
\begin{align}\label{eq:temp1126}
G(\x,y_2,\ldots,y_n) = \frac{\Delta_{2,2}(\x)}{4a_{1,1}(\x)}\left(y_2+\frac{\Delta_{2,3}(\x)}{2\Delta_{2,2}(\x)}y_3+\ldots+\frac{\Delta_{2,n}(\x)}{2\Delta_{2,2}(\x)}y_m\right)^2.
\end{align}
This implies that
\begin{align}\label{eq:relationwithdeltas}
4\Delta_{2,2}(\x)\Delta_{i,j}(\x) = \Delta_{2,i}(\x)\Delta_{2,j}(\x)
\end{align}
for each $i,j \geq 3$. The above equation constrains the possibilities for $a_{i,j}(\x)$, as we show below.
\begin{lemma}\label{lemmaslow0}
Suppose that ~\eqref{eq:relationwithdeltas} holds for each $3 \leq i, j \leq m$. If $\Delta_{2,i}(\x) = 0$ for some $i \geq 3$, then 
$a_{1,i}(\x) = a_{2,i}(\x) = 0.$
\end{lemma}

\begin{proof}
Observe that if $\Delta_{2,i}(\x) = 0$ for some $i \geq 3$, then appealing to~\eqref{eq:relationwithdeltas} with $j = i$, we get that
$
4\Delta_{2,2}(\x)\Delta_{i,i}(\x) = 0$. As $\Delta_{2,2}(\x) \in K^*$, we find that $\Delta_{i,i}(\x) = 0$, which in turn implies that
$a_{1,i}(\x) = 0$. As a consequence, we get that $\Delta_{2,i}(\x) = 4a_{1,1}(\x)a_{2,i}(\x) = 0$, whence we also get that $a_{2,i}(\x) = 0$. 
The lemma follows.
\end{proof}

\begin{lemma}\label{lemmaslow}
Suppose that~\eqref{eq:relationwithdeltas} holds for each $3 \leq i, j \leq m$. Suppose that
$\Delta_{2,i}(\x) \neq 0$ for some $3 \leq i \leq m$. Then $a_{1,2}(\x) = 0$ and there exists $\kappa_i \in \QQ$ such that $a_{2,i}(\x) = \kappa_i a_{1,i}(\x)$
and $a_{2,2}(\x) = -\kappa_i^2a_{1,1}(\x)$.
\end{lemma}
\begin{proof}
As $i \geq 3$, we have that $\Delta_{i,i}(\x) = - a_{1,i}^2(\x)$. 
Therefore, applying~\eqref{eq:relationwithdeltas} with $j = i$ we get that $\Delta_{2,i}^2(\x) = -4\Delta_{2,2}(\x)a_{1,i}^2(\x).$ This implies that
there exists a linear form $l_i(\x)$ such that $\Delta_{2,i}(\x) = a_{1,i}(\x)l_i(\x)$ and $-4\Delta_{2,2}(\x) = l_i^2(\x)$. Note that by our hypothesis 
that $\Delta_{2,i} \in K^*$, we must have that $a_{1,i},l_i \in K^*$.

Since
$a_{1,i}(\x)$ divides $\Delta_{2,i}(\x) = 4a_{1,1}(\x)a_{2,i}(\x) - a_{1,2}(\x)a_{1,i}(\x)$ and since $a_{1,i}(\x)$ is 
independent of $x_1$, we see that $a_{1,i}(\x) \mid a_{2,i}(\x)$ over $\QQ$. 
This implies that
$a_{2,i}(\x) = \kappa_i a_{1,i}(\x)$, for some $\kappa_i \in \QQ$, which implies that $l_i(\x) = 4\kappa_i a_{1,1}(\x) - a_{1,2}(\x)$. 

Inserting this into
the relation $l_i^2(\x) = -4\Delta_{2,2}(\x)$, we get that
$$
(4\kappa_ia_{1,1}(\x)-a_{1,2}(\x))^2 = -4(4a_{1,1}(\x)a_{2,2}(\x) - a_{1,2}(\x)^2).
$$
Reducing modulo the linear polynomial $a_{1,1}(\x)$ in $\ZZ[\x]$, we see that $a_{1,2}(\x)$ is divisible by $a_{1,1}(\x)$, whence we get that
$a_{1,2}(\x) = 0$, as $a_{1,2}(\x)$ is independent of $x_1$. This shows that
$l_i(\x) = 4\kappa_ia_{1,1}(\x)$, which implies that $\Delta_{2,2}(\x) = -4\kappa_i^2a_{1,1}^2(\x)$. From this we obtain the
relation $a_{2,2}(\x) = -\kappa_i^2a_{1,1}(\x)$, which completes the proof of the lemma.
\end{proof}

We are now ready to characterise all $\Psi(\x,\y)$ such that $r = 2$.
\begin{lemma}\label{lemmarank2overk}
Let $\psi_i(\y)$ be integral quadratic forms and set $\Psi(\x,\y) = \sum_{i=1}^v x_i\psi_i(\y)$. Suppose that $\rank_K \Psi(\x,\y) = 2$, 
where $K = \QQ(\x)$. Then after a $\QQ$-linear change of variables, $\Psi$ is equal to 
\begin{equation}\label{eq:option1}
\begin{split}
a_{1,1}(\x)y_1^2 + a_{1,2}(\x)y_1y_2 + a_{2,2}(\x) y_2^2,  
\end{split}
\end{equation}
or
\begin{equation}\label{eq:exps2psi}
\begin{split}
(y_1+ \kappa y_2)\left\{a_{1,1}(\x)(y_1-\kappa y_2) + \sum_{i=3}^l a_{1,i}(\x)y_i\right\},
\end{split}
\end{equation}
where $\kappa \in \QQ$, $3 \leq l \leq m$ is an integer and $a_{i,j}(\x)$ are linear forms in $\ZZ[\x]$. 
\end{lemma}
\begin{proof}
Write $\Psi(\x,\y) = \sum_{1 \leq i,j \leq m}a_{i,j}(\x)y_iy_j$ with $a_{i,j}(\x)$ linear forms in $\x$ with integer coefficients. 
As $\rank_K \Psi = 2$, equations~\eqref{eq:temp1125},
~\eqref{eq:temp1126} and~\eqref{eq:relationwithdeltas} hold. Also recall that we may assume that
$a_{i,j}(\x) = 0$ whenever $i,j \geq 3$. 

We may assume for the rest of the proof that $\Psi(\x,\y)$, considered as an integral quadratic form (in $\y$),
has at least $3$ variables. For otherwise, one can write $\Psi(\x,\y) = a_{1,1}(\x)y_1^2 + a_{1,2}(\x)y_1y_2 + a_{2,2}(\x) y_2^2$, where $a_{i,j}(\x)$ 
are linear forms, which is of the shape in~\eqref{eq:option1}. Our task then will be to show that $\Psi$ is of the shape~\eqref{eq:exps2psi}.

Suppose first that $\Delta_{2,i}(\x) = 0$ for each $i \geq 3$. Then we get 
from Lemma~\ref{lemmaslow0} and~\eqref{eq:temp1125} and~\eqref{eq:temp1126} that $\Psi(\x,\y)$ depends only on $y_1, y_2$. 
However, by the argument in the previous
paragraph, this contradicts the non-degeneracy assumption on $\Psi$. Thus there exists at least one $i \geq 3$ such that $\Delta_{2,i}(\x) \neq 0$.

Let $S \subset \left\{3, \ldots, m\right\}$ denote the set of indices $s$ such that $\Delta_{2,s}(\x) \neq 0$. Then
$S \neq \emptyset$. For ease of notation, we will assume that
$S = \left\{3, \ldots, l\right\}$ for some $3 \leq l \leq m$. Then by 
Lemmas~\ref{lemmaslow0} we get that $a_{1,j}(\x) = a_{2,j}(\x) = 0$ for each $j > l$. As $S$ is non-empty, we get 
from Lemma~\ref{lemmaslow} that $a_{1,2}(\x) = 0$.

Let $i \in S$. By Lemma~\ref{lemmaslow} we get that
 and that there exists $\kappa_i \in \QQ$ such that
$a_{2,i}(\x) = \kappa_i a_{1,i}(\x)$ and $a_{2,2}(\x) = -\kappa_i^2 a_{1,1}(\x)$. If $|S| \geq 2$, then by taking $i,j \in S$ with $i \neq j$, 
we see that $\kappa_i = \kappa_j = \kappa$, say. 

As a result, we see get that there exists $\kappa \in \QQ$ such that
\begin{equation*}
\begin{split}
\Psi(\x,\y) &= a_{1,1}(\x)y_1^2 + a_{1,2}(\x)y_2^2 + \sum_{i=3}^l a_{1,i}(\x) y_1y_i+\sum_{i=3}^l a_{2,i}(\x) y_2y_i \\
&= a_{1,1}(\x)\left(y_1^2-\kappa^2 y_2^2\right) + (y_1+\kappa y_2) \sum_{i=3}^l a_{1,i}(\x)y_i \\
&= (y_1+\kappa y_2)\left\{a_{1,1}(\x)(y_1-\kappa y_2) + \sum_{i=3}^l a_{1,i}(\x)y_i\right\},
\end{split}
\end{equation*}
which completes the proof of the lemma.
\end{proof}
\begin{proposition}\label{lemmairrfnfd}
Let $v \geq 4$. Suppose that $\Psi(\x,\y) = x_1\psi_1(\y)+\ldots+x_v\psi_v(\y)$ is irreducible over $\QQ$ and non-degenerate in the variables $x_i$. Then $\Psi$ is geometrically integral over $K = \QQ(x_1,\ldots,x_v)$. 
\end{proposition}
\begin{proof}
Let $r=\rank_K \Psi$, as before. By Lemma~\ref{quadred}, it will suffice to show that $r \geq 3$. Since $\Psi$ is irreducible over $\QQ[\x,\y]$, we may 
invoke Gauss’s lemma to conclude that $\Psi$ is irreducible over $K$, whence $r \geq 2$ and $\Psi$ is geometrically reduced. If
$r = 2$, Lemma~\ref{lemmarank2overk} shows that $\Psi$ is equal to~\eqref{eq:option1} or equal to~\eqref{eq:exps2psi}. However,
the former case is not possible as $v \geq 4$ and $\Psi$ is non-degenerate in the $x_i$. The latter case is also not possible as $\Psi$ is assumed
to be irreducible over $\QQ$. This shows that $r \geq 3$ and the proposition follows.
\end{proof}

%% file: ekedahl.tex
\subsection{Introduction}
Let $C(\x,\y)$ be as in~\eqref{eq:cubicxy} or in~\eqref{eq:cxylinearink} and let $\pi$ and $\pi'$ be the morphisms in~\eqref{eq:map} 
and~\eqref{eq:fibrationpi'} respectively with fibres $Z_\y$. In this section, we will show using the Ekedahl sieve~\cite{E91}
 that a positive proportion of fibres $Z_\y$ (a notion we will make precise below), of
$\pi$ and $\pi'$ have non-singular $\ZZ_p$ points for each prime $p$. Since the fibres of $\pi$ and $\pi'$ are cut out by the vanishing 
of equations of degree $1$ or $2$, 
the Hasse principle holds. If we assume 
that these fibres are also soluble over $\RR$, we obtain solubility over $\ZZ$. The main results are as follows. 
 
\begin{theorem}\label{thmsolubility2}
Set $\x = (x_1,\ldots,x_{n-k})$ and $\y = (y_1,\ldots,y_k)$. Let
\begin{equation}\label{eq:cxy2}
C(\x,\y):\sum_{i=1}^ky_iQ_i(\x)+\sum_{j=1}^{n-k}x_jq_j(\y) + R(\y),
\end{equation}
be an irreducible cubic form with integer coefficients, with $Q_i$ and $q_j$ quadratic forms and $R$ a 
cubic form. Suppose that $C$ has a non-singular solution over $\QQ_p$ for each prime $p$.
Let $Z$ denote the hypersurface cut out by $C = 0$
and let $\pi$ be as in~\eqref{eq:map} with fibres $Z_\y$. Assume that $\rank \sum_{i=1}^ky_iQ_i(\x)$ over the function
field $\QQ(\y)$ is at least $5$ and that the minors of order $3$ of
$M[\y]$, the matrix associated to the quadratic form $\sum_{i=1}^ky_iQ_i(\y)$, have no common factor over $\ZZ[\y]$. Let 
$\Omega_{\infty} \subset \RR^{n-k}$ be a compact set.

Then there exists an integer $M$ such that for each $Y \geq 1$ there exists a set $\mathcal{C}_k(Y) \subset Y\Omega_{\infty} \cap \ZZ^k$ 
such that the following statements hold.
\begin{enumerate}
\item Let $\y \in \mathcal{C}_k(Y)$. If $p \mid M$, there exists an integer $v_p = v_p(C)$ that depends only on $C$ such that there
exists $\x \bmod{p^{2v_p-1}}$ with the property that $C(\x,\y) \equiv 0 \bmod{p^{2v_p-1}}$
and $\frac{\partial C(\x,\y)}{\partial x_i} \not\equiv 0 \bmod{p^v}$ for some $1 \leq i \leq n-k$. If $p \nmid M$, we have
\begin{equation}\label{eq:436315}
\begin{split}
\#\left\{\x \mod{p} : C(\x,\y) \right.&\left.\equiv 0 \bmod{p}, \x \text{ non-singular} \right\} \\
&= p^{n-k-1} + O(p^{n-k-2}),
\end{split}
\end{equation}
where the implied constant is independent of $p$. 
\item If $\y \in \mathcal{C}_k(Y)$, the fibre $Z_\y$ over $\y$ has a smooth point over $\ZZ_p$ for each prime $p$.
\item There exists a constant $\sigma = \sigma(C, \Omega_{\infty}) > 0$ such that 
\begin{align*}
\lim_{Y \to \infty} \frac{\#\mathcal{C}_k(Y)}{Y^{k}} \to \sigma.
\end{align*}
\end{enumerate}
In particular, if $Y$ is large enough, then we have 
\begin{equation*}
\mathcal{C}_k(Y) \gg Y^k.
\end{equation*}
\end{theorem}

\begin{theorem}\label{thmsolubility}
Set $\x = (x_1,\ldots,x_{n-k})$ and $\y = (y_1,\ldots,y_k)$. Let 
\begin{equation}\label{eq:cxy}
C(\x,\y):x_1Q_1(\y)+\ldots+x_{n-k}Q_{n-k}(\y)+R(\y)
\end{equation}
be an irreducible cubic form with integer coefficients, with $Q_i(\y)$ quadratic forms and $R(\y)$ a cubic form. Assume that $Q_1(\y),\ldots,Q_{n-k}(\y)$
have no common factors. Suppose that $k, n-k \geq 2$, that $C$ is irreducible and that $C$ has a non-singular zero over 
every completion $\QQ_p$. Let $Z$ denote the hypersurface cut out by $C = 0$
and let $\pi'$ be as in~\eqref{eq:fibrationpi'}.  Let $\Omega_{\infty} \subset \RR^k$ be a compact set such that for any $Y\geq 1$, we have that the fibre $Z_\y$ of $\pi'$ over
$\y$ has a real point for any 
$\y \in Y\Omega_{\infty}$.

Then for each $Y \geq 1$ there exists a set $\mathcal{C}_k(Y) \subset Y\Omega_{\infty} \cap \ZZ^k$ such that the following statements hold.
\begin{enumerate}
\item For any $\y \in \mathcal{C}_k(Y)$, the fibre $Z_\y$ of $\pi'$ over $\y$ has a point over $\ZZ$ and 
$\gcd\, (Q_1(\y),\ldots,Q_5(\y)) \ll_C 1.$
\item There exists a constant $\sigma = \sigma(C, \Omega_{\infty}) > 0$ such that 
\begin{align*}
\lim_{Y \to \infty} \frac{\#\mathcal{C}_k(Y)}{Y^{k}} \to \sigma.
\end{align*}
\end{enumerate}
In particular, if $Y$ is large enough, then we have 
\begin{equation*}\label{eq:corlb}
\mathcal{C}_k(Y) \gg Y^k.
\end{equation*}
\end{theorem}

\subsection{Proof of Theorem~\ref{thmsolubility2}}
Let $\phi_1,\ldots,\phi_J$ denote the minors of order $3$ of $M[\y]$. Define the variety 
\begin{equation}\label{eq:444315}
\mathcal{L} \subset \AA^k : \phi_1(\y) = \ldots = \phi_J(\y) = 0.
\end{equation}
Then by hypothesis, $\mathcal{L}$ is codimension at least $2$ in $\AA^k$. We will now
define certain subsets $\Omega_p \subset \ZZ_p^k$ for each prime $p$. 

Let $M'$ denote the product of the primes that divide all the coefficients of each of the $\phi_i$ in~\eqref{eq:444315}. Set $M = 2M'$. 
Lemma~\ref{lemma0} shows that for each $p \mid M$ there exists
an integer $v_p \geq 1$ and 
$\r \bmod{p^{2v_p-1}}$ such that the equation 
$C(\x,\r) \equiv 0 \bmod{p^{2v_p-1}}$ is soluble with the proviso that $\frac{\partial C(\x,\r)}{\partial x_i} \not \equiv 0 \bmod{p^v}$
for some $1 \leq i \leq n-k.$ 
Define for $p \mid M$
\begin{align*}
\Omega_p = \left\{\y \in \ZZ_p^k : \y \equiv \r \bmod{p^{2v_p-1}}\right\}.
\end{align*}
As $\Omega_p$ is defined by a congruence condition, it follows that $\mu_p(\Omega_p) > 0$, where
$\mu_p$ is the normalised Haar measure on $\ZZ_p.$ If $p \nmid M$, put
$$
\Omega_p = \left\{\y \in \ZZ_p^k : \y \bmod{p} \not\in \mathcal{L}(\FF_p)\right\},
$$
with $\mathcal{L}$ as in~\eqref{eq:444315}.
Observe that if $\y \in \Omega_p$, then 
$\y + p\ZZ_p^k \subset \Omega_p$, whence $\mu_p(\Omega_p) > 0$, where $\mu_p$ is the normalised Haar measure of $\ZZ_p$.

Set 
$$
\mathcal{C}_k(Y) = \left\{\y \in \ZZ^k \cap Y\Omega_{\infty}: \y \in \Omega_p\text{ for all primes }p \right\},
$$
where $\Omega_{\infty}$ is as in the statement of the theorem. 
Since $\mathcal{L}$ is of codimension at least $2$, we see that~\cite[Equation (3.4)]{BBL16} holds. We have also
verified that $\mu_p(\Omega_p) > 0$ for each prime $p$. Thus the hypotheses of~\cite[Proposition 3.2]{BBL16} are all satisfied, whence
\begin{align*}
\lim_{Y \to \infty} \frac{\#\mathcal{C}_k(Y)}{Y^{k}} \to \sigma,
\end{align*}
for some $\sigma = \sigma(\Omega_\infty,C) > 0$. This proves the third statement of the theorem. We will now prove
the first two statements.

Let $\y \in \mathcal{C}_k(Y)$. If $p \mid M$, we get from the construction of $\Omega_p$ that $C(\x,\y) \equiv 0 \bmod{p^{2v_p-1}}$
and $\frac{\partial C(\x,\y)}{\partial x_i} \not\equiv 0 \bmod{p^v}$ for some $1 \leq i \leq n-k$. For $p \nmid M$, the
fact that $\y \bmod{p} \in \Omega_p$ implies that 
$\rank_{\FF_p} M[\y] \geq 3$. Therefore, the estimate~\eqref{eq:436315}
follows from an application of Lemma~\ref{lemmasolvequadricsmodp}. This proves the first statement of the theorem. 
The second statement follows from invoking Hensel's lemma, and the theorem follows.

\begin{remark}\label{remgcd}
Let $\mathcal{C}_k(Y)$ be as in the statement of the preceding theorem and let $\y = (y_1,\ldots,y_k) \in \mathcal{C}_k(Y)$. 
Set $g = \gcd(y_1,\ldots,y_k)$. Then observe that for $p \mid M$, the $p$-adic valuation, $v_p(g) \leq 2v_p-1$, where $v_p$ is
as in the definition of $\Omega_p$. In addition, we also have $p \nmid g$ for each prime $p \nmid M$, by construction of $\Omega_p$. 
Thus we get that $\gcd(y_1,\ldots,y_h) \ll_C 1$ for any $\y \in \mathcal{C}_k(Y)$.
\end{remark}
\subsection{Proof of Theorem~\ref{thmsolubility}}
In order to apply Ekedahl's sieve to study the solubility of~\eqref{eq:cxy}, we will require the following lemma.

\begin{lemma}\label{lemma2}
Let $k$. Suppose that $Q_1(\y),\ldots,Q_{n-k}(\y)$ have no common factors. Let 
$\mathcal{L} \subset \AA^k : Q_1(\y) = \ldots = Q_{n-k}(\y) = 0$. Suppose that $Q_i(\y)$ have no common linear factors. Suppose that $\y \in \ZZ_p^k$ such that $\y \pmod{p} \not\in \mathcal{L}(\FF_p)$. Then the fibre $Z_\y$ of $\pi'$ over $\y$
has a smooth point over $\ZZ_p$.
\end{lemma}
\begin{proof}
Under the hypothesis of the lemma, it follows that $Q_i(\y) \not\equiv 0 \pmod{p}$ for some index $i$. Thus we get a solution $\x \pmod{p}$ to the linear equation
$C(\x,\y) \equiv 0 \pmod{p}$. This is non-singular, by definition, which we can lift to a solution in $\ZZ_p$ by Hensel's lemma.
\end{proof}

\begin{proof}[Proof of Theorem~\ref{thmsolubility}]
As in the proof of Theorem~\ref{thmsolubility2}, we begin by defining certain subsets $\Omega_p$ for each prime $p$. We will show for each prime $p$ that if $\y \in \Omega_p$, then 
$Z_\y$ has a smooth point over $\ZZ_p$. Recall that the fibres 
$Z_\y$ of $\pi'$ are cut out by the vanishing of linear polynomials, so $Z_\y$ has a point over $\ZZ$ if it is everywhere locally soluble.

Let $g$ denote the greatest common factor of the coefficients of the quadratic forms
$Q_i(\y)$. Then $g\ll_C 1$. Let $p$ be a prime. If $p \mid g$,
set
$$
\Omega_p = \left\{\y \in \ZZ_p^k: \y \equiv \r \bmod{p^{2v-1}}\right\},
$$
where $\r$ and $p$ are as in the statement of Lemma~\ref{lemma0}. As $\Omega_p$ is defined by a congruence condition, it is clear that $\mu_p(\Omega_p) > 0$. 

If $p \nmid g$, define 
$$
\Omega_p = \left\{\y \in \ZZ_p^k : \y \mod{p} \not\in \mathcal{L}(\FF_p)\right\},
$$
where $\mathcal{L}$ is the variety defined in Lemma~\ref{lemma2}. Note that $\mathcal{L}$ has codimension at least $2$ since $n-k \geq 2$, by
assumption. Since the quadratic forms $Q_1(\y),\ldots,Q_m(\y)$ are coprime, it is clear that $\Omega_p \neq \emptyset$. Indeed, if $\y \in \Omega_p$, then 
$\y + p\ZZ_p^k \subset \Omega_p$, whence $\mu_p(\Omega_p) > 0$. 

Set 
$$
\mathcal{C}_k(Y) = \left\{\y \in \ZZ^k \cap Y\Omega_{\infty}: \y \in \Omega_p\text{ for all primes }p \right\},
$$
where $\Omega_{\infty}$ is as in the statement of the theorem. Observe that if $\y \in \mathcal{C}_k(Y)$, then 
$$
\gcd(Q_1(\y),\ldots,Q_m(\y)) \ll_C 1.
$$ This proves the first statement of Theorem~\ref{thmsolubility}. Since $\mathcal{L}$ is of codimension at least $2$, the second statement readily follows from~\cite[Proposition 3.2]{BBL16}. This completes the proof.
\end{proof}


%% file: linear.tex
Let $\a = (a_1,\ldots,a_n) \in \ZZ_{\prim}^n$ i.e., that $\gcd(a_1,\ldots,a_n)=1$ and let $b \in \ZZ$. Set
\begin{equation*}
N(\a,b,B) = \#\left\{(\|\x\|_2^2+1)^{1/2} \leq B : \sum_{i=1}^n a_ix_i + b = 0\right\}.
\end{equation*}
In this section, we will be obtain an asymptotic formula for $N(\a,b,B)$ that is uniform in the coefficients $a_i$ and $b$. The main result is the following proposition, which we will prove by appealing to work of Thunder~\cite{thunder}.
\begin{proposition}\label{proplineq}
Let $\a \in \ZZ_{\prim}^n$ and let $b \in \ZZ$. Suppose that $B\geq (|b|/\|\a\|_2)^{\frac{1}{1-\eta}}$, for some $\eta > 0$. Let $\lambda_1$ denote the length of the smallest non-zero vector in the lattice 
$$
\Lambda_{\a} = \left\{\x \in \ZZ^n: \langle \a, \x \rangle = 0\right\}.
$$
Then there exists a constant $c = c(n) > 0$, independent of $\a$ and $b$, such that
$$
N(\a,b,B) = \frac{cB^{n-1}}{\|\a\|_2} + O\left(\frac{B^{n-1-\eta}}{\|\a\|_2}\right) + O\left(\sum_{j=0}^{n-2}\frac{B^j}{\lambda_1^j}\right),
$$
and the implicit constants in the error terms depend only on $n$.
\end{proposition}

We will also need the following straightforward generalisation of Proposition~\ref{proplineq}. For an integer $g$ set
\begin{equation*}
N_g(\a,b,B) = \#\left\{(\|\x\|_2^2+1)^{1/2} \leq B : (\x,g) = 1 \mbox{ and }\sum_{i=1}^n a_ix_i + b = 0\right\}.
\end{equation*}
Then we have 
\begin{corollary}\label{corlineq}
Let $\a \in \ZZ_{\prim}^n$ and let $b \in \ZZ$. Suppose that $B\geq (|b|/\|\a\|_2)^{\frac{1}{1-\eta}}$, for some $\eta > 0$. Let $\lambda_1$ denote the length of the smallest non-zero vector in the lattice 
$$
\Lambda_{\a} = \left\{\x \in \ZZ^n: \langle \a, \x \rangle = 0\right\}.
$$
Then there exists a constant $c = c(n) > 0$, independent of $\a$, $b$ and $g$, such that
\begin{align*}
N_g(\a,b,B) &= \prod_{p \mid (b,g)}\left(1-\frac{1}{p^{n-1}}\right)\frac{cB^{n-1}}{\|\a\|_2} +  O\left(d(g)\frac{B^{n-1-\eta}}{\|\a\|_2}\right) \\
&\quad \quad + O\left(d(g)\sum_{j=0}^{n-2}\frac{B^j}{\lambda_1^j}\right),
\end{align*}
where $d(g)$ is the number of divisors of $g$ and the implicit constants in the error terms depend only on $n$.
\end{corollary}
Observe that Corollary~\ref{corlineq} follows from Proposition~\ref{proplineq} since
\begin{align*}
N_g(\a,b,B) = \sum_{d \mid (b,g)}\mu(d)N(\a,b/d,\sqrt{B^2/d^2-1})
\end{align*}
and the fact that $\sqrt{B^2/d^2-1} = B/d + O(d/B).$
\subsection{Preliminaries}

Define
$$
S = \left\{\x \in \QQ^{n+1}: \sum_{i=1}^na_ix_i +bx_{n+1} = 0\right\}.
$$ 
Then $S$ is an $n$-dimensional subspace of $\QQ^{n+1}$ and $S \not\subset \QQ^n$. 
Let $V = S \cap \QQ^n$, whence
 $V$ is an $n-1$ dimensional subspace of $S$ and there exists $\bs{\beta} \in \QQ^n$ such that $S = \QQ(\boldsymbol{\beta},1) \oplus V$. Clearly we must have $\langle\a,\bs{\beta}\rangle + b = 0$ and 
$$
V = \left\{\x \in \QQ^n: \a.\x = 0\right\}.
$$
Let $V_{\RR}^\perp$ denote the orthogonal complement of $V_{\RR}$. Let $\e_1,\ldots,\e_{n-1}$ be a basis for $V$. 
Then $V_{\RR}^{\perp}$ is spanned by $\a$. 
Moreover, $\e_1,\ldots,\e_{n-1},\a$ form a basis for $\RR^n$. 
Thus we may write $\bs{\beta} = \sum_{i=1}^{n-1}\lambda_i\e_i+ \lambda_n\a$ for some $\lambda_i \in \RR$. Let $\pi: \RR^n \to V_{\RR}^{\perp}$ denote the projection map onto $V_{\RR}^{\perp}$. 
Then we have $\pi(\bs{\beta}) = \lambda_n\a$. 
However, since $\a.\beta = - b = \lambda_n\|\a\|_2^2$, we get that 
$$\pi(\bs{\beta}) = -b\hat{\a}/\|\a\|_2,$$ with $\hat{\a} = \a/\|\a\|_2$. 

\subsection{Counting lattice points in certain domains}
Set $I(V) = V \cap \ZZ^{n}$. Then $I(V) = \Lambda_{\a}$.
We may assume without loss of generality that $\bs{\beta} \in \ZZ^n.$ 
Define
\begin{equation*}
\lambda(\ZZ,B) = \#\left\{\x \in I(V) : (\|\x + \bs{\beta}\|_2^2 + 1)^{1/2} \leq B  \right\}.
\end{equation*}
Then we clearly have 
\begin{equation}\label{eqlzbnab}
\lambda(\ZZ,B) = N(\a,b,B).
\end{equation}
Write $\bs{\beta} = \bs{\beta'} + \pi(\bs{\beta})$, with $\bs{\beta'} \in V_{\RR}$. 
Then for $\x \in I(V)$ we have
\begin{equation*}\label{eq:translate}
\begin{split}
\langle \x + \bs{\beta}, \x + \bs{\beta} \rangle &= \langle \x + \bs{\beta'}, \x +\bs{\beta'} \rangle + \langle \pi(\bs{\beta}), \pi(\bs{\beta})\rangle \\
&= \langle \x + \bs{\beta'}, \x +\bs{\beta'} \rangle + \frac{b^2}{\|\a\|_2^2}.
\end{split}
\end{equation*}
As a result, we have
\begin{equation*}
\lambda(\ZZ,B) = \#\left\{\x \in I(V) : \left(\|\x + \bs{\beta'}\|_2^2 + \frac{b^2}{\|\a\|_2^2}\right)^{1/2} \leq B  \right\}.
\end{equation*}
This shows that $\lambda(\ZZ,B)$ is the number of lattice points in a domain that is a translate of a ball of radius $B$.

Now, $I(V)$ is an $n-1$ dimensional lattice in $\RR^n.$ 
Therefore there exists a linear transformation 
$P: \RR^n \to \RR^{n-1}$ such that $P$ induces an isomorphism $V_{\mathbf{\RR}} \simeq \RR^{n-1}$ and $P(V_{\RR}^{\perp}) = \mathbf{0}$. In particular, $P(I(V))$ is a lattice of full rank in $\RR^{n-1}$. Furthermore, $\|P(\x)\|_2 = \|\x\|_2$ for any $\x \in V_{\RR}$. Let $P(\bs{\beta'}) = \bs{\beta''},$ for some $\bs{\beta''} \in \RR^{n-1}$. 
Set $\frac{b}{\|\a\|_2} = a.$ 
Then 
\begin{equation}\label{eq:translate1}
\lambda(\ZZ,B) = \#\left\{\x \in P(I(V)) : (\|\x + \bs{\beta''}\|_2^2 + a^2)^{1/2} \leq B\right\}.
\end{equation}

\subsection{Volume computations}
Set
\begin{equation}\label{eq:regiondb}
D_n(B) = \left\{ \x \in \RR^n : (\|\x\|_2^2 + a^2)^{1/2} \leq B\right\}
\end{equation}
and
\begin{equation}\label{eq:regiondb1}
\widetilde{D}_n(B) =  \left\{ \x \in \RR^n : (\|\x + \bs{\beta''}\|_2^2 + a^2)^{1/2} \leq B\right\}.
\end{equation}
Since translations preserve volume, we have $\vol \widetilde{D}_n(B) = \vol D_n(B)$.
\begin{lemma}\label{lemmavol}
Let $n \geq 2$. Then there exists a constant $C(n) > 0$ such that
\begin{equation*}
\vol D_{n}(B) = C(n) \int_{(u^2+a^2)^{1/2} \leq B} u^{n-1}\, du.
\end{equation*}
\end{lemma}

\begin{proof}
Let $V_n(B)$ denote the volume in question. Suppose first that $n = 2$. Then by switching to polar coordinates we get
\begin{align*}
V_2(B) &= \pi \int_{(u^2 + a^2)^{1/2} \leq B} u\, du. 
\end{align*}
Suppose next that $n \geq 3$. 
Then we proceed recursively in the following manner. 
By Fubini, we write
\begin{align*}
V_n(B) &= \int_{(x_n^2+ x_{n-1}^2 + \sum_{i=1}^{n-2} x_i^2 + a^2)^{1/2} \leq B} dx_{n-1} \, dx_n \, \prod_{i=1}^{n-2} dx_i \\
\end{align*}
Switching to polar coordinates we get
\begin{align*}
V_n(B) &= \int_{-\pi/2}^{\pi/2} d\theta \int_{(u^2 + \sum_{i=1}^{n-2} x_i^2 + a^2)^{1/2} \leq B} u\,  du \, \prod_{i=1}^{n-2} dx_i \\
&= \left(\int_{-\pi/2}^{\pi/2} d\theta \right)\left(\int_{-\pi/2}^{\pi/2} \cos \theta \, d\theta\right)\times \\
&\quad \quad  \int_{(u^2 + \sum_{i=1}^{n-3} x_i^2 +a^2)^{1/2} \leq B} u^2 \, du \, \prod_{i=1}^{n-3} dx_i.
\end{align*}
Proceeding in this fashion, after $n-2$ steps we get that
\begin{align*}
V_n(B) &= \prod_{j=0}^{n-2}\int_{-\pi/2}^{\pi/2} \cos^j \theta \, d\theta \int_{(u^2+a^2)^{1/2} \leq B} u^{n-1}\, du  \\
&=  \prod_{j = 0}^{n-2} \frac{\Gamma(\frac{1}{2})\Gamma(\frac{j+1}{2})}{\Gamma(\frac{j}{2}+1)} \int_{(u^2+a^2)^{1/2} \leq B} u^{n-1}\, du,
\end{align*}
by~\cite[Equation 3.621.5]{Grad}, for example. Taking 
$$
C(n) = \prod_{j = 0}^{n-2} \frac{\Gamma(\frac{1}{2})\Gamma(\frac{j+1}{2})}{\Gamma(\frac{j}{2}+1)} = \frac{\pi^{\frac{n-1}{2}}}{\Gamma(\frac{n}{2})}
$$
completes the proof.
\end{proof}

\begin{lemma}\label{lemmavolcomp}
For $1 \leq l \leq n$ let $V(l,B)$ (resp. $\widetilde{V}(l,B)$) denote the sum of the volumes of the regions $D_l(B)$ defined in~\eqref{eq:regiondb} (resp. $\widetilde{D}_l(B)$ defined in~\eqref{eq:regiondb1}) projected onto $\RR^{l}$ by setting $n-l$ coordinates to be zero. Set $V(0,B) = \widetilde{V}(0,B) = 1$.
Suppose that $|a| \leq B^{1-\eta}$ for some $\eta > 0$. 
Then there exists constants $c(l,n) \neq 0$ such that 
$$
\widetilde{V}(l,B) = V(l,B) = B^{l}(c(l,n) + O(B^{-\eta})).
$$
\end{lemma}

\begin{proof}
If $l = 0$, then there is nothing to prove. So we may suppose that $1 \leq l \leq n$. A calculation similar to the one in the proof of Lemma~\ref{lemmavol} shows that there exist constants $c(l,n)$ such that
$$
V(l,B) = c(l,n)\int_{(u^2+a^2)^{1/2} \leq B}u^{l-1} \, du.
$$
Set $u = aw$. Then we have 
\begin{align*}
V(l,B) &= c(l,n)a^{l}\int_{(w^2+1)^{1/2} \leq B/a}w^{l-1} \, dw.
\end{align*}
Next, make the change of variables $w^2+1 = x$. Then we get
\begin{align*}
V(l,B) &= \frac{c(l,n)}{2}a^{l}\int_{1\leq x \leq B^2/a^2}(x-1)^{\frac{l-2}{2}}\, dx\\
&= \frac{a^{l}c(l,n)}{2} \int_{0}^{B^2/a^2-1} x^{\frac{l-2}{2}} \, dx \\
&= \frac{a^{l}c(l,n)}{2l}(B^2/a^2-1)^{\frac{l}{2}} \\
&= B^{l}(c(l,n) + O(a^2/B^2)) = B^{l}(c(n,l) + O(B^{-\eta}),
\end{align*}
say, since $|a| \leq B^{1-\eta}.$ This completes the proof of the lemma.
\end{proof}

\subsection{Proof of Proposition~\ref{proplineq}}

For $1 \leq i < n-1$ let 
\begin{align*}
\d_i = \min_{\substack{L \subset I(V) \\ \rank L = i}} \det L
\end{align*}  
and set $\d_0 = 1$.
By~\eqref{eqlzbnab} and~\eqref{eq:translate1} and by invoking~\cite[Theorem 5]{thunder} we get that 
$$
N(\a,b,B) = \frac{V(4,B)}{\|\a\|_2} + O\left(\sum_{l=0}^{n-2} \frac{V(l,B)}{\d_{l}}\right),
$$
where $V(l,B)$ are as in the statement of Lemma~\ref{lemmavolcomp}.

Let $\lambda_1 \leq \lambda_2 \leq \ldots \leq \lambda_{n-1}$ denote the successive minima of the lattice $I(V)$ with respect to the unit ball. Then by Minkowski's second theorem, for $0 \leq l \leq n-2$, it follows that $\d_l \gg \lambda_1\ldots\lambda_l \geq \lambda_1^l$. Proposition~\ref{proplineq} now follows from Lemma~\ref{lemmavolcomp}.


%% file: mainterm.tex
This section is devoted to the proof of Theorems~\ref{thma1} and~\ref{thma2}. Let 
$
C(\x,\y)
$ be a cubic form in $n$ variables with $h(C) = h,$ as in~\eqref{eq:cubicxy}. Let $Z \subset \AA^n$ be the cubic hypersurface $C(\x,\y) = 0$. 
Let $\pi$ be as in~\eqref{eq:map}. Then the fibre over a point $\y$ is the
affine quadric
$
Z_\y: F_\y(\x) =  0,
$
where 
\begin{equation}\label{eq:deffyx}
F_\y(\x) = \sum_{i=1}^hy_iF_i(\x) + \sum_{j=1}^{n-h}x_j q_j(\y) + R(\y).
\end{equation}
We will denote the quadratic part of $F_\y$ by 
\begin{equation}\label{eq:defqyx}
Q_\y(\x) =  y_1F_1(\x) + \ldots + y_hF_h(\x).
\end{equation}
Let $M[\y]$ be the matrix associated with the quadratic form $Q_\y(\x)$. Let $r = r(C)$ denote the rank 
of the fibration $\pi$ (see Definition~\ref{genericrank}). 

Let 
\begin{equation}\label{eq:nyqyb}
N_\y(B) = \#\left\{|\x| \leq B: (\x,\y) = 1 \text{ and }F_\y(\x) = 0\right\}.
\end{equation}
Then for any $1 \leq Y \leq B$, we have
$$
N(B) \geq \sum_{\substack{\y \in \ZZ^h \\ |\y| \leq Y}}N_\y(B).
$$
To prove Theorems~\ref{thma1} and~\ref{thma2}, we will show that there are `many' $\y$ so that
the equation $F_\y = 0$ is soluble, and for each such $\y$ we
will use the circle method to obtain a lower bound for $N_\y(B)$. While studying the solubility of the equation $F_\y = 0$,
we will encounter two distinct classes of quadratic forms $Q_\y(\x)$: forms that can be written as $Q_\y(\x) = l(\y)F(\x)$, with $l$ 
linear and $F$ a non-singular definite quadratic form in $r(C)$ variables, and those that cannot be written in this form. 

Observe 
that if $Q_\y(\x) = l(\y)F(x_1,\ldots,x_r)$, then  
\begin{equation}\label{eq:myhyp1}
M[\y] = l(\y)N \quad \quad \text{ with } \quad \quad N = \begin{pmatrix} N_1 & \mathbf{0}_{r\times n-r} \\
\mathbf{0}_{n-r\times r} & \mathbf{0}_{n-r \times n-r}
\end{pmatrix},
\end{equation}
with $N_1$ a symmetric definite matrix of order $r$ and $\mathbf{0}_{u \times v} \in \mat_{u,v}$ is the zero matrix. If this the case, we will say that $C$ satisfies 
\begin{hypothesis}\label{h1}
Let $F_\y(\x)$ and $Q_\y(\x)$ be as in~\eqref{eq:deffyx} and~\eqref{eq:defqyx}. 
There exists a linear form $l(\y)$ and a definite quadratic form $F(\x)$ of rank $r = r(C)$ such that after a linear change of variables, 
$Q_\y(x_1,\ldots,x_{n-h}) = l(\y)F(x_1,\ldots,x_r)$, or equivalently, that there exists a symmetric definite matrix $N_1$ of order $r$ such that~\eqref{eq:myhyp1} holds. 
\end{hypothesis}
Note that if $C$ satisfies Hypothesis~\ref{h1}, then for each fixed $\y$ the quadratic form $Q_\y(\x)$ is definite. 
\subsection{Proof of Theorem~\ref{thma1}}\label{proofthma1}
To prove the theorem, we may assume that $h \leq 13$, for otherwise, the result is superseded by Theorem~\ref{proph14intro}. We have
by assumption that $r \geq \min\left\{n-h-4, 5\right\}$. By the hypothesis of Theorem~\ref{thma1}, $C$ satisfies Hypothesis~\ref{h1}. We may therefore assume without loss of generality that $Q_\y(\x) = y_1F(x_1,\ldots,x_r)$ with $r \geq 5$. Thus we have
\begin{equation}\label{eq:cdeff}
C(\x,\y) = y_1F(x_1,\ldots,x_r) + \sum_{j=1}^{n-h}x_jq_j(\y) + R(\y).
\end{equation}
We will assume that $F$ is positive-definite, the negative-definite case can be handled similarly. We will prove the following propositions, from
which Theorem~\ref{thma1} easily follows, since $n-h-4\leq r \leq n-h$, by assumption.
\begin{proposition}\label{propthma11}
Let $C$ be as in~\eqref{eq:cdeff} with $F$ positive definite. Assume that
$y_1 \mid q_i(\y)$ for each $1 \leq i \leq r$.
Then we have
$$
N(B) \gg B^{r-2+\frac{2}{3}(n-r-1)}.
$$
\end{proposition}

\begin{proposition}\label{propthma12}
Let $C$ be as in~\eqref{eq:cdeff} with $F$ positive definite. Assume that
$y_1 \nmid q_i(\y)$ for some $1 \leq i \leq r$.
Then we have
$$
N(B) \gg B^{n-h-2+\frac{h-1}{2}}.
$$

\end{proposition}

Let $M$ denote the matrix associated to $F$ and let $\Delta$ denote the discriminant of $F(\x)$. 
The key tool we will use to count solutions to~\eqref{eq:cdeff} is the following result. 
\begin{proposition}\label{propcoprime}
Let $F$ be a non-singular, positive-definite quadratic form of rank $r$ and determinant $\Delta$. Let $\bs{\xi}\in \ZZ^{r}$ and let $P^2 = N$ be an integer with
$P$ large. Suppose that $F(\bs{\xi})-N \equiv 0 \pmod{2\Delta}$. Assume that $|\bs{\xi}| \ll P$. Let $w(\x)$ be a smooth, positive bump function around a real solution to $F(\x+\bs{\xi}/P)=1$. Define 
\begin{align*}
M(F,N) = \sum_{\substack{(\x,2\Delta)=1 \\ F(\x+\bs{\xi}) = N}}w(P^{-1}\x).
\end{align*}
Then there exist constants $c = c_{F,N} > 0$ that depends only on $F$ and $N$, and $\delta > 0$ such that
$$
M(F,N) = c\prod_{p \mid 2\Delta}\left(1-\frac{1}{p^{r}}\right)P^{r-2}+O(P^{r-2-\delta}).
$$
Moreover, we have the bound $c = c_{F,N} \gg 1$, where the implied constant depends only on the coefficients of $F$.
\end{proposition}
In order to prove the proposition, we will need the following lemma.
\begin{lemma}\label{lemmacoprime}
Let $F$ be a non-singular positive definite quadratic form of rank $r$ with determinant $\Delta$. Let $d$ be a squarefree integer such that $d \mid 2\Delta$. Let $\bs{\xi}\in \ZZ^{r}$ and let $P^2 = N$ be an integer with
$P$ large. Suppose that $F(\bs{\xi})-N \equiv 0 \pmod{2\Delta}$. Assume that $|\bs{\xi}| \ll P$. Assume that $|\bs{\xi}| \ll P$. Set
$$
M_d(F,N) = \sum_{\substack{F(\x) = N \\ \x \equiv \bs{\xi} \mod{d}}}w\left(\frac{\x - \bs{\xi}}{P}\right).
$$
Then there exist constants $c = c_{F,N} > 0$ that depends only on $F$ and $N$ and $\delta > 0$ that are independent of $d$ such that
$$
M_d(F,N) = \frac{cP^{r-2} }{d^{r}}+ O(P^{r-2-\delta}).
$$
Moreover, we have the bound $c \gg 1$, where the implied constant depends only on the coefficients of $F$.
\end{lemma}
\begin{proof}
The proof of the lemma is a straightforward adaptation of~\cite[Theorem 4]{HB96}, so we will only give a brief sketch. We begin by remarking that since all constants in our work are allowed to depend on $C$, we have that $d \ll \Delta \ll 1$. Using the smooth $\delta$-function~\cite[Theorem 1]{HB96}, we may write
\begin{equation}
M_d(F,N) = \frac{c_P}{d^{r}}\sum_{\cc \in \ZZ^{r}}\sum_{q \ll P}q^{-(n-h)}S_{d,q}(\cc)I_q(\cc),
\end{equation}
where $c_P = 1 + O_J(P^{-J})$, for any integer $J \geq 1$,
\begin{align*}
S_{d,q}(\cc) = \sumstar_{a \mod{q}}\sum_{\substack{\b \mod{dq} \\ \b \equiv \bs{\xi} \mod{d}}}e_q(a(F(\b)-N))e_{dq}(\b.\cc),
\end{align*}
and
\begin{align*}
I_q(\cc) = \int w\left(\frac{\x - \bs{\xi}}{P}\right)\hf{\frac{F(\x)-N}{N}}e_{dq}(-\cc.\x)\,d\x,
\end{align*}
with $r = q/P$. The main difference between the exponential sum $S_q(\cc)$ defined above and the analogous sum in Heath-Brown's work is the appearance
of the additional congruence condition modulo $d$. In place of~\cite[Lemma 23]{HB96}, we have the relation
\begin{equation}\label{eq:twistmult}
S_{d,q}(\cc) = S_{1,q_1}(\overline{q_2}\cc)S_{d,q_2}(\overline{q_1}\cc),
\end{equation}
where $q = q_1q_2$ with $(q_1,d) = 1$ and $q_2 \mid d^{\infty}$ and $\overline{q_2}$ (resp. $\overline{q_1}$) denotes the multiplicative inverse of $q_2 \pmod{q_1}$ (resp. $q_1 \pmod{q_2}$). Using~\eqref{eq:twistmult} and by arguing as in the proof of~\cite[Lemma 28]{HB96},
we have for $|\cc| \ll P$ that
$$
\sum_{q \leq X}S_q(\cc) \ll_{\ve} X^{\frac{3+r}{2}+\ve}P^{\ve}.
$$
By~\cite[Lemma 19]{HB96}, it follows that $I_q(\cc) \ll P^{-A}$ unless $|\cc| \ll P^{\ve}$. For $\cc$ in this range,~\cite[Lemma 22]{HB96} ensures that for any $\ve > 0$ we have
$$
I_q(\cc) \ll_{\ve} P^n\left(\frac{P^2|\cc|}{q^2}\right)^{\ve}\left(\frac{P|\cc|}{q}\right)^{1-\frac{r}{2}}.
$$
Combining these estimates, we have that
$$
\frac{c_P}{d^{r}}\sum_{\cc \neq \mathbf{0}}\sum_{q \ll P}q^{-r}S_q(\cc)I_q(\cc) \ll_{\ve} P^{\frac{r}{2}+\ve}.
$$
All that remains is to evaluate the main term
$$
\frac{c_P}{d^{r}}\sum_{q \ll P}q^{-r}S_q(\mathbf{0})I_q(\mathbf{0}).
$$
Observe that~\eqref{eq:twistmult} ensures that
$$
S_q(\mathbf{0}) = \sumstar_{a \mod{q}}\sum_{\b \mod{q}}e_q(a(F(\b)-N)),
$$
which is identical to the exponential sum in the proof of~\cite[Theorem 4]{HB96}. Also,
$$
I_q(\mathbf{0}) = P^r\int w(\x)\hf{F(\x+\bs{\xi}/P)-1}\, d\x.
$$
By definition of the function $w(\x)$ it follows from~\cite[Theorem 3 and Lemma 13]{HB96} that $I_q(\mathbf{0}) = P^r (c_{\infty} + O_N((q/P)^N))$,
where
$$
c_{\infty} = \lim_{\ve \to 0}(2\ve)^{-1}\int_{|F(\x + \bs{\xi}/P)-1| < \ve} w(\x)\, d\x > 0.
$$ 
Note that although $c_{\infty}$ has 
mild dependence on $P$, it is clear that $c_{\infty} \gg 1$ uniformly in $P$. Proceeding as in the evaluation of the main term in~\cite[Theorem 4]{HB96}, we get that
$$
\frac{c_P}{d^{r}}\sum_{q \ll P}q^{-r}S_q(\mathbf{0})I_q(\mathbf{0}) = \frac{c_{F,N}c_{\infty}P^{r-2}}{d^{r}} + O_{\ve}(P^{\frac{r}{2}+\ve}),
$$
with $c_{F,N} = \sum_{q=1}^{\infty}q^{-r}S_q(\mathbf{0})$. Note that $c_{F,N} > 0$ since $\rank F \geq 5$. Setting $c=c_{\infty}c_{F,N}$ we 
obtain the asymptotic formula in the statement of the lemma. The lower bound $c \gg_F 1$ follows from~\cite[Proposition 2]{BD08}, since 
$r = \rank F \geq 5$. This completes the proof of the lemma.
\end{proof}

\begin{proof}[Proof of Proposition~\ref{propcoprime}]
We have
\begin{equation}
\begin{split}
M(F,N) &= \sum_{\substack{(\x,2\Delta)=1 \\ F(\x+\bs{\xi}) = N}}w(P^{-1}\x) \\
&= \sum_{d \mid 2\Delta}\mu(d)\sum_{F(d\x+\bs{\xi}) = N}w(P^{-1}d\x) \\
&= \sum_{d \mid 2\Delta}\mu(d)\sum_{\substack{F(\x) = N \\ \x \equiv \bs{\xi} \mod{d}}}w\left(\frac{\x - \bs{\xi}}{P}\right) \\
&= \sum_{d \mid 2\Delta}\mu(d)M_d(F,N),
\end{split}
\end{equation}
say. Note that $F(\bs{\xi}) - N \equiv 0 \pmod{2\Delta}$, by assumption. As a result, the sum $M_d(F,N)$ is well-defined. The Proposition is an immediate
consequence of Lemma~\ref{lemmacoprime}.
\end{proof}

Turning to the proof of Propositions~\ref{propthma11} and~\ref{propthma12}, we
will set $\y = 2\Delta (1,z_2,\ldots,z_h) = 2\Delta(1,\z)$ in~\eqref{eq:cdeff}. Then 
we get
\begin{equation}\label{eq:aux531sun}
\begin{split}
C(\x,\y) &= 2\Delta F(x_1,\ldots,x_r) + 4\Delta^2 \sum_{j=1}^{n-h}x_jq_j(1,z_2,\ldots,z_h) \\ 
&\quad \quad + 8\Delta^3 R(1,z_2,\ldots,z_h).
\end{split}
\end{equation}
For this choice of $\y$, we see that~\eqref{eq:cdeff} is soluble so long as
\begin{equation}\label{eq:auxillary}
F(x_1,\ldots,x_r) + 2\Delta \sum_{j=1}^{n-h}x_jq_j(1,\z) + 4\Delta^2  R(1,\z) = 0.
\end{equation}
Note that this equation has a non-trivial solution in $\ZZ_p$ for each prime $p$ provided that it is soluble modulo $\Delta$. However, this follows since $\rank F \geq 5$.
Therefore, in order to count solutions to~\eqref{eq:auxillary}, all 
that is required is to ensure solubility over $\RR$. 

Let
\begin{equation}\label{eq:definesz}
s(\z) = (q_1(1,\z),\ldots,q_{r}(1,\z)).
\end{equation}
Observe that the matrix $\adj M$, the adjoint matrix of $M$, has integer coefficients, whence 
\begin{equation*}
\Delta^2 F^{-1}(s(\z)) = F((\adj M) s(\z))
\end{equation*}
is an integer, where $F^{-1}$ is the quadratic form with matrix $M^{-1}$. 
By adding and subtracting $F((\adj M) s(\z))$, we get that~\eqref{eq:auxillary} is soluble if and only if
\begin{equation}\label{eq:auxillary1}
\begin{split}
F(\x + (\adj M) s(\z)) &= \Delta^2(F^{-1}(s(\z)) - 4R(1,\z)) \\
&\quad \quad + \Delta\sum_{i=r+1}^{n-h}x_iq_i(1,z_2,\ldots,z_h)
\end{split}
\end{equation}
is soluble.
\begin{lemma}\label{lemmasun525}
Let $C$ be as in~\eqref{eq:cdeff} with $F$ positive definite. Assume that
$y_1 \mid q_i(\y)$ for $1 \leq i \leq r$. Suppose that one of the following statements hold:
\begin{enumerate}
\item If $r < n-h$, there exists $\z \in \RR^{h-1}$ such that $q_j(0,\z) \neq 0$ for some $r+1 \leq j \leq n-h$, or
\item If $r = n-h$, there exists $\z \in \RR^{h-1}$ such that $R(0,\z) \neq 0$.
\end{enumerate}
Then there exists a set $\mathcal{C}_{n-r-1}(B) \subset \ZZ^{n-r-1} \cap [-B,B]^{n-r-1}$ such that
for any $(x_{r+1},\ldots,x_{n-h},y_2,\ldots,y_h) \in \mathcal{C}_{n-r-1}(B)$, the equation~\eqref{eq:auxillary1} is soluble over $\ZZ$
and that 
$$
\sum_{i=r+1}^{n-h}x_iq_i(1,\z) - \Delta R(1,\z) \geq B^2/4.
$$ 
Moreover, we have $|\mathcal{C}_{n-r-1}(B)| \gg B^{\frac{2}{3}(n-r-1)}.$
\end{lemma}
\begin{proof}
As we have already observed, to ensure solubility of~\eqref{eq:auxillary1} over $\ZZ$,
it suffices to check that it is soluble over $\RR$. We will treat the case where the first alternative
of the lemma holds. A similar argument works in the other case, so we omit the details.

We have that $r < n-h$ and
there exists $\z \in \RR^{h-1}$ such that $q_j(0,\z) \neq 0$ for some $r+1 \leq j \leq n-h$. 
Let $\uu \in \RR^{h-1}$ such that $q_j(0,\uu) \neq 0$. Let $U$ be an $\ve$ ball around $\uu$, for some 
fixed $0 < \ve < 1/1000$. Suppose that $\z \in \ZZ^{h-1} \cap B^{2/3}U$. Then for $B \gg 1$, we have
that $|q_j(0,\z)| \geq A_1B^{4/3}$ for some $A_1 > 0$. Note that
$$
q_j(1,\z) = q_j(0,\z) + O(B^{1/3}).
$$
As a result, we have for any $\z \in \ZZ^{h-1}\cap B^{2/3}U$ that $|q_j(1,\z)| \geq A_1B^{4/3}$, and that 
$$
\max_{\substack{r+1 \leq i \leq n-h \\ i \neq j}}|q_i(1,\z)|\leq A_2B^{4/3} \text{ and }|R(1,\z)| \leq A_2B^2,
$$ for a
constant $A_2$ that depend on $U$ and the coefficients of $C$. 

Let $c_1$ be a real number we will specify shortly. 
Define
\begin{align*}
\mathcal{C}_{n-r-1}(B) = \left\{(\w,\z) \in \ZZ^{n-r-1} \right.&: \z \in (\min\left\{1/2,(A_2\Delta)^{-1}\right\}B)^{2/3}U, \\
&\left.c_1B^{2/3}\leq |\w| \leq 2c_1B^{2/3} \text{ and}\right.\\
&\left.\sum_{i=r+1}^{n-h}w_iq_i(1,\z) - \Delta R(1,\z) \geq B^2/4\right\}.
\end{align*}
Observe that
$
\sum_{i=r+1}^{n-h}w_iq_i(1,\z) - \Delta R(1,\z) \geq B^2/4,
$ 
so long as 
$$
w_jq_j(1,\z) \geq \delta B^2 + \Delta R(1,\z) - \sum_{\substack{r+1 \leq i \leq n-h \\ i \neq j}}w_iq_i(1,\z),
$$
which we may ensure by taking $c_1 = 1/16(A_1+n),$ for example. As a result, we have shown that if $(\w,\z) \in \mathcal{C}_{n-r-1}(B)$, then
the equation~\eqref{eq:auxillary1} is soluble over $\ZZ$. 
In addition, it is easy to see that $|\mathcal{C}_{n-r+1}(B)| \gg B^{2/3(n-r-1)}$. 
\end{proof}

\begin{proof}[Proof of Proposition~\ref{propthma11}]
Let $q_i(\y) = \alpha_i y_1l_i(\y)$, where $\alpha_i(\y) \in \ZZ[\y]$ for $1 \leq i \leq r$ are linear forms. 
Suppose first that $r = n-h$. Then we have
\begin{equation}\label{eq:absurdmainterm1}
C(\x,\y) = y_1F(\x_1,\ldots,x_r) + y_1\sum_{i=1}^{n-h}\alpha_ix_il_i(\y) + R(\y).
\end{equation} 
We may suppose that $C$ is irreducible, for otherwise, we have the stronger lower bound $N(B) \gg B^{n-1}$. If
$C$ is irreducible, then $h(C) \geq 2$. Observe that 
$
C(\x,0,y_2,\ldots,y_h) = R(0,y_2,\ldots,y_h).
$
This implies that $h(R(0,\z)) \geq 1$, i.e., that the cubic
form $R(0,\z)$ does not vanish identically. As a result, we get that there exists $\z \in \ZZ^{h-1}$ such that $R(0,\z) < 0$. 

On the other hand, if $r < n-h$, then we may assume that $q_j(0,\z)$ does not vanish identically for some $r+1 \leq j \leq n-h$,
for otherwise, we may argue as in the previous paragraph. Thus we see that the hypotheses of Lemma~\ref{lemmasun525} are satisfied 
and let $\mathcal{C}_{n-r-1}(B)$ be as in statement of that lemma.

We will use the notation $\w = (w_{r+1},\ldots,w_{n-h})$ and $\z = (z_2,\ldots,z_h)$. Then we have
by~\eqref{eq:aux531sun},~\eqref{eq:auxillary} and~\eqref{eq:auxillary1} that
$$
N(B) \geq \sum_{(\w,\z) \in \mathcal{C}_{n-r-1}(B)}N_{\w,\z}(B),
$$
where
$$
N_{\w,\z}(B) = \#\left\{\x,2\Delta)=1: |\x| \leq B \text{ and } C(\x,\w,2\Delta(1,\z))=0\right\}.
$$
By setting $F= F_1$, $$
N = \Delta^2(F^{-1}(s(\z)) - 4R(1,\z)) + \Delta\sum_{i=r+1}^{n-h}w_iq_i(1,z_2,\ldots,z_h),
$$ 
and $\bs{\xi} = (\adj M) s(\z)$, it follows from 
Proposition~\ref{propcoprime} that $N_{\w,\z}(B) \gg B^{r-2}$. Putting everything together, we get that
$
N(B) \gg B^{r-2+\frac{2}{3}(n-r-1)},
$
as required.
\end{proof}

\begin{proof}[Proof of Proposition~\ref{propthma12}]
By hypothesis, there exists $1 \leq i \leq r$ such that $q_i(0,\z)$ does not 
vanish identically. Let $\w = (w_2,\ldots,w_h) \in \RR^{h-1}$ such that $q_i(0,\w) \neq 0$. Let $V$ be 
an $\ve-$ball around $\w$ with $\ve$ small enough such that $q_i(0,\z) \neq 0$ for any $\z \in V$. There
exists a constant $A_1 > 0$ such that for any $B \gg 1$, we get from 
the mean value theorem that $|q_j(1,\z)| \geq A_1B$ whenever $\z/B^{1/2} \in V$. Furthermore, 
there exists a constant $A_2 >0$ such that for any $|\z| \leq B^{1/2}$, we have that
$|R(1,\z)| \leq A_2B^{3/2}$ and $|q_i(1,\z)| \leq A_2B$ for each $1 \leq i \leq n-h$. 

Let $c_1 = \frac{1}{\Delta}\min\left\{1,A_1^{-1}\right\}$ and $c_2 = \frac{1}{1000\Delta}\min\left\{1,A_2^{-1}\right\}$. Define
\begin{align*}
\mathcal{C}_{n-r-1}(B) = \left\{(\w,\z) \in \ZZ^{n-r-1} :\right.& \w \in [-c_2B,c_2B]^{n-h-r} \text{ and} \\
&\left.\z/(c_1B)^{1/2} \in V\right\}.
\end{align*}
Then $|\mathcal{C}_{n-r-1}(B)| \gg B^{n-n-r+\frac{h-1}{2}}.$ Let $s(\z)$ be as in~\eqref{eq:definesz}. Set 
$$P^2 = N = \Delta^2(F^{-1}(s(\z)) - 4R(1,\z)) +  \Delta\sum_{i=r+1}^{n-h}x_iq_i(1,z_2,\ldots,z_h).$$
If $(\w,\z) \in \mathcal{C}_{n-r-1}(B)$, then we see that 
$N \geq B^2/4$, as $F^{-1}$ is positive-definite. Put $\bs{\xi} = (\adj M) s(\z)$. Then for any such $(\w,\z) \in \mathcal{C}_{n-r-1}(B)$,
we see that~\eqref{eq:auxillary1}
is soluble. Let
$$
N_{\w,\z}(B) = \#\left\{\x,2\Delta)=1: |\x| \leq B \text{ and } C(\x,\w,2\Delta(1,\z))=0\right\}.
$$
Using Proposition~\ref{propcoprime}, we get that $N_{\w,\z}(B) \gg B^{r-2}$ solutions. As a result,
$$
N(B) \gg B^{r-2}|\mathcal{C}_{n-r-1}(B)| \gg B^{r-2+n-h-r+\frac{h-1}{2}}. 
$$
This completes the proof of the proposition.
\end{proof}

\subsection{Preliminaries for the proof of Theorem~\ref{thma2}}\label{proofthma2} 
Let $\Delta_1,\ldots,\Delta_R$ denote the minors of $M[\y]$ of rank $r$. If $\pi$ has rank, 
then at least one of the minors $\Delta_i(\y)$ does not vanish identically. Moreover, all minors of higher rank 
vanish identically in $\y$. In this section, we will prove some technical results that we will need later in the proof of Theorem~\ref{thma2}.

\begin{lemma}\label{lemmawt}
Suppose that the fribration $\pi$ has rank $r \geq 5$ and that $Q_{\y}(\x)$ does not satisfy Hypothesis~\ref{h1}. Let $\Delta(\y)$ be a minor of order $r$ that does not vanish identically. Then there exists an open set $U$ in $\RR^h$ such that $\Delta(\y) \neq 0$ for any $\y \in U$. Furthermore, for 
any $\y \in U$, the quadratic form $Q_\y(\x)$ is indefinite.
\end{lemma}
\begin{proof}
We will show that there exists $\mathbf{u} = (u_1,\ldots,u_h) \in \RR^h$ such that $\rank \sum_{i=1}^h u_i F_i(\x)$ is $r$ and that if
$\y$ lies in an $\ve$-ball around $\mathbf{u}$ then $Q_\y(\x)$ is indefinite. By shrinking the neighbourhood, if necessary, we will also get that
$\Delta(\y) \neq 0$. This is the open set we seek. We will discuss the cases where $r = n-h$ and $r < n-h$ separately.

Suppose first that $r = n-h$. Then in this case, $\Delta(\y) = \det M[\y]$ does not vanish identically. As a result, at least one of the quadratic forms $F_i(\x)$ is non-singular. Suppose without loss of generality that $F_1(\x)$ is of full rank. If $F_1(\x)$ is also indefinite, set $\mathbf{u} = (1,0,\ldots,0).$ Then in this case, we can take $U$ to be an $\ve$-ball around $\mathbf{u}$ with $\ve < 1/2$.

Thus we may assume henceforth that $F_1(\x)$ is definite. Let $s = \rank_L \sum_{i=2}^hy_iF_i(\x)$ where $L = \QQ(y_2,\ldots,y_h)$. 
As $Q_{\y}(\x)$ does not satisfy Hypothesis~\ref{h1}, we have that $s \geq 1$. Suppose first that $s \geq 2$. Then there exists $(u_2,\ldots,u_h) \in \RR^{h-1}$ such that 
$\rank \sum_{i=2}^h u_i F_i(\x) \geq 2$. Since $F_1(\x)$ is definite, there exists an orthogonal transformation $R$ that depends on
$(u_2,\ldots,u_h)$ such that $F_1(R\x)$ and 
$\sum_{i=2}^h u_iF_i(R\x)$ 
are both diagonal quadratic forms. Let $\epsilon = 1$ if $F_1$ is positive-definite and $-1$ otherwise. 
Then there exist 
$\lambda_i(\z) \neq 0$ 
such that
$$
F_1(R\x) = \epsilon\sum_{i=1}^{n-h}x_i^2 \,\text{ and } \sum_{i=2}^h u_iF_i(R\x) = \sum_{i=1}^{n-h}\lambda_i(\z)x_i^2.
$$ 
Since at least two of $\lambda_i(\z) \neq 0$, we may choose $u_1$ suitably to ensure that 
$$
Q_{\uu}(R\x) = \sum_{i=1}^h u_iF_i(R\x) = \sum_{i=1}^{n-h}x_i^2(\epsilon u_1 + \lambda_i(\z))
$$
is non-singular and indefinite. As a result, we get that $Q_{\uu}(\x)$ is non-singular and indefinite. Let $U$ be an $\ve$-ball around $\mathbf{u}$. As 
$\lambda_i(\z)$ are continuous functions, if $\ve$ is 
small enough, we get that $Q_\y(\x)$ is non-singular and indefinite for any $\y \in U$. 

If $s = 1$, then it is easy to see that there exist linear forms $l(y_2,\ldots,y_h)$ and $m(\x)$ such that $\sum_{i=2}^h y_iF_i(\x) = l(y_2,\ldots,y_h)m(\x)^2$. Then
after a linear change of variables, we may assume that $l(y_2,\ldots,y_h)$ can be replaced by $y_2,$ whence we may assume that $F_2(\x) = m(\x)^2$ and $F_i = 0$ for 
$3 \leq i \leq n-h$. As in the previous case, 
there exists $R \in O_{n-h}(\RR)$ and $\epsilon = \pm 1$ such that $F_1(R\x)$ and $F_2(R\x)$ are simultaneously diagonal. Thus we get that 
$$
u_1F_1(R\x) + u_2F_2(R\x) = \sum_{i=1}^{n-h}(\epsilon u_1 + \lambda_iu_2)x_i^2,
$$
where $\lambda_i \neq 0$ for exactly one $1 \leq i \leq n-h$. We may now choose $u_1$ and $u_2$ appropriately to ensure that $Q_\mathbf{u}(\x)$ is non-singular and indefinite,
where $\uu = (u_1,u_2,0,\ldots,0)$. Taking $U$ to be a small neighbourhood around $\mathbf{u}$ completes the proof of the lemma when $r = n-h$.

Suppose next that $r < n-h$. Then in this case, we get from Lemma~\ref{lemma4w} that $Q_\y(\x)$ is linear in at least $n-h-\rho$ variables $x_i$. If $Q_\y(\x)$ is independent of $n-h-r$ variables, then the lemma follows by repeating the argument from the full rank case. Therefore, we may suppose that there exist variables $x_i$ and $x_j$ with $i \neq j$ such that the coefficient of $x_i^2$ in $Q_\y(\x)$ is identically zero, but the coefficient of $x_ix_j$ is not. Let $L_{i,j}(\y)$ denote the coefficient of $x_ix_j$ appearing in $Q_\y(\x)$. Then we may find $\y_0\in \RR^h$ such that $\Delta(\y_0)L_{i,j}(\y_0) \neq 0$. Then it follows that $\rank M[\y_0] = r$. Write
$$
Q_{\y_0}(\x) = L_{i,j}(\y_0)x_ix_j + L_{j,j}(\y_0)x_j^2+\widetilde{Q}(\x,\y_0),
$$
where $\widetilde{Q}(\x,\y_0)$ is a quadratic form in $\x$ that vanishes if we let all the variables except $x_i$ and $x_j$ to be $0$. By setting all the $x_k$ except $x_i$ and $x_j$ to be $0$, it is clear that $Q_{\y_0}(\x)$ is indefinite. Taking $U$ to be an $\ve$-ball around $\y_0$ completes the proof of the lemma.
\end{proof}

Next, we characterise all matrices $M[\y]$ with the property that all its minors of order $3$ have a common factor. 
\begin{lemma}\label{lemmarank3ormore}
Let $\phi(\y) \in \ZZ[\y]$ be the greatest common divisor of the minors of order $3$ in 
$M[\y] = \sum_{i=1}^h y_iM_i$. Suppose that $\phi(\y)$ is a non-constant polynomial. 
Assume that $r = r(C) \geq 5$ and suppose that $M_1 = \diag (a_1,\ldots,a_r,0,\ldots,0)$ with $a_i \neq 0$. 
Then after a $\QQ$-linear change of variables,
we may write $M[\y] = \sum_{i=1}^h y_iN_i$, such that $\rank N_1 = r$ and $\rank \sum_{i=2}^h y_iN_i \leq 2$ for each $(y_2,\ldots,y_h) \in \AA^{h-1}$. 
\end{lemma}

\begin{proof}
We begin by recording an observation that we will use repeatedly in the proof below: that $\phi(1,0,\ldots,0) \neq 0$, for otherwise, we
obtain that
$\rank M_1 \leq 2$, which contradicts our assumption that $\rank M_1$ is at least $5$. This observation implies that the degree of $y_1$
in the polynomial $\phi(\y)$ is equal to the degree of $\phi$. 

Suppose first that $l(\y) \mid \phi(\y)$, where $l$ is a linear form. Making the change of variables $l(\y) = z_1$, $y_2 = z_2, \ldots, y_n = z_n$,
we get that there exist $\alpha_i \in \QQ$ (with $\alpha_1 \neq 0$) such that
\begin{align*}
M[\y] &= \left(\sum_{i=1}^h \alpha_i z_i \right) M_1 + \sum_{i=2}^h z_iM_i \\
&= \alpha_1z_1M_1 + \sum_{i=2}^h z_iN_i,
\end{align*}
for some matrices $N_i$. Note that $\rank N_1 = \rank \alpha_1M_1 = \rank M_1$, since $\alpha_1 \neq 0$. Moreover, the condition
$z_1 = 0$ implies that $\rank \sum_{i=2}^h z_iN_i \leq 2$, by our assumption that $l$ divides $\phi$, which in turn divides every minor of order $3$. 
Thus the lemma
follows in this case. 

We will now proceed to show that $\phi$ is always divisible by a linear factor. As $\phi$ divides every minor of order $3$,
we see that the degree of $\phi$ is at most $3$, thus it suffices to show that $\phi$ is reducible over $\QQ$.

Let $1 \leq i_1 < i_2 < i_3 \leq n-h$ and $1 \leq j_1 < j_2 < j_3 \leq n-h$. Set $I = \left\{i_1,i_2,i_3\right\}$ and $J = \left\{j_1,j_2,j_3\right\}$. The minors of
$M[\y]$ of order $3$ are of the form
$$
\Phi_{I,J}(\y) = \det (m_{i,j}(\y))_{\substack{i \in I \\ j \in J}},
$$
where $m_{i,j}(\y)$ denotes the entries of the matrix $M[\y]$. 

Suppose first that $\phi$ is irreducible polynomial of degree $3$. Suppose that there exists a minor $\Phi_{I,J}(\y)$ with $|I \cap J| = 2$ that 
does not vanish identically. As $\phi$ divides $\Phi_{I,J}$, this forces the degree of $y_1$ in $\phi$ to be at most $2$, which in turn forces
$\phi(1,0,\ldots,0) = 0$, as the degree of $\phi$ is equal to $3$. This implies that the minors $\Phi_{I,J}(\y)$ must
all vanish identically whenever $|I \cap J|= 2$. For any $(i,j)$ with $1 \leq i,j \leq n-h$ and $i \neq j$, consider the minor $\Phi_{I,J}(\y)$ 
with $I = \left\{a,b,i\right\}$ and $J = \left\{a,b,j\right\}$ with $a,b \leq r$ and $a \neq b$. As $\Phi_{I,J}(\y)$ vanishes identically, we see that 
the coefficient of $y_1^2y_k$ must be $0$. This forces $m_{i,j}(\y) = 0$ whenever $i \neq j$. As a result, if $|I \cap J| = 3$, with
$I = \left\{i_1,i_2,i_3\right\}$, say, we get that
\begin{align*}
\Phi_{I,J}(\y) = \det \begin{pmatrix} m_{i_1,i_1}(\y) & 0 & 0 \\
0 & m_{i_2,i_2}(\y) & 0 \\
0 & 0 & m_{i_3,i_3}(\y) \end{pmatrix}.
\end{align*}
Therefore, we obtain that either $\Phi_{I,J}$ vanishes identically, or it is a product of three linear factors, which contradicts our assumption that $\phi$ is 
irreducible.

All that remains is to eliminate the possibility that $\phi$ is an irreducible polynomial of degree $2$. Suppose that this is the case. 
As we have already observed, $\phi$ is quadratic in $y_1$. Thus
 we see that the minors $\Phi_{I,J}(\y)$ vanish
identically whenever $|I \cap J| \leq 1$. If $|I \cap J| = 2$, then we get that
\begin{equation*}
\begin{split}
\phi_{I,J}(\y) &= y_1^2L(y_2,\ldots,y_h) + y_1U(y_2,\ldots,y_h) + V(y_2,\ldots,y_h),
\end{split}
\end{equation*}
where $L$ is linear, $U$ is quadratic and $V$ is cubic. And if $|I \cap J| = 3$, we see that
\begin{equation*}
\begin{split}
\phi_{I,J}(\y) &= \alpha y_1^3 + y_1^2S(y_2,\ldots,y_h) + y_1T(y_2,\ldots,y_h),
\end{split}
\end{equation*}
with $\alpha \in \ZZ$, $S$ linear and $T$ quadratic. In either case, using the fact that $\phi$ is quadratic in $y_1$, we get that
$y_1 \mid \phi_{I,J}(\y)$, which contradicts the assumption that $\phi$ is irreducible. This completes the proof of the lemma. 
\end{proof}

We end this section by using a result of Rojas-Le\'on~\cite{RL05} to count the number of $\x \in \FF_p^m$ such that
$f(\x)$ is a quadratic residue modulo $p$ for homogeneous polynomials $f$ subject to some mild hypotheses 
on the dimension of their singular loci.
\begin{lemma}\label{lemmakatz}
$f(x_1,\ldots,x_m)$ be a homogeneous form of degree $d$. Suppose that the 
dimension of the singular locus of $f$ in $\AA^m$ is at most $m-2$. Then for $p \nmid d$ sufficiently large in terms of the
coefficients of $f$, we have
$$
\#\left\{\x \in \FF_p^{m}: \jacobi{f(\x)}{p} = 1\right\} \gg p^m.
$$

\end{lemma}
\begin{proof}
Let $X = \PP^{m}$ denote projective space with homogeneous coordinates $x_1,\ldots,x_m,z$. We will
denote by $L$ the hypersurface in $X$ given by $z = 0$. Let
$H = f$ and let $H$ also denote the hypersurface cut out by its vanishing in $X$. 
Let $V = X \backslash (L \cup H)$. We see that $H/L^4$ is a function from $V \to \AA^1$. 
Note that 
$X \cap L \cap H$ is pure of codimension $2$ in $X$. 

Let $\ve' = \dim \sing X \cap H$ and let $\delta = \dim \sing X \cap L \cap H$. Then
$\delta$ is the dimension of the singular locus of $f$ in $\AA^m$. Observe that $\ve' = \delta$. Let $\chi(t)$ denote
the multiplicative character $\jacobi{t}{p}$. Set 
\begin{align*}
S = \sum_{w \in V(\FF_p)}\chi(H(w)/L(w)^4). 
\end{align*}
Then we clearly have
\begin{align*}
S &= \sum_{\x \in \AA^m(\FF_p)}\chi(f(\x)) - \sum_{\x \in \AA^m(\FF_p) : f(\x) = 0}1.
\end{align*}

We are now ready to apply~\cite[Theorem 1.1]{RL05} with $\ve' = \delta$. 
Note that all the hypotheses of that theorem are satisfied and we deduce that 
$$
S \ll p^{\frac{m+\delta+1}{2}}.
$$
As $\delta \leq m-2$, by assumption, we get that
$S \ll p^{m-1/2}$.

Observe that
\begin{align*}
2\#\left\{\x \in \FF_p^{m}: \jacobi{f(\x)}{p} = 1\right\} &= \sum_{\x \in \FF_p^m} 1 + \sum_{\x \in \FF_p^m}\chi(f(\x)) - \sum_{\x \in \FF_p^m : f(\x) = 0}1 \\
&\quad  \\
&= p^m + S = p^m + O(p^{m-1/2}),
\end{align*}
by the bound in the previous paragraph. This completes the proof of the lemma.

\end{proof}
\subsection{Counting points on a family of affine quadrics}\label{secaffquad}
To prove Theorem~\ref{thma2}, we may assume once again that $h \leq 13$. For otherwise, a stronger result follows
from Theorem~\ref{proph14intro}. Let $r = r(C)$ denote the rank of the fibration $\pi$. We have by assumption that $r \geq \min\left\{n-h - 4,5\right\}$. 
In addition, by the hypothesis of Theorem~\ref{thma2}, it follows that $Q_{\y}(\x)$ does not satisfy Hypothesis~\ref{h1}. 

Let $\mathbf{u} = (u_1,\ldots,u_h) \in \RR^h$ be a point in the open set $U$ that was constructed in Lemma~\ref{lemmawt}. Let
$
\Omega_{\infty}
$
be a compact ball of radius $c > 0$ around $\mathbf{u}$, where $c$ is chosen so that $\Omega_{\infty} \subset U$. Let $1 \leq Y \leq B^{2/3-\eta}$ be a parameter, for 
some $\eta > 0$. 

For $\y \in \ZZ^h$ let $F_\y(\x)$ be as in~\eqref{eq:deffyx}.
Then $F_\y(\x)$ is a quadratic polynomial in $\x$ whose quadratic part is given by $Q_\y(\x)$ (see ~\eqref{eq:defqyx}). 
For $\y \in \ZZ^h \cap Y\Omega_{\infty}$, we have by construction that $Q_\y(\x)$ is indefinite and of rank equal to $r \geq 5$. Furthermore, we have that $\Delta(\y) \gg Y^r$, 
where $\Delta(\y)$ is a minor of $M[\y]$ of rank $r$.

Let $\lambda_1(\y),\ldots,\lambda_r(\y)$ denote the non-zero eigenvalues of the quadratic form $Q_\y$. 
As $\lambda_i(\y)$ are the eigenvalues of a symmetric matrix of norm $O(Y)$, we see that 
$|\lambda_i(\y)| \ll Y$. As $|\Delta(\y)| \gg Y^r$, we get that $\prod_{i=1}^r |\lambda_i(\y)| \gg Y^r$, whence we obtain the bound
\begin{equation}\label{eq:imp1}
Y \ll |\lambda_i(\y)| \ll Y
\end{equation} 
for each $1 \leq i \leq r$, with the implied constants depending only on the coefficients of the cubic form $C$.

Let $R_\y \in O_{n-h}(\RR)$ such that 
$Q_\y(R_\y\x) = \sum_{i=1}^r \lambda_i(\y)x_i^2$ is a diagonal quadratic form. Let $\z \in \RR^{n-h}$ be a real solution to the equation 
$$
\tilde{Q}_\y(\x) = \sum_{i=1}^{r}\frac{\lambda_i(\y)}{|\lambda_i(\y)|}x_i^2 = 0.
$$
Let $\|F_\y\|$ denote the largest coefficient (in absolute value) of the quadratic form $Q_\y(\x)$.
$$
u_\y(\x) = \prod_{i=1}^{n-h}w_0(z_i-x_i).
$$
Define
\begin{equation}\label{eq:wyxdefuy}
w_\y(\x) = u_\y(R_\y^tD\x),
\end{equation}
where 
$$
D=\diag(|\lambda_1(\y)|^{1/2},\ldots,|\lambda_{r}(\y)|^{1/2},\|F_\y\|^{1/2},\ldots,\|F_\y\|^{1/2}).
$$

For $P \geq 1$ set
$$
N(P,w_\y) = \sum_{\substack{\x \in \ZZ^{n-h} \\ F_\y(\x) = 0 \\ (\x,\y)=1}} w_\y(P^{-1}\x).
$$
Recall the counting function $N_\y(B)$ defined in~\eqref{eq:nyqyb}. Since $w_\y(\x)$ is non-negative, we see by setting $P = B\sqrt{Y}$ and~\eqref{eq:imp1} that 
\begin{equation}\label{eq:nybggnbsqrty}
N_\y(B) \gg N(B\sqrt{Y},w_\y).
\end{equation}
Let $\mathcal{C}_h(Y) \subset \ZZ^h \cap Y\Omega_{\infty}$ be a set such that the equation $F_\y = 0$ is soluble over $\ZZ_p$ for 
each prime $p$. Then we obtain
$$
N(B) \gg \sum_{\y \in \mathcal{C}_h(Y)}N(B\sqrt{Y},w_\y) .
$$
In order to get a lower bound for $N(B\sqrt{Y},w_\y)$, we will call upon work of Kumaraswamy~\cite[Theorem 1.1]{VVK24} to 
count solutions to the quadratic equation $F_\y = 0$. Note that all the hypotheses of that theorem are satisfied, as $Q_\y$ is indefinite of rank equal to $r \geq 5$ by construction, 
$Y \leq B^{2/3-\eta}$ and~\eqref{eq:imp1}.

For the rest of this section, we will set  
$
m = n-h
$ and define
\begin{equation}\label{eq:sqiqdef}
S_q(\y) = \sumstar_{a\bmod{q}}\sum_{\b \bmod{q}}e_q(aF_\y(\b)).
\end{equation}
Define the singular series
$$
\mathfrak{S}(F_\y) = \sum_{q=1}^{\infty}q^{-m}S_q(\y).
$$
If $\y \in \mathcal{C}_h(Y)$, then $\rank Q_\y \geq 5$, whence we have
$$
\mathfrak{S}(F_\y) = \prod_{p} \sigma_p(F_\y),
$$
where
$$
\sigma_p(F_\y) = \sum_{k=0}^{\infty}\frac{S_{p^k}(\y)}{p^{km}} = \lim_{t \to \infty} p^{-(m-1)t}N_\y(p^t),
$$
with
$$
N_\y(p^t) = \#\left\{\x \bmod{p^t} : F_\y(\x) \equiv 0 \bmod{p^t}\right\}.
$$ 

On applying~\cite[Theorem 1.1]{VVK24}, we will encounter the constant $\mathfrak{S}(F_\y)$. If $F_\y$ is soluble over $\ZZ_p$ for each prime $p$, then
it follows that $\mathfrak{S}(F_\y) > 0$. We will prove a lower bound for $\mathfrak{S}(F_\y)$ that is uniform in $\y$ on average. 

Let $\left\{\phi_1(\y),\ldots,\phi_J(\y)\right\}$ be the set of 
minors of order $3$ in $M[\y]$. Recall the variety
$$\mathcal{L} \subset \AA^h: \phi_1(\y) = \phi_2(\y) = \ldots = \phi_J(\y) = 0$$
that we used in the proof of Theorem~\ref{thmsolubility2}. We have the following results.
\begin{proposition}\label{propsingserieslbav}
Suppose that $\dim \mathcal{L} \leq h-2$. For each $Y \gg 1$, there exists $\mathcal{C}_h(Y) \subset \ZZ^h \cap Y\Omega_\infty$
such that if $\y \in \mathcal{C}_h(Y)$, the following statements hold:
\begin{enumerate}
\item For any $\ve > 0$ we have
$\mathfrak{S}(F_\y) \gg_{\ve} Y^{-\ve}.
$
\item The equation $F_\y=0$ is soluble over $\ZZ$ 
\item We have that $\gcd (y_1,\ldots,y_h) \ll_C 1$.
\end{enumerate}
\end{proposition}

\begin{proposition}\label{propsingserieslbav2}
Suppose that $\dim \mathcal{L} = h-1$. For each $Y \gg 1$, there exists $\mathcal{C}_h(Y) \subset \ZZ^h \cap Y\Omega_\infty$
such that if $\y \in \mathcal{C}_h(Y)$, then the equation $F_\y=0$ is soluble over $\ZZ$ and $\gcd (y_1,\ldots,y_h) = 1$.
Moreover, we have for any $\ve > 0$ that 
\begin{equation}\label{eq:54422}
\sum_{\y \in \mathcal{C}_h(Y)}\mathfrak{S}(F_\y) \gg_{\ve} Y^{h-\ve}.
\end{equation}
\end{proposition}
We will prove Proposition~\ref{propsingserieslbav} first. 
\subsubsection{When $\dim \mathcal{L} \leq h-2$}
The following is the key result which allows us to obtain strong lower bounds for $\mathfrak{S}(F_\y)$.
\begin{lemma}\label{lemmalbforsfy}
Let $\y \in \ZZ^h$ and assume that $\rank Q_\y = r\geq 5$. Suppose that there exists a constant $A = A(C)$ that depends only on the cubic form
$C$ such that 
\begin{equation}\label{eq:hyplemmalbforsfy}
\#\left\{\x \in \FF_p^m : F_\y(\x) = 0 \text{ and } \nabla F_\y(\x) \neq \mathbf{0}\right\}  \geq p^{m-1} + O(p^{m-2})
\end{equation}
for all $p \nmid A$. Suppose also that for each prime $p \mid A$, we have that there exists $v_p = v_p(C) \geq 1$ such that the set 
\begin{equation}\label{eq:hyp2lemmalbforsfy}
\begin{split}
\left\{ \x \bmod{p^{2v_p-1}} \right.: 
&\left.F_\y(\x) \equiv 0 \bmod{p^{2v_p-1}} \text{ and } \nabla F_\y(\x) \not\equiv \mathbf{0} \bmod{p^{v_p}}\right\}
\end{split}
\end{equation} 
is non-empty. Then we have that $\mathfrak{S}(F_\y) \gg (1+|\y|)^{-\ve}$.
\end{lemma}
\begin{proof}

Let $\phi_1(\y),\ldots,\phi_J(\y)$ be the minors of 
$M[\y]$ of order $r$. Then $\phi_j(\y) \neq 0$ for some $j$, by assumption. Set $$\mathcal{P}(\y) = \left\{p \text{ prime} : p\nmid A \gcd(\phi_1(\y),\ldots,\phi_J(\y))\right\}.$$ 
As we have already remarked, the singular series converges absolutely as $\rank Q_\y \geq 5$. So we may write
$$
\mathfrak{S}(F_\y) = \left(\prod_{p \in \mathcal{P}(\y)} \sigma_p(F_\y)\right)
\left(\prod_{p \not\in \mathcal{P}(\y)} \sigma_p(F_\y)\right).
$$
We have
$$
\sigma_p(F_\y) = 1 + \sum_{k=1}^{\infty}\frac{S_{p^k}(\y)}{p^k},
$$
where $S_{p^k}(\mathbf{0})$ is as in~\eqref{eq:sqiqdef}. Then
if $p \in \mathcal{P}(\y)$, we obtain from~\cite[Lemma 2.4]{VVK24} that 
$$
\sigma_p(F_\y) = 1 + O(p^{-3/2}),
$$
since $\rank Q_\y \geq 5$, where the implied constant is independent of $F_\y$. As a result, we get that
\begin{equation}\label{eq:singseriesawayfrombadprimes}
\left(\prod_{p \not\in \mathcal{P}(\y)} \sigma_p(F_\y)\right) \gg 1.
\end{equation}

Suppose next that~\eqref{eq:hyplemmalbforsfy} holds, we get from Lemma~\ref{lemmadavhensel} that 
$$
N_\y(p^t) \geq p^{t(m-1)} + O(p^{t(m-1)-1}).
$$
As a result, we get for any $p \nmid A$ that
$
\sigma_p(F_\y) \gg 1-c/p,
$
for some $c > 0$, independent of $p$. 

Finally, if $p \mid A$, we get from~\eqref{eq:hyp2lemmalbforsfy} and
Lemma~\ref{lemmadavhensel} that $N_\y(p^t) \gg p^{t(m-1)}$ if $t > 2v_p$, where the implicit constant
depends only on $C$, by hypothesis. This implies that
$\sigma_p(F_\y) \gg_C 1$ if $p \mid A$. 
As a result, we get for any $\ve > 0$ that
\begin{align*}
\left(\prod_{p \in \mathcal{P}(\y)} \sigma_p(F_\y)\right) &\gg \prod_{p \mid A}\sigma_p(F_\y)\prod_{p \mid \gcd(\phi_1(\y),\ldots,\phi_J(\y))} \left(1-\frac{c}{p}\right) \\
&\gg_{\ve} (1+|\y|)^{-\ve},
\end{align*}
since $\phi_i(\y) \ll_C (1+|\y|)^r$ for each $i$. The lemma follows from~\eqref{eq:singseriesawayfrombadprimes}.
\end{proof}
\begin{proof}[Proof of Proposition~\ref{propsingserieslbav}]
If $\dim \mathcal{L} \leq h-2$, let 
$\mathcal{C}_h(Y) \subset \ZZ^h \cap Y\Omega_\infty$ be the set we obtain from applying Theorem~\ref{thmsolubility2}. Let $\y \in \mathcal{C}_h(Y)$. 
The first statement of the proposition
follows from Lemma~\ref{lemmalbforsfy}. As $\mathfrak{S}(F_\y) > 0$, we obtain that $F_\y=0$ is soluble over $\ZZ_p$ for each prime $p$. Moreover,
by definition, any such $\y$ also lies in $Y\Omega_\infty$, whence we get that $F_\y=0$ has a real solution, which shows that $F_\y$ is everywhere locally soluble. 
Therefore, by the Hasse principle for quadratic polynomials, $F_\y=0$ is soluble over $\ZZ$. The third statement follows from Remark~\ref{remgcd} 
and this completes the proof of the proposition.  
\end{proof}
\subsubsection{When $\dim \mathcal{L} = h-1$}
By Lemma~\ref{lemmarank3ormore}, we get after making a linear change of variables that
$$
F_\y(\x) = C(\x,\y) = \sum_{i=1}^h y_iF_i(\x) + \sum_{j=1}^{m}x_jq_j(\y) + R(\y),
$$
where $m = n-h$, $F_1(\x)$ is an indefinite quadratic form of rank equal to $r$ and $\rank \sum_{i=2}^hy_iF_i(\x) \leq 2$ for every $(y_2,\ldots,y_h) \in \AA^{h-1}$.
Note that the polynomials $F_i$, $q_j$ and $R$ could be different to the ones we started with. We are recycling notation for ease of exposition.

For the remainder of this section, we set
\begin{equation}\label{eq:sdefn}
s = \rank_L \sum_{i=2}^h y_iF_i(\x)
\end{equation}
where $L = \QQ(y_2,\ldots,y_h)$. Although our construction of the set $\mathcal{C}_h(Y)$ will 
depend on whether $s = 0, 1$ or $2$, we will adopt a similar approach in each case, which we explain below. 

We may assume without loss of generality that $F_1$ is non-singular of rank $r$ with discriminant $D$. Let $M$ denote the
product of all the coefficients of $C$. 
Define 
\begin{equation}\label{eq:defaprod}
A = 2\prod_{p \mid DM}p.
\end{equation}
For each prime $p \mid A$, let $v_p$ and $\r_p = (r_p^{(1)},\ldots,r_p^{(h)})$ 
denote the exponent and residue class that we 
obtain from  
applying Lemma~\ref{lemma0}. Let $\delta$ be a small fixed constant such that
$\delta^{-1} > \prod_{p \mid A} p^{2v_p-1}$ and let $\mathcal{P}(\delta Y)$ denote the set of 
primes in the interval $[\delta Y/2, \delta Y]$ that are congruent to $r_p^{(1)} \bmod{p^{2v_p-1}}$
for every $p \mid A$. We will always choose $\mathcal{C}_h(Y)$ so that if $\y$ belongs to it, then $y_1 \in \mathcal{P}(\delta Y)$. 
Then by arguing as we did to
obtain~\eqref{eq:singseriesawayfrombadprimes}, 
we get that
\begin{equation*}\label{eq:boundawayfrombadprimescodim1}
\prod_{p \nmid Ay_1}\sigma_p(F_\y) \gg 1.
\end{equation*}
Our choice of 
$\mathcal{C}_h(Y)$ be will be strongly infulenced by studying the solubility of $F_\y$ over $\FF_{y_1}$. Finally, for the finitely many primes $p \mid A$, 
we will ensure that~\eqref{eq:hyp2lemmalbforsfy} holds by restricting $\y \in \mathcal{C}_h(Y)$
to lie in the arithmetic progression modulo $\r_p \bmod{p^{2v_p-1}}$.

\begin{lemma}\label{lemmas0}
Suppose that $s = 0$ and that $h(C) \geq 6$. Then for each $Y \gg 1$ 
there exists a set $\mathcal{C}_h(Y) \subset \ZZ^h \cap Y\Omega_\infty$ such that 
$$
\sum_{\y \in \mathcal{C}_h(Y)}\mathfrak{S}(F_\y) \gg Y^h/\log Y.
$$
\end{lemma}
\begin{proof}
If $s = 0$, then we see that $F_i(\x) = 0$ for each $2 \leq i \leq h$. Suppose first that there exist linear forms 
$l_i$ such that $q_i(\y) = y_1l_i(\x)$ for each $1 \leq i \leq m$. Then we have
$$
C(\x,\y) = y_1F_1(\x) + y_1\sum_{i=1}^m x_il_i(\y)+ R(\y).
$$
Observe that $C(\x,0,y_2,\ldots,y_h) = R(0,y_2,\ldots,y_h)$ is a cubic form in $n-1$ variables with $h$-invariant at least $h(C)-1$. As a result, we see that
$R(0,y_2,\ldots,y_h)$ is a cubic form in $h(C)-1$ variables whose $h$-invariant is equal to $h(C)-1$, which implies, in particular that it is non-degenerate.
As a result, we find that the equation
$R(0,y_2,\ldots,y_h) = 0$ has no non-trivial solutions over $\FF_p$ for large enough $p$. However, according to~\cite[Theorem 3]{DL62},
any non-degenerate cubic form with $h$-invariant at least $5$ has a non-trivial solution over $\FF_p$ for any prime $p$. This 
is a contradiction, since $R(0,y_2,\ldots,y_h)$ has $h$-invariant at least $5$.

As a result, we may assume that $y_1 \nmid q_i$ for some $1 \leq i \leq m$. Let $y_1 \in \mathcal{P}(\delta Y)$. This implies that the equation
$q_i(0,z_2,\ldots,z_h)$ is not identically zero over $\FF_{y_1}$. As a result, if we set
$$
S(y_1) = \left\{(z_2,\ldots,z_h) \in \FF_{y_1}^{h-1} : q_i(0,z_2,\ldots,z_h) \neq 0\right\},
$$
we must have that $|S(y_1)| \gg y_1^{h-1}$. Let $A$ be as in~\eqref{eq:defaprod} and let $v_p$ and $\r_p$ be the exponents
obtained by applying Lemma~\ref{lemma0} for each $p \mid A$.  Define
\begin{align*}
\mathcal{C}_h(Y) = \left\{\y \in \ZZ^h \cap Y\Omega_{\infty} : \right.&\, y_1 \in \mathcal{P}(\delta Y), (y_2,y_1)=1, \text{ for each } p \mid A \\
&\left. (y_2,\ldots,y_h) \equiv (r_p^{(2)},\ldots,r_p^{(h)}) \bmod{p^{2v_p-1}} \right.\\
&\left.  (y_2,\ldots,y_h) \bmod{y_1} \in S(y_1) \right\}.
\end{align*}
Then it is easy to see that $|\mathcal{C}_h(Y)| \gg Y^h/\log Y$.

Let $\y \in \mathcal{C}_h(Y)$. We will now show that $\mathfrak{S}(F_\y) \gg 1$. Indeed, 
if $p \mid A$, then the estimate~\eqref{eq:hyp2lemmalbforsfy} holds by construction. 
If $p \nmid A$, and $p \neq y_1$, then $\rank_{\FF_p} F_\y \geq 3$ and the
estimate~\eqref{eq:hyplemmalbforsfy} follows from Lemma~\ref{lemmasolvequadricsmodp}. As a result, we obtain that
\begin{equation}\label{eq:singseriesneqy1}
\prod_{p \neq y_1}\sigma_p(F_\y) \gg 1.
\end{equation}
Finally, 
if $p = y_1$, we see that 
$\nabla F_\y(\x) \not\equiv \mathbf{0}$ for any $\x \in \FF_{y_1}^m$. As $F_\y(\x)$ is a linear equation modulo $y_1$, we see that
$$
N_{\y}(y_1^t) \gg y_1^{t(m-1)},
$$
for all $t \geq 1$, whence we get that $\sigma_{y_1}(F_\y) \gg 1$, and the lemma follows.


\end{proof}

Next, we deal with the case where $s = \rank_L \sum_{i=2}^hy_iF_i(\x) = 1$. Then by an application of Lemma~\ref{lemma4w}, we get
after a linear change of variables that $F_i(\x) = \alpha_ix_1^2$ for some $\alpha_i \in \ZZ$, not all equal to $0$. After making 
another change of variables,
we may suppose that there exists $\alpha \neq 0$ such that
\begin{equation}\label{eq:eqncs1}
C(\x,\y) = y_1F_1(\x) + \alpha y_2x_1^2 + \sum_{i=1}^m x_iq_i(\y) + R(\y),
\end{equation}
with $F_1(\x)$ an indefinite, non-singular quadratic form of rank $r$. We will need the following lemma later on in the argument. 

\begin{lemma}\label{lemmacalsingloc}
Let $R(y_1,\ldots,y_m)$ be a cubic form with $h$-invariant $h(R).$ Suppose that $h(R) \geq 3$ and let $q(y_1,\ldots,y_m)$ be a
quadratic form. Set $f(\y) = q^2(\y)-4y_1R(\y)$. Let $\sing f \subset \AA^m$ denote the singular locus of $f$. Then 
$\dim \sing f \leq m-2$. 
\end{lemma}
\begin{proof}
In order to obtain a contradiction, assume that $\dim \sing f = m-1$. This implies that
there exists a polynomial $\phi(\y)$ such that $\phi \mid \gcd\left(\frac{\partial f}{\partial y_1},\ldots,\frac{\partial f}{\partial y_m}\right)$.
If $y_1 \mid \phi$, then we get that $y_1 \mid \frac{\partial f}{\partial y_1}$ and that $y_1 \mid f(\y)$. We may then deduce that
$y_1 \mid R(\y)$, which contradicts our assumption that $h(R) \geq 3$. As a result, we must have that $y_1 \nmid \phi$.

Let $\psi(\z) = \phi(0,\z)$. Then $\psi$ doesn't vanish identically. 
Observe that $\psi(\z) \mid f^{(i)}(0,\z)$ for each $1 \leq i \leq m$. This in turn implies
that $\psi(\z) \mid f(0,\z)$, whence we get that $\psi(\z) \mid q(0,\z)$. Using the fact that $\psi(\z) \mid \frac{\partial f(0,\z)}{\partial z_1}$,
we also obtain that $\psi(\z) \mid - 2q(0,\z)\frac{\partial q(0,\z)}{\partial z_1}+R(0,\z) $, whence we get that $\psi(\z) \mid R(0,\z)$. However,
this forces $h(R(0,\z)) \leq 1$, which contradicts our assumption that $h(R(0,\z)) \geq 2$. This completes the proof of the lemma.
\end{proof}

\begin{lemma}\label{lemmas1}
Suppose that $s = 1$. Assume that $h(C) \geq 5$. Then for each $Y \gg 1$ there exists a set $\mathcal{C}_h(Y) \subset \ZZ^h \cap Y\Omega_\infty$ such that 
$$
\sum_{\y \in \mathcal{C}_h(Y)}\mathfrak{S}(F_\y) \gg Y^h/\log Y.
$$
\end{lemma}
\begin{proof}
Let $C$ be as in~\eqref{eq:eqncs1}. Suppose first that $y_1 \nmid q_t(\y)$ for some $2 \leq t \leq m$. We will assume without
loss of generality that $t = 2$. 
Define
\begin{align*}
\mathcal{C}_h(Y) = \left\{\y \in \ZZ^h \cap Y\Omega_{\infty} : \right.&\,y_1 \in \mathcal{P}(\delta Y), (y_2,y_1)=1, \text{for each } p \mid A \\
&\left. (y_2,\ldots,y_h) \equiv (r_p^{(2)},\ldots,r_p^{(h)}) \bmod{p^{2v_p-1}}\right.\\
&\left. \text{and }q_2(0,y_2,\ldots,y_h) \not\equiv 0 \bmod{y_1}
\right\},
\end{align*}
where $\mathbf{r}_p$ and $A$ are as before. As $y_1 \nmid q_2(\y)$,
we get that the set of $\z \in \FF_{y_1}^{h-1}$ such that $q_2(0,\z) = 0$ has cardinality $O(y_1^{h-2})$.
Thus we get that
$\mathcal{C}_h(Y) \gg Y^h/\log Y$.

Turning to estimating $\mathfrak{S}(F_\y)$, note that by the same argument as in Lemma~\ref{lemmas0}, we have 
that~\eqref{eq:singseriesneqy1}
holds. Finally, 
over $\FF_{y_1}$, we find that 
$$
F_\y(\x) = \alpha y_2 x_1^2 + \sum_{i=1}^m x_iq_i(0,\y) + R(0,\y)
$$
and
$$
\nabla F_\y(\x) = \left(2\alpha y_2x_1 + q_1(0,\y),q_2(0,\y),\ldots,q_m(0,\y)\right).
$$
If $\y \in \mathcal{C}_h(Y)$, using the fact that $q_2(0,\y) \in \FF_{y_1}^*$, we get that
the number of non-singular
solutions to $F_\y = 0$ over $\FF_{y_1}$ is equal to $y_1^{m-1}$. Therefore, we have
by Lemma~\ref{lemmadavhensel} that $N_{\y}(y_1^t) \gg y_1^{t(m-1)}$ for each
$t \geq 1$, whence $\sigma_{y_1}(F_\y) \gg 1$. This shows that $\mathfrak{S}(F_\y) \gg 1$ and the lemma follows.

Suppose next that $y_1 \mid q_t(\y)$ for each $2 \leq t \leq m$.
Define
\begin{align*}
\mathcal{C}_h(Y) = \left\{\y \in \ZZ^h \cap Y\Omega_{\infty} : \right.&\, y_1 \in \mathcal{P}(\delta Y), \text{ for each } p \mid A \\
&\left. (y_2,\ldots,y_h) \equiv (r_p^{(2)},\ldots,r_p^{(h)}) \bmod{p^{2v_p-1}}, \right.\\
&\left. (y_2,y_1)=1 \text { and } \right. \\
&\left. q_1^2(0,y_2,\ldots,y_h) - 4\alpha y_2R(0,y_2,\ldots,y_h) \text{ is}\right.\\
&\left.\text{a non-zero quadratic residue modulo }y_1\right\}.
\end{align*}
Observe that if we set $x_1 = y_1 = 0$ in~\eqref{eq:eqncs1}, we get $R(0,y_2,\ldots,y_h)$. 
As a result, we see that $h(R(0,y_2,\ldots,y_h)) \geq 3$. Let $f(\y) = q_1^2(0,\y) - 4\alpha y_2R(0,\y)$. By Lemma~\ref{lemmacalsingloc},
we get that $\dim \sing f \leq h-3$, and therefore by Lemma~\ref{lemmakatz} we get that $\mathcal{C}_h(Y) \gg Y^h/\log Y$.

By the argument used earlier in the lemma, we get for any $\y \in \mathcal{C}_h(Y)$ that~\eqref{eq:singseriesneqy1} holds.
All that remains is to estimate $\sigma_{y_1}(F_\y)$. Since $y_1 \mid q_t(\y)$ for each $2 \leq t \leq m$, by assumption, we see that 
$$
F_\y(\x) = \alpha y_2x_1^2 + x_1q_1(0,\y) + R(0,\y)
$$
modulo $y_1$.
If $q_1^2(0,\y)-4\alpha y_2R(0,\y)$ is a non-zero quadratic residue modulo $y_1$, we get from Lemma~\ref{lemmaevalgauss} 
and~\eqref{eq:evalexpsums} (or by direct computation) that
$F_\y(\x) = 0$ has precisely $2y_1^{m-1}$ non-singular solutions over $\FF_{y_1}$. As a result, we 
obtain from Lemma~\ref{lemmadavhensel} that $N_{\y}(y_1^t) \gg y_1^{t(m-1)}$ for all $t\geq 1$, which implies that
that $\sigma_{y_1}(F_\y) \gg 1.$ This completes the proof of the lemma.
\end{proof}

Finally, suppose that $s = 2$. Set $\z = (y_2,\ldots,y_h)$ so that $\y = (y_1,\z)$. Then we obtain from Lemma~\ref{lemmarank2overk} 
that there exist integer linear forms $l_1(\z), l_2(\z)$ and $l_3(\z) \in L$ such that $4l_1l_3-l_2^2 \in L^*$ 
\begin{equation}\label{eq:rank2option1}
C(\x,\y) = y_1F_1(\x) + l_1(\z) x_1^2+ l_2(\z)x_1x_2 + l_3(\z) x_2^2 + \sum_{i=1}^m x_iq_i(\y) + R(\y),
\end{equation}
or  
\begin{equation}\label{eq:rank2option2}
C(\x,\y) = y_1F_1(\x) + \tilde{l}_2(\z)x_1x_2 + x_1\sum_{i=3}^m\tilde{l}_i(\z)x_i + \sum_{i=1}^mx_iq_i(\y) + R(\y),
\end{equation}
with $\tilde{l}_i\in L$ linear forms such that $\tilde{l}_2(\z) \in L^*$. In either case, we may
assume that $F_1$ is non-singular with rank $r$.
\begin{lemma}\label{lemmas2}
Suppose that $s = 2$. Suppose that if $C$ is as in~\eqref{eq:rank2option1} then $\dim_L V \neq 1$, 
where $V = \spn_{\QQ}\left\{l_1,l_2,l_3\right\}$,
or that $C$ is as in~\eqref{eq:rank2option2}.
Then there exists a set $\mathcal{C}_h(Y) \subset \ZZ^h \cap Y\Omega_\infty$ such that 
$$
\sum_{\y \in \mathcal{C}_h(Y)}\mathfrak{S}(F_\y) \gg Y^h/\log Y.
$$
\end{lemma}
\begin{proof}
Define 
\begin{equation}\label{eq:defg}
G = \begin{cases} 4l_1l_3-l_2^2 &\mbox{ if $C$ is of the form~\eqref{eq:rank2option1},} \\
4\tilde{l}_2^2 &\mbox{ if $C$ is of the form~\eqref{eq:rank2option2}}.
\end{cases}
\end{equation}
Note that $G \in L^*$ by assumption. We set
\begin{equation*}
\begin{split}
\mathcal{C}_h(Y) = \left\{\y \in \ZZ^h \cap Y\Omega_{\infty} : \right.&\, y_1 \in \mathcal{P}(\delta Y), (y_2,y_1) = 1, \\
&\left. G(y_2,\ldots,y_h) \not\equiv 0 \bmod{y_1}, \right.\\
&\left. \jacobi{G(y_2,\ldots,y_h)}{y_1} = 1 \text{ and for each } p \mid A \right.\\
&\left. (y_2,\ldots,y_h) \equiv (r_p^{(2)},\ldots,r_p^{(h)}) \bmod{p^{2v_p-1}} \right\}.
\end{split}
\end{equation*}
We claim that $\mathcal{C}_h(Y) \gg Y^h/\log Y$. This is obvious if $G = 4\tilde{l}_2^2$. So suppose that $G = 4l_1l_3-l_2^2$. 
As $\dim_L V \neq 1$, by hypothesis of the lemma, it is easy to see that $\dim \sing G \leq h-3$ 
and the claim now follows from
Lemma~\ref{lemmakatz}.

Let $\y \in \mathcal{C}_h(Y)$. 
We must now obtain a lower bound for $\mathfrak{S}(F_\y)$. By arguing as in the proof of Lemma~\ref{lemmas0}, 
we see that~\eqref{eq:singseriesneqy1} holds,
and all that remains is to consider the behaviour of $\sigma_p(F_\y)$ for $p = y_1$.

If $F_\y$ is as in~\eqref{eq:rank2option2}, then over $\FF_{y_1}$ we may make the linear change of variables
$$\tilde{l}_2(y_2,\ldots,y_h)x_2+\sum_{i=3}^m \tilde{l}_i(y_2,\ldots,y_h)x_i \mapsto x_2,
$$
and observe that $F_\y$ has the form
$$
x_1x_2 + x_1q_1(0,\y) + \sum_{i=1}^m B_ix_i + N
$$
for some $B_i$ and $N \in \FF_{y_1}$.

If $F_\y$ is as in~\eqref{eq:rank2option1}, then over $\FF_{y_1}$, we see that
$$
F_\y(\x) =  l_1(\y) x_1^2+ l_2(\y)x_1x_2 + l_3(\y) x_2^2 + \sum_{i=1}^m B_ix_i + N,
$$
for some $B_i$ and $N \in \FF_{y_1}$. If $(B_i,y_1) = 1$, for some $i \geq 3$, as we have observed previously,
we obtain the estimate $N_{y_1}(F_\y) = y_1^{m-1}$, whence we get that $\sigma_{y_1}(F_\y) = 1$, which
is sufficient to deduce the lemma.

Therefore, we may suppose that $B_i = 0$ for each $3 \leq i \leq m$. As a result, over $\FF_{y_1}$, the equation $F_\y$ is of the form
\begin{equation}\label{eq:finals2}
\x^tM\x + \mathbf{B}^t\x + N
\end{equation}
where $M$ is given by one of the following two matrices,
$$
\begin{pmatrix} 0 & 1/2 \\ 1/2 & 0 
\end{pmatrix} \, \, \text{ or } \, \, \begin{pmatrix} l_1(\z) & l_2(\z)/2 \\ l_2(\z)/2 & l_3(\z) \end{pmatrix},
$$
$\mathbf{B} = (B_1,B_2) \in \FF_{y_1}^2$ and $N \in \FF_{y_1}$.

By Lemma~\ref{lemmaevalgauss}, we get that the number of non-singular solutions to~\eqref{eq:finals2} is equal to
\begin{equation}\label{eq:temp534}
y_1^{m-2}\left\{y_1 + \jacobi{-\det M}{y_1}y_1K_2(4N-\bs{B}^tM^{-1}\bs{B},y_1) - \kappa_{y_1}\right\}.
\end{equation}
Note that (up to squares) $-\det M = G$, where $G$ is as in~\eqref{eq:defg}. Observe by~\eqref{eq:evalexpsums} that~\eqref{eq:temp534}
 is $\gg y_1^{m-1}$, 
so long as $\jacobi{G}{y_1} = 1$, which holds by construction of $\mathcal{C}_h(Y)$. Thus we get by applying Lemma~\ref{lemmadavhensel} that
$N_{\y}(y_1^t) \gg y_1^{t(m-1)}$ for all $t \geq 1$, whence we get that $\sigma_{y_1}(F_\y) \gg 1$ and this completes the proof of the lemma.
\end{proof}

\begin{lemma}\label{lemmas22}
Suppose that $s = 2$. Suppose that $C$ is as in~\eqref{eq:rank2option1} and that $\dim_L V = 1$, 
where $V = \spn_{\QQ}\left\{l_1,l_2,l_3\right\}$.  
Then there exists a set $\mathcal{C}_h(Y) \subset \ZZ^h \cap Y\Omega_\infty$ such that 
$$
\sum_{\y \in \mathcal{C}_h(Y)}\mathfrak{S}(F_\y) \gg Y^h/\log Y.
$$
\end{lemma}
\begin{proof}
As $\dim_L V = 1$, we may assume without loss of generality that $l_1 = Ay_2, l_2 = By_2, l_3 = Cy_2$, for some $A, B, C \in \ZZ$ 
such that $B^2-4AC \neq 0$. The proof will be similar to that of Lemma~\ref{lemmas2}, so we will only point out the 
key differences.

Let $J = 4AC - B^2$. Then by quadratic reciprocity, there exist integers $r$ and $\Delta$ (depending on $J$)
such that if $p \equiv r \bmod{\Delta}$ then
$\jacobi{J}{p} = 1$. Consider the cubic polynomial
$$
\tilde{C}(\x,\y) = C(\x,\Delta y_1 + r, y_2,\ldots,y_h).
$$
Observe that $\tilde{C}$ is irreducible, given that $C$ is irreducible. As the set of $\QQ_p$ rational points of 
$C$ are Zariski dense and since $\tilde{C}(\x,0,y_2,\ldots,y_h) = C(\x,r,y_2,\ldots,y_h)$, we see that $\tilde{C}$ has a non-singular
solution over $\QQ_p$ for any $p$. Furthermore, it is clear that the partial derivatives $\frac{\partial \tilde{C}}{\partial x_i}$ cannot 
all vanish. Thus all the hypotheses of Lemma~\ref{lemma4w} are satisfied. 

We redefine the parameter $A$ from~\eqref{eq:defaprod} as follows:
$$
A = 2\prod_{p \mid D\Delta M}p,
$$ 
where $D$ is the discriminant of $F_1$ and $M$ is the product of coefficients of $C$.

For each prime $p \mid A$, let $v_p$ and $\r_p = (r_p^{(1)},\ldots,r_p^{(h)})$ 
denote the exponent and residue class that we 
obtain from  
applying Lemma~\ref{lemma0} to $\tilde{C}$. Let $\delta$ be a small fixed constant such that
$\delta^{-1} > \prod_{p \mid A} p^{2v_p-1}$ and let $\mathcal{P}(\delta Y)$ denote the set of 
primes in the interval $[\delta Y/2, \delta Y]$ that are congruent to $r_p^{(1)} \bmod{p^{2v_p-1}}$
for every $p \mid A$. As before, we will choose $y_1 \in \mathcal{P}(\delta Y)$. Observe
now that if $Y$ is large enough, then we have ensured by construction of $\tilde{C}$ that $\jacobi{J}{p} = 1$. 

Define
\begin{equation*}
\begin{split}
\mathcal{C}_h(Y) = \left\{\y \in \ZZ^h \cap Y\Omega_{\infty} : \right.&\, y_1 \in \mathcal{P}(\delta Y), (y_2,y_1) = 1 \\
&\left. \text{and for each } p \mid A \right.\\
&\left. (y_2,\ldots,y_h) \equiv (r_p^{(2)},\ldots,r_p^{(h)}) \bmod{p^{2v_p-1}} \right\}.
\end{split}
\end{equation*}
Then it is easily verified that $|\mathcal{C}_h(Y)| \gg Y^h/\log Y$. 

To estimate $\mathfrak{S}(F_\y)$ for $\y \in \mathcal{C}_h(Y)$, we will
follow the argument from Lemma~\ref{lemmas2}. In particular, it suffices to analyse the behaviour of~\eqref{eq:finals2}. 
In this case, we see that $-\det M = Jy_2^2$, whence we obtain $\jacobi{-\det M}{y_1} = 1$. Thus we get $\mathfrak{S}(F_\y) \gg 1$ and 
the lemma follows. 
\end{proof}

\begin{proof}[Proof of Proposition~\ref{propsingserieslbav2}]
If $\dim \mathcal{L} = h-1$,~\eqref{eq:54422} follows from taking
$\mathcal{C}_h(Y)$ to be as in Lemmas~\ref{lemmas0},~\ref{lemmas1},
~\ref{lemmas2}, or~\ref{lemmas22} depending on the value of $s$ (see~\eqref{eq:sdefn}). In each case, we see that $\gcd (y_1,\ldots,y_h) = 1$ for any $\y \in \mathcal{C}_h(Y)$,
by construction, as we have ensured that $(y_1,y_2)=1$. Next, observe that by our
analysis of $\mathfrak{S}(F_\y)$, we obtain that $F_\y=0$ is soluble over $\ZZ_p$ for each prime $p$, whenever $\y \in \mathcal{C}_h(Y)$. As $\y$ also lies in $Y\Omega_\infty$, 
we also get real solubility for the equation $F_\y = 0$.  
Therefore, by the Hasse principle for quadratic polynomials, $F_\y=0$ is soluble over $\ZZ$. This completes the proof of the proposition.
 \end{proof}
\subsection{Proof of Theorem~\ref{thma2}}
Suppose that Hypothesis~\ref{h1} holds. Let $\Omega_{\infty}$ be as in \S~\ref{secaffquad}. Let $1 \leq Y \leq B^{2/3-\eta}$ for some $\eta > 0$. 
Define
\begin{equation*}
\mathcal{E}(B,Y) = B^{m-\frac{r}{2}-\frac{1}{2}}Y^{\frac{m}{2}+\frac{r-3}{4}} + B^{m-\frac{r}{2}-\frac{\kappa}{2}}Y^{\frac{m}{2}+\frac{r+\kappa}{4}-2}.
\end{equation*}
Our task 
now will be to show that
\begin{equation}\label{eq:nbbm-2yh-1}
N(B) \gg_{\ve} B^{m-2}Y^{h-1-\ve} + O_{\ve}((BY)^{\ve}\mathcal{E}(B,Y)Y^h).
\end{equation}
Then Theorem~\ref{thma2} will follow by choosing $Y$ appropriately.

Let $\mathcal{L}$ be the variety defined in~\eqref{eq:444315}. 
Suppose that $\dim \mathcal{L}=h-1$ and let $\mathcal{C}_h(Y) \subset \ZZ^h \cap Y\Omega_\infty$
be as in Proposition~\ref{propsingserieslbav2}. Let $\y \in \mathcal{C}_h(Y)$. Then $\gcd(y_1,\ldots,y_h)=1$. 
Recall the counting function $N_\y(B)$ from~\eqref{eq:nyqyb}. We have by~\eqref{eq:nybggnbsqrty} that
\begin{align*}
N_\y(B) \gg N(B\sqrt{Y},w_\y),
\end{align*}
where $w_\y$ is as in~\eqref{eq:wyxdefuy} and 
$$
N(B\sqrt{Y},w_\y) =  \sum_{\substack{F_\y(\x) = 0}}w_\y\left(\left(B\sqrt{Y}\right)^{-1}g\x\right).
$$ 
Set 
\begin{equation*}
\kappa = \begin{cases}0 & \mbox{ if $2 \mid r,$} \\ 1 &\mbox{ otherwise.}
\end{cases}
\end{equation*}

We get from Theorem~\cite[Theorem 1.1]{VVK24}, that for any $\ve > 0$ we have
\begin{equation}\label{eq:lbnbsqrtywy}
N(B\sqrt{Y},w_\y) = \frac{\sigma_{\infty}(\tilde{Q}_\y,u_\y)\mathfrak{S}(F_\y)B^{m-2}Y^{\frac{m}{2}-1}}{\prod_{i=1}^r|\lambda_i(\y)|^{\frac{1}{2}}\|F_{\y}\|^{\frac{m-r}{2}}} + O((BY)^{\ve}\mathcal{E}(B,Y)).
\end{equation}
Moreover, we get by using the fact that $\tilde{Q}_{\y}(\x)$ has a solution in the support of $u_{\y}$ that
\begin{equation}\label{eq:boundsingintmainthm}
\sigma_{\infty}(\tilde{Q}_\y,u_\y) \gg 1
\end{equation}
independent of $\y$. Thus we obtain 
from~\eqref{eq:imp1},~\eqref{eq:lbnbsqrtywy} and~\eqref{eq:boundsingintmainthm}
that
\begin{equation*}\label{eq:conclusioncodim2}
N_\y(B) \gg N(B\sqrt{Y},w_\y) \gg_{\ve} \mathfrak{S}(F_\y)B^{m-2}Y^{-1} + O_{\ve}((BY)^{\ve}\mathcal{E}(B,Y)).
\end{equation*}
Summing over $\y \in \mathcal{C}_h(Y)$ we get from Proposition~\ref{propsingserieslbav2} that~\eqref{eq:nbbm-2yh-1} holds.

Suppose next that $\dim \mathcal{L} \leq h-2$. We will show that~\eqref{eq:nbbm-2yh-1} holds in this case as well. 
Let $\mathcal{C}_h(Y)$ be the set we obtain from Proposition~\ref{propsingserieslbav}. Let $\y \in \mathcal{C}_h(Y)$ and 
let $g = g_\y = \gcd(y_1,\ldots,y_h)$. Then $g \ll_C 1$. In this case we have
\begin{align*}
N_\y(B) &\gg \sum_{\substack{(\x,g)=1 \\ F_\y(\x) = 0}}w_\y\left(\left(B\sqrt{Y}\right)^{-1}g\x\right) \\
&\gg \sum_{d \mid g}\mu(d)\sum_{\substack{F_\y(d\x) = 0}}w_\y\left(d\left(B\sqrt{Y}\right)^{-1}\x\right).
\end{align*}
Observe that
$
F_\y(d\x) = C(d\x,\y) = d^3 C(\x,\y/d) = d^3 F_{\y/d}(\x).
$
As $\y/d \in \mathcal{C}_h(Y)$, we see that $\mathfrak{S}(F_{\y/d}) \gg Y^{-\ve}$. As a result, we may proceed as before to conclude once again that
$$
N_\y(B)\gg_{\ve} \mathfrak{S}(F_\y)B^{m-2}Y^{-1} + O_{\ve}((BY)^{\ve}\mathcal{E}(B,Y)).
$$ 
Using the lower bound $\mathfrak{S}(F_\y) \gg Y^{-\ve}$ and summing over $\mathcal{C}_h(Y)$, we get that~\eqref{eq:nbbm-2yh-1} holds, as required.

All that remains now is to choose $Y$ in~\eqref{eq:nbbm-2yh-1} so that the main term exceeds the error term. On writing
\begin{align*}
\mathcal{E}(B,Y) = B^{m-\frac{r}{2}}Y^{\frac{m}{2}+\frac{r-3}{4}}\left\{B^{-\frac{1}{2}} + B^{-\frac{\kappa}{2}}Y^{\frac{\kappa-5}{4}}\right\},
\end{align*}
we see that $$
\max\left\{B^{-\frac{1}{2}}, B^{-\frac{\kappa}{2}}Y^{\frac{\kappa-5}{4}}\right\} = 
\begin{cases} B^{-\frac{\kappa}{2}}Y^{\frac{\kappa-5}{4}} &\mbox{ if $\kappa = 0$ and $Y \ll B^{\frac{2}{5}}$},\\
B^{-\frac{1}{2}} &\mbox{ otherwise.}
\end{cases}
$$ 
Recall that $m = n-h$ and define for $\ve > 0$
\begin{align*}
Y &= \begin{cases} B^{\frac{2(r-4)}{2m+r-4}-\ve} &\mbox{ if $r$ is even and $2r-m < 8$,} \\
B^{\frac{2(r-3)}{2m+r+1}-\ve} &\mbox{ otherwise.}
\end{cases}
\end{align*}
With this choice of $Y$, we see that the error term in~\eqref{eq:nbbm-2yh-1} is smaller than the main term and we get that
\begin{align*}
N(B) &\gg_{\ve} B^{n-h-2}Y^{h-1-\ve}. 
\end{align*}
This completes the proof of Theorem~\ref{thma2}.

%% file: pieces.tex
In this section we will examine cubic forms that are of the following shape:
\begin{equation}\label{eq:shape}
C(\x) = x_1Q_1(\y)+\ldots+x_5Q_5(\y) + R(\y),
\end{equation}
where $\y = (x_6,\ldots,x_n)$, the $Q_i(\y)$ are non-zero quadratic forms and $R(\y)$ is a cubic form. Our aim will be to prove Theorem~\ref{thmb}.
\subsection{Preliminaries}
Let 
\begin{equation}\label{eq:fxydefn355}
F(\x,\y) = \sum_{i=1}^5 x_iQ_i(\y).
\end{equation}
 We will assume that $F(\x,\y)$ is non-degenerate in the $x_i$ variables. This implies that none of the quadratic forms $Q_i(\y)$ can vanish identically.

Suppose that $F(\x,\y)$ is reducible over $\QQ$ and let 
$$
F(\x,\y) = l(\x,\y)q(\x,\y),
$$ 
where $l(\x,\y)$ is a linear form and $q(\x,\y)$ is a quadratic form. We will 
now show that $l(\x,\y)$ is independent of $\x$ and that $q(\x,\y) = \sum_{i=1}^5 x_il_i(\y)$ for linear forms $l_i(\y)$. Note that such 
$l_i(\y)$ are necessarily coprime. For otherwise, $F(\x,\y)$ becomes degenerate in $\x$.

Let 
$l(\x,\y) = \sum_{i=1}^5 \alpha_i x_i + \sum_{i=1}^{n-5}l_iy_i$, 
say. Suppose first that $l_i = 0$ for each $i$. As $F$ is linear in $\x$, we see that $q(\x,\y)$ can only comprise of terms of terms that are 
quadratic in $y_j$. Then $q(\x,\y) = G(\y)$, whence $F(\x,\y) =  l(\x)G(\y)$. Making a linear change of variables, $F(\x,\y)$ can be made independent of $x_2,\ldots,x_4$, which contradicts our assumption that $F(\x,\y)$ is non-degenerate in $\x$. 

So we must have that $l_i \neq 0$ for some $i$. As $F(\x,\y)$ is at most quadratic in $\y$, we see that $q(\x,\y)$ cannot have terms that are quadratic in $\y$. Observe
now that $\alpha_j = 0$ for each $1 \leq j \leq 5$, for otherwise, by the linearity of $F$ in $\x$, $q$ cannot have terms of the form $x_uy_v$ or $x_ux_v$, which would
force $q = 0$, which is impossible. 

To recap, we have shown that $l(\x,\y) = l(\y)$ depends only on $\y$ and that 
$q(\x,\y) = \sum_{u,v}s_{u,v}x_uy_v = \sum_u x_u \sum_v s_{u,v}y_v$, for $s_{u,v} \in \ZZ$. Put $l_u(\y) = \sum_v s_{u,v}y_v$. Then we get 
$$
F(\x,\y) = l(\y)\sum_{u=1}^5 x_ul_u(\y),
$$
as required. We record our findings in the following result.
\begin{lemma}\label{lemmacharacterisation}
Suppose that $F(\x,\y)$ given in~\eqref{eq:fxydefn355} is non-degenerate in the $x_i$. If $F(\x,\y)$ is reducible over $\QQ$, then there exist linear forms 
$l(\y)$, $l_1(\y)$, $l_2(\y)$, $\ldots$, $l_5(\y)$ $\in$ $\ZZ[\y]$ such that $F(\x,\y) = l(\y)\sum_{i=1}^5x_il_i(\y).$ 
\end{lemma}

\begin{remark}
Although we will not use this observation in our work, the preceding argument shows that if $F$ is irreducible, then $F$ is, in fact, geometrically irreducible. 
\end{remark}

As a result of Lemma~\ref{lemmacharacterisation}, our proof of Theorem~\ref{thmb} naturally splits into the case where $F(\x,\y)$ is 
irreducible, and the case where 
$F(\x,\y)$ has a linear factor over $\QQ$. The rest of the section is devoted to the proof of the following theorems, which clearly imply
Theorem~\ref{thmb}.

\begin{theorem}\label{thmbc}
Let $C(\x)$ be a form in $n$ variables such that $h(C) \geq 8$. Suppose that $C(\x)$ is of the form given in~\eqref{eq:shape}. Assume
 that $\sum_{i=1}^5x_iQ_i(\y)$ is reducible over $\QQ$ and non-degenerate in the $x_i$. Then for any $\ve>0$ we have
$$
N(B) \gg_{\ve} B^{n-3-\ve}.
$$
\end{theorem}

\begin{theorem}\label{thmbb}
Let $C(\x)$ be a cubic form in $n \geq 10$ variables. Suppose that $C(\x)$ is of the form given in~\eqref{eq:shape}. Assume that the cubic
form $\sum_{i=1}^5 x_iQ_i(\y)$ is irreducible over $\QQ$ and non-degenerate in the $x_i$.
Then for any $\ve>0$ we have
$$
N(B) \gg_{\ve} B^{n-3-\ve}.
$$
\end{theorem}

Our approach in proving the theorems will be as follows. For any non-empty compact set $\Omega_{\infty} \subset \RR^{n-5}$, and a parameter $Y \leq B$, both of which will be chosen later, we have
$$
N(B) \geq \sum_{\substack{\y \in Y\Omega_{\infty} \cap \ZZ^{n-5} \\ \text{ such that} \\ C(\x,\y) = 0\text{ is} \\ \text{ soluble over }\ZZ}}
\sum_{\substack{|\x| \leq B \\ (\x,\y) \in \ZZ_{\prim}^n \\ C(\x,\y) = 0}}1.
$$
For an integer $l$, let 
\begin{equation}\label{eq:nylbdef}
N_{\y,l}(B) = \#\left\{|\x| \leq B : (\x,l) = 1 \text{ and }C(\x,\y) = 0\right\}.
\end{equation} 
Our strategy will be to obtain an asymptotic formula for $N_{\y}(B)$ by appealing to Proposition~\ref{proplineq}. To be able to execute the $\y$-sum in Theorem~\ref{thmbb}, we will need to 
restrict $\y$ to lie in a carefully chosen subset, which ensures, in particular, the solubility of the equation $C(\x,\y) = 0$.

\begin{lemma}\label{lemmacyred}
Let 
\begin{equation*}\label{eq:cxyykr1}
C(\x,\y) = y_k\sum_{i=1}^2\alpha_ix_iy_i + R_1(x_3,x_4,x_5,\y)
\end{equation*}
where $R_1(x_3,x_4,x_5,\y) = \alpha_3x_3y_3y_k+ \alpha_4x_4y_4y_k+ \alpha_5x_5y_5y_k + R(\y),$ with
$\alpha_i \neq 0$ and $R$ a cubic form. Suppose that $h(C) \geq 8$. Let $g = \gcd(\alpha_1,\alpha_2)$ and let $\beta_i = \alpha_i/g$ for $1 \leq i \leq 2$. 
Define
\begin{equation}
\begin{split}
\mathcal{C}(Y) = \left\{|x_3|,|x_4|,|x_5|,|\y| \leq Y:\right.&\, (\beta_1 y_1,\beta_2 y_2) = 1, y_k \in \mathcal{P}(\delta Y), \\ 
&\left.R_1(x_3,x_4,x_5,\y) \equiv 0 \bmod{y_k} \right\}
\end{split}
\end{equation}
If $(x_3,x_4,x_5,\y) \in \mathcal{C}(Y)$, then $C(x_1,x_2,gx_3,gx_4,gx_5,g\y) = 0$ soluble over $\ZZ$. Moreover, we have $|\mathcal{C}(Y)| \gg Y^{n-3}/\log Y$. 
\end{lemma}
\begin{proof}
Let $(x_3,x_4,x_5,\y)\in \mathcal{C}(Y)$. For $\b = (b_3,\ldots,b_{n-5}) \in \AA^{n-7},$
let $\hat{\b} = (b_3,\ldots,b_{k-1},0,b_{k+1},\ldots,b_{n-5}) \in \AA^{n-7}.$ Define
$$
M(y_1,y_2,y_k) = \#\left\{\b \in \FF_{y_1}^{n-5} : R(\hat{\b}) \equiv 0 \mod{y_k}\right\}.
$$
Observe that modulo $y_k$ we have $C(\x,\y) = R(\hat{\y}) = R_1(x_3,x_4,x_5,\hat{\y}).$ For any fixed $y_1,y_2$, 
set 
\begin{align*}
S_{y_1,y_2}(y_3,\ldots,y_{k-1},y_{k+1},\ldots,y_{n-5}) =\quad \quad \quad \quad \quad &\\ R(y_1,y_2,y_3,\ldots,y_{k-1},0,y_{k+1},\ldots,y_{n-5}).
\end{align*}
Then $S_{y_1,y_2}$ is a cubic polynomial in $n-8$ variables whose cubic part is independent of 
$y_1$ and $y_2$. 

Since $S_{y_1,y_2}$ was obtained by fixing $3$ of the variables in $C(\x,\y)$, we have $h(S_{y_1,y_2}) \geq h(C) - 3.$ Then according to~\cite[Theorem 3]{DL62}, we have
\begin{equation}\label{eq:my1y2yk}
\begin{split}
M(y_1,y_2,y_k) &= y_k^{n-9} + O\left(y_k^{n-8-\frac{h(C)-3}{4}}\right) \\
&\gg y_k^{n-9} 
\end{split}
\end{equation}
if $h(C) \geq 8$. Note that the implied constant depends only on the cubic part of $S$, which, as we have already noted, is
independent of $y_1,y_2$. 

We may write
\begin{align*}
|\mathcal{C}(Y)| &= \sum_{|x_3|,|x_4|,|x_5| \leq Y}\sum_{\substack{|\y| \leq Y \\ (\beta_1y_1,\beta_2y_2) = 1 \\ y_k \in \mathcal{P}(\delta Y) \\ R_1(x_3,x_4,x_5,\y) \equiv 0 \bmod{y_k}}}1 \\
&\geq Y^3\sum_{\substack{|\y| \leq Y \\ (\beta_1y_1,\beta_2y_2) = 1 \\ y_k \in \mathcal{P}(\delta Y) \\ R_1(x_3,x_4,x_5,\y) \equiv 0 \bmod{y_k}}}1,
\end{align*}
as $R_1$ is independent of $x_3,x_4$ and $x_5$ modulo $y_k$. Fixing $y_1,y_2$ and $y_k$ and splitting up the other variables $y_i$ into progressions modulo $y_k$,
we get from~\eqref{eq:my1y2yk} that
\begin{align*}
|\mathcal{C}(Y)| &\gg Y^3\sum_{\substack{|\y| \leq Y \\ (\beta_1y_1,\beta_2y_2) = 1 \\ y_k \in \mathcal{P}(\delta Y)}}(Y/y_k)^{n-8}y_k^{n-9} \\
&\gg Y^{n-3}/\log Y,
\end{align*}
since
\begin{align*}
&\#\left\{|y_1|,|y_2| \leq Y : (\beta_1y_1,\beta_2y_2) = 1\right\} = \\
&\quad \quad \quad Y^2\prod_{p \nmid \beta}\left(1-\frac{1}{p^2}\right)\prod_{p \mid \beta}\left(1- \frac{\prod_{i=1}^2(p,\alpha_i)}{p^2}\right) + O(Y).
\end{align*}
This completes the proof of the lemma.

\end{proof}

\begin{lemma}\label{lemmalinwt}
Let $C$ be as in~\eqref{eq:shape} and that $n \geq 10$. Let $Q_1(\y),\ldots,Q_5(\y)$ be non-zero quadratic forms with no common factors.
Then there exists a compact set $\Omega_{\infty} \subset \RR^{n-5}$ with the property that for any $P \gg_C 1$, there exists a set $\mathcal{C}(P) \subset  \ZZ^{n-5} \cap P\Omega_{\infty},$ for which the following hold.
\begin{enumerate}
\item For any $\y \in \mathcal{C}(P)$, there exists $i$ such that $|Q_i(\y)| \gg P^2$.
\item For any $\y \in \mathcal{C}(P)$, the equation~\eqref{eq:shape} is soluble over $\ZZ$ and $\gcd(Q_1(\y),\ldots,Q_5(\y)) \ll_C 1$. 
\item We have $\#\mathcal{C}(P) \gg P^{n-5}$.
\end{enumerate}

\end{lemma}

\begin{proof}
We begin by constructing the set $\Omega_{\infty}$. By hypothesis, none of the quadratic forms $Q_i(\y)$ vanish identically. In particular, we have that $\rank Q_1 = r \geq 1$. Then there exists an orthogonal matrix $U \in O_n(\RR)$ such that $Q_1(U\y) = \sum_{i=1}^r A_iy_i^2$ is diagonal and $A_i \neq 0$. Consequently, we can find disjoint subsets $I, J \subset \left\{1,2,\ldots,r\right\}$ such that $I \cup J = \left\{1,2,\ldots,r\right\}$ and
$$
\sum_{i=1}^r A_iy_i^2 = \sum_{i \in I}|A_i|y_i^2 - \sum_{j \in J}|A_i|y_i^2.
$$
For $i \in I$ and $j \in J$ let
$$
\mathcal{U}_i = [1/\sqrt{|A_i|},2/\sqrt{|A_i|}] \, \, \text{ and }\, \, \mathcal{V}_i = [1/4\sqrt{n|A_i|},1/2\sqrt{n|A_i|}].
$$

Define $\mathcal{B} \subset \RR^{n-5}$ to be the set consisting of tuples $(x_1,\ldots,x_{n-5})$ such that $x_i \in \mathcal{U}_i$ if $i \in I$, $x_j \in \mathcal{V}_j$ if $j \in J$ and $x_k \in [-1,1]$ if $k > r$. Then $Q_1(U\z) \gg 1$ for any $\z \in \mathcal{B}$. Set $\Omega_{\infty} = U\mathcal{B}$. 

Let $P \geq 1$. Observe that if $\y \in P\Omega_{\infty}$, then $\y = U\z$, for some $\z \in P\mathcal{B}$, whence $|Q_1(\y)| = |Q_1(U\z)| = |\sum_{i=1}^rA_iz_i^2| \gg P^2$. This completes the proof of the second statement of the lemma.

We will deduce the other statements by appealing to Proposition~\ref{thmsolubility}. Let $X \in \AA^n$ denote the hypersurface cut out by the equation $C(\x,\y) = 0$. Recall the map $\pi$ from Section~\ref{ekedahlapplication}. Then for any $P \geq 1$, it follows that $P\Omega_{\infty} \subset \pi(X(\RR)).$ To see this, note that $Q_1(\y) \neq 0$ whenever $\y \in P\Omega_{\infty}$. As a result, there exists $\x \in \RR^5$ such that $\sum_{i=1}^5 x_iQ_i(\y) + R(\y) = 0$. Thus the hypotheses of Theorem~\ref{thmsolubility} are all satisfied and the lemma follows.
\end{proof}

\subsection{Proof of Theorem~\ref{thmbc}}
Let $C(\x,\y)$ be as in~\eqref{eq:shape}. Suppose that $l(\y)$ is a linear form that divides each $Q_i(\y)$. Then after making a change 
of variables, we may suppose that 
\begin{equation}\label{eq:cxyreducible}
C(\x,\y) = y_k\sum_{i=1}^5\alpha_ix_iy_i + R(\y)
\end{equation}
where $k$ is an integer between $1$ and $n-5$ and $\alpha_i$ are integers. Suppose that $k \geq 5$. The case where $1 \leq k \leq 4$ is handled similarly, the details of 
which are left to the reader. 

Since $C(\x,\y)$ is non-degenerate, in the $x_i$, we see that $\alpha_i \neq 0$. We will rewrite this as follows,
\begin{equation*}
C(\x,\y) = y_k\sum_{i=1}^2\alpha_ix_iy_i + R_1(x_3,x_4,x_5,\y)
\end{equation*}
where $R_1(x_3,x_4,x_5,\y) = \alpha_3x_3y_3y_k+ \alpha_4x_4y_4y_k+ \alpha_5x_5y_5y_k + R(\y).$ 

Let $1 \leq Y \leq B$ be a parameter. Let $\mathcal{P}(\delta Y)$ denote the set of primes in the interval
$[\delta Y,2\delta Y]$, where $0 < \delta < 1/10$ is a small fixed real number. Let $g = (\alpha_1,\alpha_2)$ and set $\beta_1 = \alpha_1/g$ and $\beta_2 = \alpha_2/g$. 
We have the following lemma.

Replacing $x_3, x_4$ and $x_5$ by $gx_3, gx_4$ and $gx_5$
and $\y$ by $g\y$ in~\eqref{eq:cxyreducible}, we see that to solve the equation $C(x_1,x_2,gx_3,x_4,gx_5,g\y) = 0$ it suffices to solve   
$$
y_k\left(\sum_{i=1}^2 \beta_ix_iy_i\right) + R_1(x_3,x_4,x_5,\y) = 0,
$$
where $\beta_i = \alpha_i/g$. Note that $(\beta_1,\beta_2)=1$. Thus we have
\begin{align}\label{eq:nb907}
N(B) \geq \sum_{(x_3,x_4,x_5,\y) \in \mathcal{C}(Y)}\sum_{\substack{|x_1|,|x_2| \leq B/g \\ (x_1,x_2,g)=1 \\ \sum_{i=1}^2 \beta_ix_iy_i + R_1(x_3,x_4,x_5,\y)/y_k = 0}}1.
\end{align}
By construction, we have that $R_1(x_3,x_4,x_5,\y)/y_k \ll Y^2$. Let 
$$
\Lambda_{y_1,y_2} = \left\{\x \in \ZZ^2 : \sum_{i=1}^2 \beta_ix_iy_i = 0\right\}.
$$
Since $\alpha_i \neq 0$, we see that $\rank \Lambda_{y_1,y_2} = 1$ and $\det \Lambda_{y_1,y_2} \gg Y$. Consequently, if $Y \leq B^{1-\ve}$, for any $\ve > 0$, by  
Corollary~\ref{corlineq} we get that
\begin{align*}
&\#\left\{
|x_1|,|x_2| \leq B/g \text{ and }(\x,g) = 1: \sum_{i=1}^2 \beta_i x_i y_i = -R_1(x_3,x_4,x_5,\y)/y_k\right\} \\
&\quad\quad\quad\quad
\gg B/Y + O(1).
\end{align*}
Therefore, by~\eqref{eq:nb907} and by Lemma~\ref{lemmacyred} we have for any $Y \leq B^{1-\ve}$ that
$$
N(B) \gg \frac{BY^{n-4}}{\log Y}+O_{\ve}\left(Y^{n-3}/\log Y\right).
$$
Theorem~\ref{thmbc} follows from setting $Y = B^{1-\ve}$.

\subsection{Proof of Theorem~\ref{thmbb}} We begin with the following result.
\begin{lemma}\label{lemmaschemey}
Let $v \geq 4$. Let $\psi_1(\y),\ldots,\psi_v(\y)$ be quadratic forms in $m$ variables with integer coefficients. Suppose that $\Psi(\x,\y) = x_1\psi_1(\y)+\ldots+x_v\psi_v(\y)$ is irreducible over $\QQ$ and non-degenerate in the variables $x_i$. Let $Z_{\x} \subset \AA^{m}$ denote the variety cut out by the equation $\Psi(\x,\y) = 0$. Then the set 
$$
\left\{\x \in \AA^v : Z_\x \text{ is not integral}\right\} 
$$
is contained in a proper subvariety of $\AA^v$. In particular, there exists a minor of order $3$ in $M[\x]$ that does not vanish identically, 
where $M[\x]$ denotes the matrix of the quadratic form $\Psi$ over $\QQ(x_1,\ldots,x_v)$.
\end{lemma}
\begin{proof}
Let $Z \subset \AA^{m+v}$ denote the variety cut out by $\sum_{i=1}^vx_i\psi_i(\y) = 0$. Let $\pi: Z \to \AA^v$ denote the natural projection map $(\x,\y) \to \x$. It is clear that $\pi$ is a morphism of finite type with fibres $Z_{\x}$. Let 
$\eta = (0)$ denote the generic point of $\AA^v$. Then the generic fibre $Z_{\eta}$ is nothing but $\spec \frac{K[\y]}{(F(\x,\y))},$ where $K = \QQ(x_1,\ldots,x_v)$. Recall 
that a scheme is integral if and only if it is reduced and irreducible. We have by Proposition~\ref{lemmairrfnfd} that $Z_{\eta}$ is geometrically integral. As a result, the first assertion of the lemma follows from~\cite[\href{https://stacks.math.columbia.edu/tag/0559}{Tag 0559}]{stacks} and~\cite[\href{https://stacks.math.columbia.edu/tag/0578}{Tag 0578}]{stacks} applied to $\pi$, since $\AA^v$ is irreducible. The second statement of the lemma follows from Lemma~\ref{quadred}. 
\end{proof}

Let $M[\x]$ denote the matrix associated with the quadratic form $\sum_{i=1}^5 x_iQ_i(\y)$. Let $\left\{\phi_1(\x),\ldots,\phi_J(\x)\right\}$ denote the set of minors of $M[\x]$ that do not vanish identically. This set is non-empty by Lemma~\ref{lemmaschemey}. We will now discuss two separate cases: 
\begin{enumerate}
\item The minors $\phi_i(\x)$ have no common linear factor.
\item $l(\x) \mid \phi_i(\x)$ for each $1 \leq i \leq J$ and $l(\x)$ is a linear form. 
\end{enumerate}
\subsubsection{Case 1: The minors $\phi_i(\x)$ have no common linear factor}\label{case1}
Let $1 \leq Y \leq B$ be a parameter we will choose later. Since $F(\x,\y)$ is irreducible, it follows from Lemma~\ref{lemmacharacterisation} that $Q_i(\y)$ have no common factor. 

Let $\mathcal{C}(Y)$ denote the set constructed in Lemma~\ref{lemmalinwt} and let $\y \in \mathcal{C}(Y)$. Then 
$\gcd(Q_1(\y),\ldots,Q_5(\y)) \ll_C 1$. As a result, if $p \mid \y$, then $p^2 \mid \gcd(Q_1(\y),\ldots,Q_5(\y)) \ll_C 1$. This shows that for any
$\y \in \mathcal{C}(Y)$, we have that $\gcd(y_1,\ldots,y_{n-5}) \ll_C 1$. Therefore if we let $g(\y) = \gcd(y_1,\ldots,y_{n-5})$, we have
\begin{equation}\label{eq:nbnyb}
N(B) \geq \sum_{\y \in \mathcal{C}(Y)}N_{\y,g(\y)}(B),
\end{equation}
with $N_{\y,g(\y)}(B)$ as in~\eqref{eq:nylbdef} and $g(\y) \ll_C 1$.

Let
\begin{equation}\label{eq:latticelambday}
\Lambda_{\y} = \left\{\x \in \ZZ^5:\sum_{i=1}^5 x_iQ_i(\y) = 0\right\}.
\end{equation} 
Then $\Lambda_{\y}$ is a lattice in $\RR^5$. If $\y \in \mathcal{C}(Y)$, as $Q_i(\y) \gg Y^2$ for some $i$, we get that $\rank \Lambda_{\y} = 4$. Let $G(\y) = \sum_{i=1}^5 Q_i(\y)^2$. Then we also have
\begin{align*}
\det \Lambda_{\y} &= \frac{\sqrt{G(\y)}}{\gcd(Q_1(\y),\ldots,Q_5(\y))}\\
&\gg Y^2.
\end{align*}
As a result, if $\y \in \mathcal{C}(Y)$ then $|R(\y)|/\det \Lambda_\y \ll B^{1-\eta}$. Therefore, appealing to Corollary~\ref{corlineq} and
noting that $g(\y) \ll 1$, we get that
\begin{align*}
N_{\y,g(\y)}(B) &\gg \frac{B^4}{\det \Lambda_\y} + O\left(\sum_{j=0}^3\frac{B^j}{\lambda_1(\y)^j}\right),
\end{align*}
where $\lambda_1(\y)\geq 1$ is the first successive minimum of the lattice $\Lambda_\y$. 

As 
$\lambda_1(\y) \ll (\det \Lambda_\y)^{1/4} \ll Y^{1/2} \ll B^{1/2-\eta/2}$, we see that $B^j/\lambda_1(\y)^j \leq B^3/\lambda_1(\y)^3$ for 
each $0 \leq j \leq 2$. Therefore, we have 
\begin{align*}
N_{\y,g(\y)}(B) &\gg \frac{B^4}{\det \Lambda_\y} + O\left(\frac{B^3}{\lambda_1(\y)^3}\right).
\end{align*}
Summing over vectors $\y \in \mathcal{C}(Y)$, we get that
\begin{equation}\label{eq:mtprop2case1}
\begin{split}
N(B) &\geq \sum_{\y \in \mathcal{C}(Y)}N_{\y}(B) \\
&\gg B^4Y^{-2}\sum_{\y \in \mathcal{C}(Y)} 1 + O\left(\sum_{\substack{|\y|\leq Y \\ \rank \Lambda_\y = 4}}\frac{B^3}{\lambda_1(\y)^3}\right)  \\
&\gg B^4Y^{n-7} + O\left(\sum_{\substack{|\y|\leq Y \\ \rank \Lambda_\y = 4}}\frac{B^3}{\lambda_1(\y)^3}\right) 
\end{split}
\end{equation}
by Lemma~\ref{lemmalinwt}. Let
\begin{equation}\label{eq:etprop2case1}
ET = \sum_{\substack{|\y|\leq Y \\ \rank \Lambda_\y = 4}}\frac{B^3}{\lambda_1(\y)^3}.
\end{equation}
To estimate $ET$, we will require the following lemma, which is an effective version of Hilbert's irreducibility theorem.
\begin{lemma}\label{lemmairreducibility}
Suppose that $F(\x,\y)$ satisfies the hypotheses in the statement of Proposition~\ref{thmbb}. Let $\left\{\phi_1(\x),\ldots,\phi_J(\x)\right\}$ be the set of minors of order $3$ of $M[\x]$ that do not vanish identically. Suppose that $\phi_i(\x)$ have no common linear factors. Then for $\ve > 0$ we have
$$
\#\left\{|\x| \leq R: \sum_{i=1}^5 x_iQ_i(\y) \text{ is reducible over } \QQ\right\} \ll_{\ve} R^{3+\ve}.
$$
\end{lemma}
\begin{proof}
We have by Lemma~\ref{lemmaschemey} that the set of $\x \in \AA^5$ for which $\sum_{i=1}^5x_iQ_i(\y)$ is reducible is contained
in a proper subvariety in $\AA^5$. For any such $\x$, Lemma~\ref{quadred} ensures that $\rank M[\x] \leq 2$. Therefore, if $\x \in \AA^5$ 
such that $\sum_{i=1}^5 x_iQ_i(\y)$ is reducible, then $\phi(\x) = 0$, for each minor $\phi$ of order $3$ in $M[\x]$. 

Suppose first that there exists a minor $\Delta(\x)$ of order $3$ that is irreducible over $\QQ$. Then an upper bound 
for the number of $|\x| \leq R$ such that $\sum_{i=1}^5x_iQ_i(\y)$ is reducible is given by 
$
\#\left\{|\x| \leq R: \Delta(\x) = 0 \right\},
$
which is $O_{\ve}(R^{3+\ve})$, for any $\ve > 0$, by~\cite[Theorem 3]{BHB05}.

Suppose next that we are in the case where all the minors $\phi_i$ have a linear factor.
Then by our assumption that $\phi_i$ have no common factor, there exist at least two minors of order $3$, 
$\Delta_1(\x) = l_1(\x)q_1(\x)$ and $\Delta_2(\x) = l_2(\x)q_2(\x)$, say, with $l_i(\x)$ non-constant. In this case, the set of $\x \in \AA^5$ such that $\sum_{i=1}^5x_iQ_i(\y)$ is reducible lies in
$$
\mathcal{X} = \left\{\x \in \AA^5: \Delta_1(\x) = \Delta_2(\x) = 0 \right\}.
$$
Using the fact that $l_1(\x) \nmid \Delta_2(\x)$ and $l_2(\x) \nmid \Delta_1(\x)$, we get that $\mathcal{X}$ lies in a codimension $2$ subvariety in $\AA^5$. This completes the proof of the lemma.
\end{proof}
As $\rank \Lambda_\y = 4$ for each $\y$ in~\eqref{eq:etprop2case1}, we have that 
$$
\lambda_1(\y) \ll (\det \Lambda_\y)^{1/4} \ll Y^{1/2}.
$$ Therefore we may write
\begin{align*}
ET &= B^3\sum_{\substack{1 \leq R \ll Y^{1/2} \\ R \text{ dyadic}}}R^{-3}\sum_{\substack{|\y|\leq Y \\ \text{such that } \\ \rank \Lambda_\y = 4 \\ \text{ and} \\ R \leq \lambda_1(\y) \leq 2R}}1 \\
&\ll B^3\sum_{\substack{1 \leq R \ll Y^{1/2} \\ R \text{ dyadic}}}R^{-3}\sum_{|\y| \leq Y}\sum_{\substack{|\x| \leq 2R \\ \x \in \Lambda_\y}}1 \\
&\ll B^3\sum_{\substack{1 \leq R \ll Y^{1/2} \\ R \text{ dyadic}}}R^{-3}\sum_{|\x| \leq 2R}\sum_{\substack{|\y| \leq Y \\ \sum_{i=1}^5 x_iQ_i(\y) = 0}} 1.
\end{align*}
By Lemma~\ref{lemmairreducibility} and~\cite[Theorem 2]{HB02} we get
\begin{align*}
ET &\ll_{\ve} B^3\sum_{\substack{1 \leq R \ll Y^{1/2} \\ R \text{ dyadic}}}R^{-3}\left\{R^5Y^{n-7+\ve} + R^3Y^{n-6+\ve}\right\} \\
&\ll_{\ve} B^3Y^{n-6+\ve},
\end{align*}
for any fixed $\ve > 0$. By setting $Y = B^{1-2\ve}$ the first case of Proposition~\ref{thmbb} follows from~\eqref{eq:mtprop2case1}.

\subsubsection{Case 2: The minors $\phi_i(\x)$ are divisible by a common linear factor}
Let $l(\x)$ be a linear form in $\x$. Let $Z_l \subset \AA^n$ be the variety defined by $ F(\x,\y) = l(\x) = 0$. Make the change of variables $l(\x) = z_5$. Then we get
$$
F(\x,\y) = m_1(\mathbf{z})Q_1(\y) + \ldots + m_5(\mathbf{z})Q_5(\y)
$$
for linear forms $m_1(\mathbf{z}),\ldots,m_5(\mathbf{z})$. Let $m_i(\z) = \sum_{j=1}^5\tau_{i,j}z_j$ and set
$$
\tau_i(z_1,\ldots,z_4) = m_i(z_1,\ldots,z_4,0) = \sum_{j=1}^4 \tau_{i,j}z_j.
$$ 
Let
\begin{align*}
F_l(\z,\y) = \sum_{i=1}^5\tau_i(\z)Q_i(\y) = \sum_{j=1}^4z_j\sum_{i=1}^5 \tau_{i,j}Q_i(\y) = \sum_{j=1}^4z_jG_j(\y).
\end{align*}
Then we find that 
\begin{align*}
Z_l \subset \AA^4 \times \AA^{n-5} &: F_l(\z,\y) = 0.
\end{align*}

Suppose first that $Z_l$ is irreducible. Since $Z_l$ is non-degenerate in $z_i$, we get from Lemma~\ref{lemmaschemey} that
get that the set 
$$
\left\{\z \in \AA^4: F_l(\z,\y) \text{ is reducible over }\QQ\right\}
$$
is contained in a proper subvariety of $\AA^4$. Taking $l(\x)$ to be the linear factor that divides every minor of order $3$ in $M[\x]$, this in turn implies that the set
$$
\left\{\x \in \AA^5: F(\x,\y) \text{ is reducible over }\QQ\right\}
$$
is contained in a codimension $2$ subvariety in $\AA^5$. As a result, we get that
$$
\#\left\{|\x| \leq R: \sum_{i=1}^5 x_iQ_i(\y) \text{ is reducible over } \QQ\right\} \ll R^{3}.
$$
Using the above estimate in place of Lemma~\ref{lemmairreducibility}, we may argue as in Section~\ref{case1} to conclude that $N(B) \gg_{\ve} B^{n-3-\ve}$ if $Z_l$ is irreducible.

Turning to the case where $Z_l$ is reducible, we find that 
\begin{align*}
F_l(\z,\y) = t(\y)\sum_{j=1}^4z_jg_j(\y),
\end{align*}
for linear forms $t(\y)$ and $g_j(\y)$. Moreover, we have
\begin{align*}
F(\x,\y) &= \sum_{i=1}^5\tau_i(\z)Q_i(\y) + z_5\sum_{i=1}^5\tau_{i,5}Q_i(\y) \\
&= F_l(\z,\y) + z_5\sum_{i=1}^5\tau_{i,5}Q_i(\y) \\
&= t(\y)\sum_{j=1}^4g_j(\y)z_j + z_5\sum_{i=1}^5\tau_{i,5}Q_i(\y)
\end{align*}
Therefore, by making a change of variables, we may suppose that 
$$
F(\x,\y) = y_1\sum_{i=1}^4 z_il_i(\y) + z_5Q(\y),
$$
for certain linear forms $l_i(\y)$ that have no common factor, and a non-zero quadratic form $Q(\y)$.

We proceed once again as in Section~\ref{case1}. Let $\eta > 0$ and $1 \leq Y \leq B^{1-\eta}$. Let $\mathcal{C}(Y)$ be as in 
Lemma~\ref{lemmalinwt} and let $\Lambda_\y$ be as in~\eqref{eq:latticelambday}. Then we get from~\eqref{eq:nbnyb} 
and~\eqref{eq:mtprop2case1} that
\begin{align}\label{eq:nbprop2case2}
N(B) \gg \sum_{\y \in \mathcal{C}(Y)}\frac{B^4}{\det \Lambda_\y} + 
O\left(\sum_{\substack{|\y| \leq Y \\ \rank \Lambda_{\y} = 4}}\frac{B^3}{\lambda_1(\y)^3}\right)
\end{align}
where $\lambda_1(\y)$ is the length of the shortest non-zero vector in $\Lambda_\y$. Note that $\lambda_1(\y) \ll (\det \Lambda_\y)^{1/4} \ll Y^{1/2}$.
Arguing as before, we get from Lemma~\ref{lemmalinwt} that 
\begin{align}\label{eq:mtprop2case2}
\sum_{\y \in \mathcal{C}(Y)}\frac{B^4}{\det \Lambda_\y} \gg B^4Y^{n-7}.
\end{align}

Moving on to estimating the error term in~\eqref{eq:nbprop2case2}, define the lattice
$$
\widetilde{\Lambda}_\y = \left\{\x \in \ZZ^4 : \sum_{i=1}^4x_il_i(\y) = 0\right\}.
$$
If $l_i(\y) \neq 0$, for some $i$, we see that $\widetilde{\Lambda}_\y$ is a sublattice of $\Lambda_\y$ of rank $3$. Observe that if $\x = (x_1,\ldots,x_5) \in \Lambda_\y$, then $(x_1,\ldots,x_4) \in \widetilde{\Lambda}_\y$ as long as $y_1 \neq 0$.

Suppose that $R \leq \lambda_1(\y) \leq 2R$ for some $R \geq 1$. Then there exists a vector $\x$ in the annulus $R \leq |\x|_2 \leq 2R$ such that $\x$ has the least length in $\Lambda_\y$. Then if $y_1 \neq 0$, we find that $\z = (x_1,\ldots,x_4)\in \widetilde{\Lambda}_\y$, which has shorter norm. So the shortest non-zero vector in $\Lambda_\y$ must, in fact, lie in $\widetilde{\Lambda}_\y$. Note that the length of the shortest integer vector in $\widetilde{\Lambda}_\y$ is bounded by $Y^{1/r},$ where $r = \rank \widetilde{\Lambda}_\y$. As a result, if we let
$$
ET = B^3\sum_{\substack{|\y| \leq Y \\ \rank \Lambda_{\y} = 4}}\frac{1}{\lambda_1(\y)^3},
$$
we must have
\begin{align*}
ET &\ll B^3\sum_{\substack{1 \leq R \ll Y^{1/3} \\ Y \text{ dyadic}}}\frac{1}{R^3} \sum_{\substack{|\y| \leq Y \\ y_1 \neq 0 \\ l_i(\y) \neq 0 \\ \text{for some } i}}\sum_{\substack{\x \in \widetilde{\Lambda}_\y \cap \ZZ^4_{\prim} \\R \ll |\x| \ll R}}1  \\
&\quad + B^3\sum_{\substack{|\y| \leq Y \\ \rank \Lambda_\y = 4 \\\text{either }y_1 = 0, \\ \text{ or } l_i(\y) = 0\\ \text{for each i}}}\frac{1}{\lambda_1(\y)^3} \\
&=  B^3\sum_{\substack{1 \leq R \ll Y^{1/3} \\ Y \text{ dyadic}}}\frac{1}{R^3} \sum_{\substack{\x \in \ZZ^4_{\prim} \\ R \ll |\x| \ll R}}\sum_{\substack{|\y| \leq Y \\ \sum_{i=1}^4x_il_i(\y) = 0}} 1 \,+ B^3Y^{n-6},
\end{align*}
where we have used the fact that if either $y_1 = 0$, or if $l_i(\y) = 0$ for each $i$, then $\rank \Lambda_\y \neq 4$ unless $Q(\y) \neq 0$, which forces $x_5 =0$, whence $\lambda_1(\y) = 1$. 

Since $\x \in \ZZ_{\prim}^4$ and since $l_i(\y)$ have no common linear factors, we have
$$
\#\left\{|\y| \leq Y : \sum_{i=1}^4 x_il_i(\y) = 0\right\} \ll \frac{Y^{n-6}}{|\x|_2} + Y^{n-7}.
$$
Consequently, we get
\begin{align*}
ET &\ll B^3\sum_{\substack{1 \leq R \ll Y^{1/3} \\ Y \text{ dyadic}}}\frac{1}{R^3} \sum_{\substack{\x \in \ZZ^4_{\prim} \\ R \ll |\x| \ll R}}\left(\frac{Y^{n-6}}{R} + B^3Y^{n-7}\right) + B^3Y^{n-6} \\
&\ll B^3Y^{n-6}\sum_{\substack{1 \leq R \ll Y^{1/3} \\ Y \text{ dyadic}}}\left(1+ R/Y\right) \ll_{\ve} B^3Y^{n-6+\ve},
\end{align*}
for any $\ve > 0$. Thus by setting $Y = B^{1-2\ve}$, we get from~\eqref{eq:mtprop2case2} that
$$
N(B) \gg_{\ve} B^{n-3-\ve}.
$$
This completes the proof of Theorem~\ref{thmbb}.
